\documentclass[letterpaper,reqno,12pt]{amsart}
\usepackage[margin=1.2in]{geometry}

\usepackage{amssymb, enumerate, color}
\usepackage{mathrsfs}
\usepackage[all]{xy}

\usepackage[colorlinks, citecolor=blue]{hyperref}

\usepackage{enumitem}
\usepackage{graphicx}

\usepackage{tikz}
\numberwithin{equation}{section} 

\theoremstyle{plain}
\newtheorem{thm}{Theorem}[section] 
\newtheorem{cor}[thm]{Corollary}
\newtheorem{prop}[thm]{Proposition}
\newtheorem{conj}[thm]{Conjecture}
\newtheorem{lem}[thm]{Lemma}
\newtheorem{thmA}{Theorem}[section]

\theoremstyle{definition} 
\newtheorem{defn}[thm]{Definition}
\newtheorem{setting}[thm]{Setting}
\newtheorem{eg}[thm]{Example}

\theoremstyle{remark}
\newtheorem{rem}[thm]{Remark}
\newtheorem{ques}[thm]{Question}
\newtheorem*{acknowledgement}{Acknowledgments}
\newtheorem*{claim}{Claim}

\def\ge{\geqslant}
\def\le{\leqslant}
\def\phi{\varphi}
\def\epsilon{\varepsilon}
\def\tilde{\widetilde}

\def\to{\longrightarrow}
\def\mapsto{\longmapsto}

\def\onto{\relbar\joinrel\twoheadrightarrow}

\newcommand{\sO}{\mathcal{O}}

\newcommand{\N}{\mathbb{N}}
\newcommand{\Q}{\mathbb{Q}} 
 
\newcommand{\R}{\mathbb{R}} 
\newcommand{\Z}{\mathbb{Z}}
\newcommand{\PP}{\mathbb{P}}

\newcommand{\ba}{\mathfrak{a}}
\newcommand{\bb}{\mathfrak{b}}
\newcommand{\m}{\mathfrak{m}}

\newcommand{\WDiv}{\mathrm{WDiv}}
\newcommand{\CDiv}{\mathrm{Div}}
\newcommand{\vol}{\mathrm{vol}}
\newcommand{\ses}[2][x]{\epsilon(#2; #1)}
\newcommand{\ejet}[2][x]{\epsilon_{\mathrm{jet}}(||#2||; #1)}
\newcommand{\B}{\mathbf{B}}
\newcommand{\Bs}{\mathrm{Bs}}
\newcommand{\fpt}{\mathrm{fpt}}

\newcommand{\mld}{\mathrm{mld}}

\newcommand{\RInt}{\mathrm{Relint}}
\newcommand{\NE}{\overline{\mathrm{NE}}}

\newsavebox{\circlebox}
\savebox{\circlebox}{\fontencoding{OMS}\selectfont\Large\char13}
\newlength{\circleboxwdht}

\def\Hom{\operatorname{Hom}}
\def\Spec{\operatorname{Spec}}

\def\Supp{\operatorname{Supp}}
\def\Pic{\operatorname{Pic}}
\def\Div{\operatorname{div}}
\def\Exc{\operatorname{Exc}}
\def\mult{\operatorname{mult}}
\def\ord{\operatorname{ord}}

\def\id{\operatorname{id}}

\def\cod{\operatorname{cod}}
\renewcommand{\Im}{\mathrm{Im}}
\newcommand{\Ker}{\mathrm{Ker}}
\newcommand{\Cok}{\mathrm{Cok}}
\title{Boundedness of weak Fano threefolds with fixed Gorenstein index in positive characteristic}

\author{Kenta Sato}
\address{Faculty of Mathematics, Kyushu University, 744 Motooka, Nishi-ku, Fukuoka 819-0395, Japan}
\email{ksato@math.kyushu-u.ac.jp}

\thanks{}


\dedicatory{}

\begin{document}

\begin{abstract}
In this paper, we give a partial affirmative answer to the BAB conjecture for $3$-folds in characteristic $p>5$.
Specifically, we prove that a set $\mathcal{D}$ of weak Fano $3$-folds over an uncountable algebraically closed field is bounded, if each element $X \in \mathcal{D}$ satisfies certain conditions regarding the Gorenstein index, a complement and Kodaira type vanishing.
In the course of the proof, we also study a uniform lower bound for Seshadri constants of nef and big invertible sheaves on projective $3$-folds.

\end{abstract}

\maketitle
\markboth{K.~SATO}{BOUNDEDNESS OF WEAK FANO THREEFOLDS WITH FIXED GORENSTEIN INDEX}

\setcounter{tocdepth}{1}
\tableofcontents

\section{Introduction}

A normal projective variety over an algebraically closed field $k$ is \emph{weak Fano} if $X$ has only klt singularities and the anti-canonical divisor $-K_X$ is nef and big.
Weak Fano varieties enjoy several nice properties and are in the center of interest in algebraic geometry.
One such noteworthy attribute is the finiteness of them.
In dimension one, there is only one weak Fano variety up to isomorphism, namely $\PP^1_k$.
In dimension two, Alexeev (\cite{Ale}) proved that if we fix a real number $\epsilon>0$, then the set of weak Fano surfaces with the minimal log discrepancy at least $\epsilon$ forms a bounded family.
This result is generalized and encapsulated in what is known as the BAB conjecture:

\begin{conj}[BAB conjecture]
Fix an algebraically closed field $k$, an integer $d>0$ and a rational number $\epsilon > 0$.
Then $d$-dimensional normal projective varieties $X$ over $k$ such that 
\begin{itemize}
  \item $-(K_X+B)$ is nef and big for some $\Q$-divisor $B$ on $X$, and 
  \item the minimal log discrepancy $\mld(X,B)$ of $(X,B)$ is larger than $\epsilon$
\end{itemize}
forms a bounded family.
\end{conj}

In characteristic $0$, this conjecture has been thoroughly investigated by numerous authors (\cite{KMM}, \cite{HMX2}, and so on), and was ultimately solved in the celebrated paper by Birkar (\cite{BirBdd}).

In positive characteristic, the classification by Tanaka and Asai-Tanaka (\cite{Tan1}, \cite{Tan2}, \cite{AT3}, \cite{Tan4}) implies that smooth Fano threefolds form a bounded family.
For the singular case, Zhuang (\cite{Zhu}) proved the birational boundedness of Fano threefolds with sufficiently large Seshadri constants of $-K_X$.
However, the BAB conjecture for singular $3$-folds in positive characteristic still far from being fully resolved.
Our main result provides a partial solution to this challenge.

\begin{thmA}[Theorem \ref{bdd of weak Fano}]\label{A}
  Fix a DCC subset $I \subseteq [0,1] \cap \Q$, a rational number $\epsilon>0$, integers $r,A>0$ and an uncountable algebraically closed field $k$ of characteristic larger than $5$.
  Suppose that $\mathcal{D} =\{(X_j, \Delta_j)\}_{j \in J}$ is a set of $3$-dimensional projective log pairs over $k$ such that for every $j \in J$, the following conditions are satisfied:
  \begin{enumerate}[label=\textup{(\roman*)}]
    \item $-rK_{X_j}$ is a nef and big Cartier divisor.
    \item $K_{X_j}+\Delta_j \sim_{\Q} 0$, $\epsilon< \mld(X_j,\Delta_j)$ and every coefficient of $\Delta_j$ is contained in $I$.
    \item $H^1(X_j, \sO_{X_j}(irK_{X_j}))=0$ for every integer $i \ge 0$.
    \item $H^2(X_j, \sO_{X_j})=0$ and $h^2(X_j, \sO_{X_j}(irK_{X_j})) < A $ for $i=1,2,3$.
  \end{enumerate}
  Then $\mathcal{D}$ is bounded.
\end{thmA}

\begin{rem}
  Given the failure of Kodaira type vanishing in positive characteristic, it becomes necessary to impose conditions (iii) and (iv) in Theorem \ref{A}. Nonetheless, there are instances where we can relax these assumptions as follows:
  \begin{enumerate}
    \item If $X_j$ is Fano, then the assumption (iii) follows from \cite[Theorem 1.9]{PW} except for the case $i=0$.
    \item If $X_j$ is \emph{globally $F$-regular}, then the assumptions (iii) and (iv) follow from \cite[Corollary 4.4]{Smi}.
    Here, globally $F$-regular varieties are a special type of varieties in positive characteristic defined in terms of Frobenius morphisms (see \cite[Section 3]{Smi} for more details).
  \end{enumerate}
\end{rem}

Our hypothesis concerning the Gorenstein index (i) and a complement (ii) is a standard assumption frequently encountered in various statements related to boundedness. 
For instance, refer to \cite[Corollary 1.7 and Corollary 1.8]{HMX2} for a partial result of the BAB conjecture in characteristic $0$, which involves complements and the Gorenstein index.
In this sense, our result is an important first step toward the solution of the BAB conjecture in positive characteristic.

Given that the anti canonical volume $(-K_{X_j})^3$ is uniformly bounded according to \cite{Das}, Theorem \ref{A} above would be readily apparent if we could establish the existence of a uniform number $M>0$ such that $-MK_{X_j}$ is very ample for every $j$.
Unfortunately, this approach is not viable due to the absence of a positive characteristic analogue of Matsusaka's big theorem (\cite{MatBig}, \cite{KMBig}). 
Consequently, we commence by seeking a uniform number $M>0$ such that $-MK_{X_j}$ defines a birational map.

\begin{thmA}[Corollary \ref{bb dim3}]\label{B}
  Let $X$ be a $3$-dimensional normal projective variety with rational singularities over an uncountable algebraically closed field of positive characteristic, $\ell \ge 1$ be an integer and $D$ be a nef and big Cartier divisor on $X$.
  We further assume that the following conditions are satisfied:
  \begin{enumerate}[label=\textup{(\roman*)}]
  \item $H^1(X, \sO_X)=0$.
  \item $h^0(X,\sO_X({\ell}D)) > \ell \vol(D)+1$.
  \end{enumerate}
Then $\Phi_{|K_X+(6\ell+1)D|}$ is birational onto its image.
\end{thmA}

In the situation of Theorem \ref{A}, combining the assumptions on the cohomologies with Theorem \ref{B}, we find a uniform integer $M>0$ such that $-MK_{X_j}$ is birational for every $j$ (Corollary \ref{bb weak Fano}).
See also \cite{Zha} for related results.

The cornerstone of the proof of Theorem \ref{B} is a uniform lower bound of \emph{Seshadri constants}.
Let $X$ be a normal projective variety and $L$ be a nef and big invertible sheaf on $X$.
Then the Seshadri constant $\ses{L}$ of $L$ at a closed point $x$ measures the local positivity of $L$ at $x$ (see Definition \ref{defn ses}).
In characteristic $0$, Ein-K\"{u}chle-Lazarsfeld proved in their celebrated work \cite{EKL} that we have
\[
  \ses{L} \ge \frac{1}{\dim X}
\]
if $x \in X$ is very general.
It is natural to ask whether such an inequality holds in positive characteristic:

\begin{ques}\label{question on lower bound}
Let $k$ be an uncountable algebraically closed field of positive characteristic and $d>0$ be an integer.
Is there a real number $\delta(k,d)>0$ such that for any normal projective variety $X$ over $k$ with dimension $d$ and for any nef and big line bundle $L$ on $X$, we have
\[
  \ses{L} \ge \delta(k,d)
\]
at a very general point $x \in X$?
\end{ques}

It is worth mentioning that in positive characteristic, this question remains unresolved even in dimension $2$. 
In Section \ref{sec3}, we provide a partial answer to Question \ref{question on lower bound} for dimension $2$ (see Proposition \ref{ses dim2}) and subsequently, extend our analysis to the $3$-dimensional case as follows.

\begin{thmA}[Theorem \ref{Lower bound for seshadri}]\label{C}
    Let $X$ be a $3$-dimensional normal projective variety with rational singularities over an uncountable algebraically closed field, $\ell \ge 1$ be an integer and $L$ be a nef and big invertible sheaf on $X$.
    We further assume that the following conditions are satisfied:
    \begin{enumerate}[label=\textup{(\roman*)}]
    \item $H^1(X, \sO_X)=0$.
    \item $h^0(X,L^{\ell}) > \ell \vol(L)+1$.
    \end{enumerate}
    Then we have
    \[
      \ses{L} \ge \frac{1}{\ell}
      \]
    for a very general closed point $x \in X$.
\end{thmA}

We note that under specific assumptions on higher cohomologies and volumes, we can find a uniform integer $\ell$ such that the assumption (ii) is satisfied (Corollary \ref{lower bound for seshadri vanishing} or Corollary \ref{lower bound for seshadri weak Fano}).
By combining Theorem \ref{C} with a theory of Seshadri constants developed by Musta\c{t}\u{a}-Schwede (\cite{MS}) and Murayama (\cite{Mur}), we prove Theorem \ref{B} in Section \ref{sec4}.

Now, let us return to the outline of Theorem \ref{A}. 
It follows from Theorem \ref{B} that the set $\{X_j \}_{j \in J}$, obtained from $\mathcal{D}$ by forgetting the complements $\Delta_j$, is birationally bounded.
To progress further, it is imperative to establish the boundedness of the complements $\Delta_j$ as well. 
To this end, we provide an upper bound for intersection numbers in terms of volumes (Proposition \ref{HMX lemma}), and as a consequence, we establish the following result:

\begin{thmA}[Theorem \ref{criterion of lbb}]\label{D}
  Let $V>0$ be an integer and $k$ be an algebraically closed field of positive characteristic. 
  Suppose that $\mathcal{E}$ is a set of $n$-dimensional projective log pairs over $k$ such that for every $(X,\Gamma) \in \mathcal{E}$, the following conditions are satisfied:
  \begin{enumerate}[label=\textup{(\roman*)}]
  \item $\Gamma$ is reduced, 
  \item there is a log resolution of $(X, \Gamma)$, and 
  \item there exists a Weil divisor $E$ on $X$ such that $\Phi_{|E|}$ is birational onto its image, $\vol(E)<V$ and $\vol(K_X+\Gamma+2(2n+1) E)<V$.
  \end{enumerate}
  Then $\mathcal{E}$ is log birationally bounded.
\end{thmA}

The basic idea of the proof is similar to that of \cite[Theorem 3.1]{HMX13}.
However, their approach depends on Kodaira-type vanishing, which is not applicable in positive characteristic.
To circumvent this limitation, we once again focus on Seshadri constants.

In the final step, we demonstrate how to reduce the boundedness of weak Fano varieties to the log birationally boundedness of them.
More specifically, we prove the following result:

\begin{thmA}[Theorem \ref{Bir bdd to bdd}]\label{E}
    Let $0<\delta, \epsilon$ be rational numbers, $\pi : \mathcal{Z} \to T$ be a smooth projective morphism between smooth varieties over an algebraically closed field $k$ of characteristic $p>5$ and $B$ be a reduced divisor on $\mathcal{Z}$ such that $(\mathcal{Z}, B)$ is relatively SNC over $T$.
    
    For an open dense subset $U \subseteq T$, we denote by $\mathcal{D}_U$ the set of all $3$-dimensional projective log pairs $(X,\Delta)$ with the following properties:
    \begin{enumerate}[label=\textup{(\roman*)}]
      \item There exist a closed point $t \in U(k)$ and a birational contractions 
      \[
        f: \mathcal{Z}_{t} \dashrightarrow X
      \] 
      such that $B_t$ coincides with the union of $\Supp(f^{-1}_*\Delta)$ and all $f$-exceptional divisors.
      \item $-K_{X}$ is nef and big.
      \item $K_{X}+\Delta \sim_{\Q} 0$, $\epsilon < \mld(X,\Delta)$ and every coefficient of $\Delta$ is larger than $\delta$
      \item $H^1(X, \sO_{X})=H^2(X, \sO_{X})=0$.
    \end{enumerate}
    We further assume that we have 
    \[
      \kappa(\mathcal{Z}/T, K_{\mathcal{Z}}+(1-\epsilon)B) \ge 1.
    \]
    Then there exists an open dense subset $U \subseteq T$ such that $\mathcal{D}_U$ is bounded.
  \end{thmA}

Here, we outline the proof of Theorem \ref{E}.
Similar to the conventional approach in characteristic $0$ (cf.~\cite{HMX2}, \cite{HX}), our strategy is to reduce the theorem to the finiteness of minimal models and lc models on the generic fiber $\mathcal{Z}_{\eta}$.
Given the development of the theory of the minimal model program over an imperfect field presented in \cite{DW} and \cite{Wal}, such finiteness holds true for $\Q$-divisors on $\mathcal{Z}_{\eta}$ containing a fixed big $\Q$-divisor.

The challenge here is to give a sufficient condition for a given $\Q$-divisor on $\mathcal{Z}_{\eta}$ to be big.
In characteristic $0$, the "invariance of plurigenera" is a conventional tool for this purpose, which fails in positive characteristic (\cite{Bri}).
To circumvent this obstacle, we rely on the fact that the abundance conjecture is valid when the Iitaka dimension is positive (\cite{BirMMP}, \cite{DWAb}, \cite{WalAb}). 
This is the reason why we need the assumption on the Iitaka dimension in Theorem \ref{E}.

As a consequence of Theorems \ref{B}, \ref{C}, \ref{D} and \ref{E}, we obtain Theorem \ref{A}.

\begin{small}
\begin{acknowledgement}
The author wishes to express his gratitude to Shunsuke Takagi, Shou Yoshikawa, Tatsuro Kawakami, Teppei Takamatsu and Hiromu Tanaka for their helpful comments and suggestions.
This work was supported by JSPS KAKENHI Grant Number 20K14303.  
\end{acknowledgement}
\end{small}

\section{Preliminaries}

\subsection{Notation and conventions}\label{Notation}

In this subsection, we summarize notation used in this paper.
\begin{enumerate}
\item Throughout this paper, all rings are assumed to be commutative and with a unit element, all schemes are assumed to be quasi-projective over a ground field $k$, and all fields of characteristic $p>0$ are assumed to be \emph{$F$-finite}, that is, the extension degree $[k: k^p]$ is finite.

\item A \emph{variety} over a field $k$ is an irreducible and reduced scheme of finite type over $k$.
A normal variety $X$ is said to be \emph{$\Q$-factorial} if every Weil divisor is $\Q$-Cartier.
A \emph{curve} in a variety $X$ is a closed subvariety of dimension one.
A \emph{pointed curve} in $X$ is a pair $(C,x)$ of a curve $C$ in $X$ and a closed point $x \in X$ contained in $C$.

\item A \emph{log pair} (resp.~\emph{projective log pair}) is a pair $(X, \Delta)$ of a normal (resp.~normal projective) variety $X$ over a field $k$ and an effective $\Q$-Weil divisor $\Delta$ on $X$. 
A log pair $(X, \Delta)$ is said to be \emph{SNC} if $X$ is regular and every stratum of $\Supp(\Delta)$ is regular.
We further assume that there exists an flat morphism $\pi : X \to T$ to a regular variety $T$.
An SNC pair $(X,\Delta)$ is said to be \emph{relatively SNC over $T$} if for every closed point $t \in T$, the log pair $(X_t, \Delta_t)$ induced on the fiber $X_t$ is SNC.

\item 
Let $X$ be a variety over a field $k$, $K(X)$ be the function field of $X$ and $\mathcal{K}_X$ be the constant sheaf on $X$ determined by $K(X)$.
We denote by $\CDiv(X)=H^0(X, \mathcal{K}_X^*/\sO_X^*)$ the set of all Cartier divisors and by $\WDiv(X)$ the set of all Weil divisors.

\item
For a Cartier divisor $D$ on a variety $X$, we write $[D] \in \WDiv(X)$ the Weil divisor defined by the Cartier divisor $D$ (cf.~\cite[Section 2.1]{Ful}) and $\sO_X(D)$ the invertible sheaf associated to $D$.
When $X$ is normal and $\Gamma$ is a Weil divisor on $X$.
We also denote by $\sO_X(\Gamma)$ the coherent sheaf defined by 
\[
\sO_X(\Gamma)(U) = \{ 0 \neq f \in K(X) \mid D+ \Div_X(f) \ge 0\} \cup \{0\},
\]

\item
We say that a Cartier divisor $D$ on a variety $X$ is \emph{effective} if for every point $x \in X$, the defining equation $f_x \in K(X)$ of $D$ at $x \in X$ is contained in $\sO_{X,x}$.
We note that if $D$ is effective, then so is the Weil divisor $[D]$, and the converse holds true if $X$ is normal.
We sometimes identify an effective Cartier divisor $D$ on $X$ with the closed subscheme of $X$ defined by the ideal sheaf $\sO_X(-D)$.

\item Let $D$ be an effective Cartier divisor on a variety $X$ over a field $k$.
We denote by $\Supp(D)$ the set of all points $x \in X$ such that the equation $f_x \in \sO_{X,x}$ of $D$ at $x$ is contained in the maximal ideal $\m_x$.
Noting that $\Supp(D)$ is a purely codimension one closed subset of $X$, it coincides with the support of the Weil divisor $[D]$.

\item
For a $\Q$-Weil divisor $\Gamma=\sum_{i=1}^n a_i \Gamma_i$ on a variety $X$ where $a_i \neq 0$ is a rational number and $\Gamma_i$ is a prime divisor, the \emph{support} $\Supp(\Gamma)$ of $\Gamma$ is the union $\cup_{i=1}^n \Gamma_i$.
The round-down of $\Gamma$ is $\lfloor \Gamma \rfloor=\sum_i \lfloor a_i \rfloor \Gamma_i$, where $\lfloor a_i \rfloor$ denotes the largest integer smaller than or equal to $a_i$. 

\item
Let $Z \subseteq X$ be a closed subvariety of a variety $X$ and $x \in Z$ be a point.
The \emph{multiplicity} of $Z$ at $x$ is the Hilbert-Samuel multiplicity of the local ring $\sO_{Z,x}$ and denote it by $\mult_x(Z)$.
We make the convention that $\mult_x(Z) : = 0$ if $x \not\in Z$.
For a Weil divisor $\Gamma=\sum_{i=1}^n a_i E_i$ on $X$ where $E_i$ is a prime divisor, we define 
\[
\mult_{x} (\Gamma) : =\sum_{i=1}^n a_i \mult_x(E_i).
\]

\item
Let $X$ be a proper variety over a field $k$ and $\mathcal{F}$ be a coherent sheaf on $X$.
For every integer $i \ge 0$, we denote by $h^i(X, \mathcal{F})$ the dimension $\dim_k(H^i(X, \mathcal{F}))$ of the cohomology $H^i(X, \mathcal{F})$.

\item
For an invertible sheaf $L$ on a proper variety $X$, the \emph{complete linear system} of $L$ is
\[
|L| : = \{D \in \CDiv(X) \mid \sO_X(D) \cong L \}.
\]
For a non-zero global section $s \in H^0(X,L)$, there is an (unique) injective morphism $\iota : L \hookrightarrow \mathcal{K}_X$ which maps $s$ to $1$.
We denote by $\Div_L(s) \in |L|$ the Cartier divisor such that $\sO_X(\Div_L(s))=\Im(\iota)$.
We define the \emph{base locus} $\Bs(|L|)$ of $|L|$ by
\[
\Bs(|L|) : = \bigcap_{D \in |L|} \Supp(D)
\]
which is equal to the zero locus of the \emph{base ideal} $\bb(|L|) : = \Im( H^0(X,L) \otimes_k L^{-1} \to \sO_X)$.

\item Let $X$ be a proper variety and $L$ be an invertible sheaf on $X$ with $H^0(X,L) \neq 0$. 
We denote by $\Phi_{|L|} : X \dashrightarrow \PP(H^0(X,L))$ the rational map defined by $|L|$ (cf. \cite[Theorem II.7.1]{Har}).
If we further assume that $X$ is normal, then for a Weil divisor $D$ on $X$ with $H^0(X,\sO_X(D)) \neq 0$, we denote by $\Phi_{|D|} : X \dashrightarrow \PP(H^0(X, \sO_X(D))$ the rational map corresponding to 
\[
\Phi_{|\sO_U(D|_U)|}: U \dashrightarrow \PP(H^0(X,\sO_X(D))),
\]
where $U$ is the regular locus of $X$.

\item Let $X$ be a normal projective $n$-dimensional variety over a field $k$.
For a $\Q$-Weil divisor $D$, the \emph{volume} of $D$ is 
\[
\vol(D) = \vol(X,D) : = \limsup_{m \to \infty} \frac{h^0(X,\sO_X(mD))}{(m^n/n!)},
\]
where $m$ runs through all integers such that $mD$ is integral.
We say that $D$ is \emph{big} if we have $\vol(D)>0$.
We also say that $D$ is \emph{pseudo-effective} if $D+A$ is big for every ample $\Q$-Cartier divisor $A$ on $X$.

\item Let $X$ be a normal projective variety over an algebraically closed field $k$ and $D$ be a Cartier divisor on $X$.
The \emph{Iitaka dimension} of $D$ is defined to be
\[
  \kappa(D)=\kappa(X, D) : = \max_{m} \{\dim(\Im(\Phi_{|mD|}))\},
\]
where $m$ runs through all integers with $H^0(X, \sO_X(mD)) \neq 0$. 
If we have $H^0(X, \sO_X(mD))=0$ for all $m>0$, one puts $\kappa(D)=-\infty$.

\item 
Let $X$ be a projective normal variety over a field $k$ and $D$ be a Cartier divisor on $X$.
We further assume that $X$ is geometrically normal and geometrically connected over $k$.
The \emph{Iitaka dimension} of $D$ is defined to be 
\[
  \kappa(D)=\kappa(X, D) : = \kappa(X_{\overline{k}}, \overline{D}),
\]
where $\overline{D}$ is the flat pullback of $D$ by the natural morphism $X_{\overline{k}}=X \times_k \Spec \overline{k} \to X$.
By \cite[Corollary 2.1.38]{Laz} and \cite[Proposition III.9.3]{Har}, the Iitaka dimension is also characterized in terms of the growth of the sequence $\{h^0(X, \sO_X(mD))\}_{m}$.

\item 
Let $\pi : X \to T$ be a projective morphism between normal varieties over a field $k$ and $D$ be a $\Q$-Weil divisor on $X$.
We say that $D$ is \emph{big over $T$} (resp.~\emph{pseudo-effective over $T$}) if the flat pullback $D_{\eta}$ of $D$ to the generic fiber $X_{\eta}$ is big (resp.~pseudo-effective).
We further assume that $D$ is $\Q$-Cartier and $X_{\eta}$ is geometrically normal and geometrically connected over $K(T)$.
Then we write
\[
  \kappa(X/T, D) : = \kappa(X_{\eta}, D_{\eta}).
\]

\item Let $X, T$ be normal varieties over a field $k$ such that $X$ is projective over $T$.
An $\R$-Cartier divisor $D$ on $X$ is said to be \emph{semiample} over $T$ if there exist a $T$-morphism $f: X \to Y$ to a projective $T$-scheme $Y$ and an ample $\R$-Cartier divisor $H$ on $Y$ such that $D \sim_{\R} f^*H$.
When $D$ is $\Q$-Cartier, this definition coincides with the usual one, that is, $D$ is semiample if and only if some multiple of $D$ is free over $T$.

\end{enumerate}

\subsection{Singularities}

Here, we collect the definitions of singularities treated in this paper. 
Throughout this subsection, let $X$ be a normal variety over a field $k$.

The \emph{canonical sheaf} $\omega_X$ is the first non-zero cohomology of the dualizing complex $\omega_X^{\bullet} : = \pi^! \sO_{\Spec k}$, where $\pi : X \to \Spec k$ is the structure morphism.
It is well-known that $\omega_X$ satisfies the Serre's second condition $(S_2)$.
If $X$ is geometrically normal over $k$, then it follows from \cite[Lemma 0EA0]{Sta} that $\omega_X$ is isomorphic to the reflexive hull of $\Omega_{X/k}^{\wedge \dim X}$.
A \emph{canonical divisor} $K_X$ on $X$ is a Weil divisor on $X$ such that $\sO_X(K_X) \cong \omega_X$.

\begin{defn}
Let $\Delta$ be an effective $\R$-Weil divisor on $X$ such that $K_X+\Delta$ is $\R$-Cartier. 
\begin{enumerate}[label=\textup{(\roman*)}]
\item Given a projective birational morphism $\pi:Y \to X$ from a normal variety $Y$ and a prime divisor $E \subseteq Y$, we define the \emph{log discrepancy} $a_E(X,\Delta)$ of $(X,\Delta)$ at $E$ by 
\[
a_E(X,\Delta) : = \ord_E(K_Y -\pi^*(K_X+\Delta))+1 .
\]
\item We say that $(X, \Delta)$ is \emph{klt} (resp.~lc) if all the $a_E(X,\Delta) >0$ (resp.~$\ge 0$) for every prime divisor $E$ over $X$.
\item We define the \emph{minimal log discrepancy} $\mld(X,\Delta)$ of $(X,\Delta)$ by 
\[
  \mld(X,\Delta) : = \inf_{E} a_E(X,\Delta),
\]
where the infimum runs over all prime divisors $E$ over $X$.
\end{enumerate}
\end{defn}

\begin{lem}[\textup{\cite[Proposition 2.14]{BMPS+}}]
Let $V$ be a quasi-projective variety over a field $k$ of dimension three and a proper closed subset $T \subsetneq V$.
Then there exists a \emph{log resolution} $f : W \to V$, that is, a proper birational morphism $f$ from a regular variety $W$ such that $f^{-1}(T) \cup \Exc(f)$ is SNC.
Moreover, we may choose such a log resolution so that $f$ is a projective morphism.
\end{lem}

\begin{rem}[\textup{\cite[Corollary 2.31]{KM}}]
Let $\Delta$ be an effective $\R$-Weil divisor such that $K_X+\Delta$ is $\R$-Cartier.
Suppose that $f : Y \to X$ is a log resolution of $(X,\Delta)$.
Then $(X,\Delta)$ is klt (resp.~lc) if and only if $a_E(X,\Delta)>0$ (resp.~$ \ge 0$) for every prime divisor $E$ in $Y$.
\end{rem}

\begin{lem}[\textup{\cite[Lemma 2.34]{BMPS+}, cf.~\cite[Lemma 2.2]{Wal}}]\label{Bertini}
Let $X$ be a $3$-dimensional normal projective variety over a field and $B$ be an effective $\Q$-Weil divisor on $X$ such that $(X,B)$ is klt.
For an ample $\Q$-Cartier divisor $A$, there exists $0 \le A' \sim_{\Q} A$ such that $(X, B+A')$ is klt. 
\end{lem}

\begin{defn}
We say that $X$ has only \textit{rational} singularities if $X$ is Cohen-Macaulay and if there exists a resolution of singularities $f: Y \to X$ with $R^i f_*\sO_Y = 0$ for every integer $i>0$.
\end{defn}

\begin{rem}[\cite{CR}]\label{rmk on rational sing}
Assuming the existence of resolution of singularities, $X$ has only rational singularities if and only if for \emph{every} resolution of singularities $g: Z \to X$, we have $R^ig_*\sO_Z=0$ for $i>0$. 
\end{rem}

\begin{thm}[\textup{\cite[Corollary 1.3]{ABL}}]\label{ABL}
Let $(X,\Delta)$ be a klt pair over a perfect field of characteristic $p>5$.
Then $X$ has rational singularities.
\end{thm}

\begin{defn}
Suppose that $k$ is of characteristic $p>0$ and $F :X \to X$ denotes the Frobenius morphism. 
Let $\Delta \ge 0$ be an effective $\Q$-Weil divisor on $X$, $\ba \subseteq \sO_X$ be a non-zero coherent ideal and $t \ge 0$ be a real number.
\begin{enumerate}[label=\textup{(\roman*)}]
  \item We say that the triple $(X, \Delta, \ba^t)$ is \emph{$F$-pure} at $x \in X$ if for every large integer $e \gg 0$, there exists $d \in \ba_x^{\lfloor t(p^e-1) \rfloor}$ such that the morphism
\[
\sO_{X,x} \to (F^e_* \sO_{X}(\lfloor (p^e-1)\Delta \rfloor))_x \ ; \ a \mapsto F^e_*(a^{p^e}d)
\]
  splits as $\sO_{X,x}$-module homomorphism.
  \item We say that the pair $(X, \Delta)$ is $F$-pure at $x$ if so is the triple $(X, \Delta, \sO_X^{0})$.
  We say that $X$ is $F$-pure at $x$ if so is the pair $(X, 0)$.
  \item When $(X,\Delta)$ is $F$-pure at $x$, we define the \emph{$F$-pure threshold} $\fpt_x(X,\Delta; \ba)$ of $\ba$ at $x$ with respect to $(X,\Delta)$ by
  \[
\fpt_x(X,\Delta; \ba) : = \sup \{ t \ge 0 \mid (X, \Delta, \ba^t) \textup{ is $F$-pure at $x$}\}.
\]
\end{enumerate}
\end{defn}

\begin{defn}
Suppose that $k$ is of characteristic $p>0$. 
We say that $X$ is \emph{$F$-injective} at $x \in X$ if for every $i \ge 0$, the morphism
\[
H^i_{\m_x}(\sO_{X,x}) \to H^i_{\m_x}(\sO_{X,x})
\]
induced by the Frobenius homomorphism $F^{\#} : \sO_{X,x} \to \sO_{X,x}$ is injective.
\end{defn}

\begin{rem}\label{hierarchy}
Suppose that $k$ is of characteristic $p>0$.
\begin{enumerate}[label=\textup{(\roman*)}]
\item $($\textup{\cite{Fed}, \cite[Corollary 6.8]{HR}}$)$ We have the following hierarchy of the singularities:
\[
\textup{regular} \Longrightarrow \textup{$F$-pure} \Longrightarrow \textup{$F$-injective}.  
\]
\item $($\textup{\cite{Fed}, cf.~\cite[Proposition 2.6]{HW}}$)$ If $(X, D)$ is an SNC pair,  then $(X,D)$ is $F$-pure.
\end{enumerate}
\end{rem}

\subsection{Augmented base locus}
In this subsection, we recall the definition and basic properties of the augmented base loci.

\begin{defn}
Let $X$ be a proper variety over an algebraically closed field $k$ and $D$ be a Cartier divisor on $X$.
The \emph{stable base locus} of $D$ is the closed subset
\[
\B(D) : = \bigcap_{m \in \N_{>0}} \Bs(|mD|).
\]
\end{defn}

Noting that $\B(D)=\B(nD)$ for every integer $n>0$, we define the stable base locus $\B(E)$ of a $\Q$-Cartier divisor $E$ by $\B(E) : = \B(mE)$ where $m>0$ is any integer such that $mE$ is Cartier.

\begin{defn}
Let $X$ be a proper variety over an algebraically closed field $k$ and $D$ be a $\Q$-Cartier divisor on $X$.
The \emph{augmented base locus} of $D$ is the closed subset
\[
\B_+(D) : = \bigcap_{A} \B(D-A),
\]
where the intersection runs over all ample $\Q$-Cartier divisors $A$.
\end{defn}

\begin{rem}[\textup{\cite[Example 1.7]{ELMNP}}]\label{B+ basic}
Let $X$ be a proper variety over an algebraically closed field $k$ and $D$ be a $\Q$-Cartier divisor on $X$.
$D$ is ample if and only if $\B_+(D)=\emptyset$, and $D$ is big if and only if $\B_+(D) \neq X$.
\end{rem}

\begin{lem}[\textup{\cite[Theorem 1.4]{BirB+}, cf.~\cite[Theorem 0.3]{Nak}}]\label{B+ big}
Let $X$ be a projective variety over an algebraically closed field $k$, $L$ be a nef invertible sheaf on $X$ and $Z \subseteq X$ be a closed subvariety with $\dim Z \ge 1$.
If we have $Z \not\subseteq \B_+(L)$, then the restriction $L|_Z$ is big.
\end{lem}

\subsection{Seshadri constant}
In this subsection, we recall the definition and basic properties of the Seshadri constant.

\begin{defn}[\textup{cf.~\cite[Proposition 5.1.5]{Laz}}]\label{defn ses}
Let $X$ be a proper variety over an algebraically closed field $k$,  $D$ be a nef $\Q$-Cartier divisor on $X$ and $x \in X(k)$ be a closed point.
The \emph{Seshadri constant} of $D$ at $x$ is
\[
\ses{D} : = \inf_{C} \frac{(D \cdot C)}{\mult_x(C)},
\]
where the infimum runs over all curves $C$ in $X$ containing $x$.
\end{defn}

\begin{defn}
Let $X$ be a proper variety over an algebraically closed field $k$,  $Z=\{x_1, \dots, x_r\} \subseteq X$ be a finite set of closed points of $X$ and $I_Z \subseteq \sO_X$ be the reduced ideal which defines $Z$.
\begin{enumerate}[label=\textup{(\roman*)}]
\item For an integer $\ell \ge -1$,  a coherent sheaf $\mathcal{F}$ \emph{separates $\ell$-jets} at $Z$ if the restriction morphism
\[
H^0(X, \mathcal{F}) \to H^0(X, \mathcal{F} \otimes \sO_X/I_Z^{\ell+1})
\]
is surjective.
\item For a coherent sheaf $\mathcal{F}$, we denote by $s(\mathcal{F}; Z)$ the largest integer $\ell \ge -1$ such that $\mathcal{F}$ separates $\ell$-jets at $Z$.
\item $($\cite[Definition 7.2.4]{Mur}$)$ For a Cartier divisor $D$ on $X$, we set
\[
\ejet[Z]{D}: =\limsup_{m \to \infty} \frac{s(\sO_X(mD); Z)}{m}
\]
If $Z=\{x\}$ is a single point, then we denote $\ejet{D} : = \ejet[Z]{D}$.
\end{enumerate}
\end{defn}

By Lemma \ref{ejet basic}(3) below,  we define $\ejet[Z]{E}$ of a $\Q$-Cartier divisor $E$ by 
\[
\ejet[Z]{E} : = \ejet[Z]{mE}/m
\] where $m>0$ is any integer such that $mE$ is Cartier.

\begin{lem}[\textup{\cite[Lemma 7.2.5]{Mur}}]\label{ejet basic}
Let $X$ be a proper variety over an algebraically closed field $k$,  $Z=\{x_1, \dots, x_r\} \subseteq X$ be a finite set of closed points of $X$ and $I_Z \subseteq \sO_X$ be the reduced ideal which defines $Z$.
\begin{enumerate}[label=\textup{(\arabic*)}]
\item For coherent sheaves $\mathcal{F}$ and $\mathcal{G}$,  if $s(\mathcal{F};Z), s(\mathcal{G}; Z)>0$, then we have
\[
s(\mathcal{F}; Z)+s(\mathcal{G}; Z) \le s(\mathcal{F} \otimes \mathcal{G} ; Z).
\]
\item For a Cartier divisor $D$ on $X$,  we have
\[
\ejet[Z]{D} = \lim_{m \to \infty} \frac{s(\sO_X(mD); Z)}{m} = \sup_{m \in \N_{>0}} \frac{s(\sO_X(mD); Z)}{m}
\]
\item For a Cartier divisor $D$ and an integer $n>0$,  we have
\[
\ejet[Z]{nD}=n\ejet[Z]{D}.
\]
\end{enumerate}
\end{lem}

\begin{proof}
The proof is similar to that of the case where $Z$ is a single point (see \cite[Lemma 7.2.5]{Mur}).
\end{proof}

\begin{lem}[\textup{\cite[Proposition 7.2.10]{Mur}}]\label{ses vs ejet}
Let $X$ be a proper variety over an algebraically closed field $k$,  $D$ be a Cartier divisor on $X$ and $x \in X(k)$ be a closed point.
If $D$ is nef and $x \not\in \B_+(D)$, then one has
\[
\ses{D}=\ejet{D}.
\]
\end{lem}

\begin{lem}\label{s semiconti}
Let $X$ be a projective variety over an algebraically closed field $k$ and $D$ be a Cartier divisor on $X$.
Then the map
\[
X(k) \to \Z \ ; \ x \mapsto s(\sO_{X}(D); x)
\]
is lower semicontinuous.
\end{lem}

\begin{proof}
Fix a closed point $x \in X$ and we set $\ell : = s(\sO_{X}(D); x)$.
It suffices to show that there is an open neighborhood $U$ of $x$ such that for every closed point $y \in U$, $D$ separates $\ell$-jets at $y$.
If $\ell=-1$,  then there is nothing to say.
We assume that $\ell \ge 0$.
Since $\sO_{X}(D)$ is globally generated at $x$,  after replacing $D$ by a suitable divisor in $|D|$, we may assume that $x \not\in \Supp(D)$.

Since $V : = H^0(X, \sO_X(D)) \subseteq K(X)$ is contained in $\sO_{X,x}$,  we may find an open affine neighborhood $U=\Spec(A)$ of $x$ such that $V \subseteq A$.
We may also assume that $D|_U=0$.
Let $I \subseteq A \otimes_k A$ be the ideal of the diagonal immersion $U \hookrightarrow U \times_k U$ and we set
\[
U^{(\ell)} : = Z(I^{\ell+1}) = \Spec((A \otimes_k A)/I^{\ell+1})  \subseteq U \times_k U.
\]
We consider $M: = \sO_{U^{(\ell)}} = (A\otimes_k A)/I^{\ell+1}$ as an $A$-module by the action on the second factor.

Take any closed point $y \in U$ and let $\m_y \subseteq A$ be the corresponding maximal ideal.  
Then one has the following commutative diagram
\[
\xymatrix{
V \ar@{^{(}->}[r] & A \ar@{->>}[r] \ar^-{f}[d]& A/\m_y^{\ell+1} \ar^-{g}[d] \\
& M \ar@{->>}[r] & M/\m_y M
},\]
where $f$ is the $k$-linear map $a \mapsto \overline{a \otimes 1}$ and $g$ is the induced map which is bijective.
Then $D$ separates $\ell$-jets at $y$ if and only if the horizontal morphism $V \hookrightarrow A \onto A/\m_y^{\ell+1}$ is surjective,  which is equivalent to the surjectivity of the composite map $V \xrightarrow{f} M \onto M/\m_y M$.
Let $N \subseteq M$ be the sub $A$-module generated by $f(V) \subseteq M$.
Then it follows from the Nakayama's lemma that $D$ separates $\ell$-jets at $y$ if and only if $N_{\m_y}=M_{\m_y}$, which is an open condition as desired.
\end{proof}

\begin{prop}\label{ejet semiconti}
Let $X$ be a projective variety over an algebraically closed field $k$, $E$ be a Cartier divisor on $X$ and $\alpha>0$ be a real number.
Then the locus
\[
\{x \in X(k) \mid \ejet{E}>\alpha\}
\]
is open.
\end{prop}

\begin{proof}
It follows from Lemma \ref{ejet basic} (2) and Lemma \ref{s semiconti}.
\end{proof}

\subsection{Non-fixed effective divisors}
In this subsection, we collect several properties concerning a divisor $D$ on a proper variety $X$ with $h^0(X, \sO_X(D)) \ge 2$.

  \begin{lem}\label{universal family}
  Let $M$ be an invertible sheaf on a proper variety $X$ over an  algebraically closed field $k$.
  We further assume that $H^0(X,M) \neq 0$.
  We write $\PP : = \PP(H^0(X,M)^*)$
  \begin{enumerate}[label=\textup{(\arabic*)}]
  \item There exists an effective Cartier divisor $\mathcal{Z}_M \subseteq X \times_k \PP$ such that 
  \begin{enumerate}[label=\textup{(\roman*)}]
  \item $\mathcal{Z}_M$ is flat over $\PP$.
  \item For every closed point $t \in \PP$, the fiber $(\mathcal{Z}_M)_t \subseteq X$ is contained in $|M|$. 
  \item The induced map
  \[
  \Phi: \PP(k) \to |M| \ ; \ t \mapsto (\mathcal{Z}_M)_t
  \]
  is bijective.
  \end{enumerate}
  \item Let $N$ be another invertible sheaf with $H^0(X,N) \neq 0$.
  Then there exists a $k$-morphism $\mu: \PP(H^0(X,M)^*) \times_k \PP(H^0(X,N)^*) \to \PP(H^0(X, M \otimes N)^*)$ which corresponds to the map
  \[
  |M| \times |N| \to |M \otimes N| \ ; \ (D, E) \mapsto D+E.
  \]
  \end{enumerate}
  \end{lem}
  
  \begin{proof}
  (1) Take a $k$-basis $s_0, s_1, \dots, s_m \in H^0(X,M)$ and let $s_0^*, \dots, s_m^* \in H^0(X,M)^*$ be the dual basis.
  Consider the invertible sheaf $\mathcal{M} : = p_1^*M \otimes_{\sO_X} p_2^* \sO_{\PP}(1)$ on $X \times_k \PP$ and the global section
  \[
  s : = \sum_{i=0}^m p_1^*(s_i) \otimes p_2^*(s_i^*) \in H^0(X \times \PP, \mathcal{M}).
  \]
  Let $\mathcal{Z}_{M}$ be the closed subscheme corresponding to the effective Cartier divisor $\Div_{\mathcal{M}}(s)$.
  The assertions in (ii) and (iii) are obvious and (i) follows from \cite[Theorem 22.6]{Mat}.
  
  For (2), we fix a $k$-basis $t_0, \dots, t_n$ of $H^0(X, N)$ and a $k$-basis $u_0, \dots, u_l \in H^0(X, M \otimes N)$.
  For every $i,j$, we can write 
  \[
  s_i \otimes t_j = \sum_{q=1}^l a_{i,j}^{(q)} u_q
  \]
  for some $a_{i,j}^{(q)} \in k$.
  Then we can show that a map 
  \[
  \mu: \PP(H^0(X,M)^*)(k) \times \PP(H^0(X,N)^*)(k) \to \PP(H^0(X, M \otimes N)^*)(k)
  \] which sends a closed point $( [b_0, \dots, b_m], [c_0, \dots, c_n])$ of $\PP(H^0(X,M)^*) \times_k \PP(H^0(X,N)^*)$ to $[\sum_{i,j} b_i c_j a_{i,j}^{(0)}, \sum_{i,j} b_i c_j a_{i,j}^{(1)}, \dots , \sum_{i,j} b_i c_j a_{i,j}^{(l)}]$ is a (well-defined) $k$-morphism, which corresponds to the map
  \[
  |M| \times |N| \to |M \otimes N| \ ; \ (D, E) \mapsto D+E.
  \]
  \end{proof}
  
  \begin{lem}\label{movable divisor}
  Let $L$ be an invertible sheaf on a proper variety $X$ over an algebraically closed field $k$.
  Then the following conditions are equivalent to each other.
  \begin{enumerate}[label=\textup{(\alph*)}]
  \item $h^0(X,L) \ge 2$.
  \item $\#|L| \ge 2$.
  \item For any closed point $x \in X$, there exists a divisor $D \in |L|$ such that $x \in \Supp(D)$.
  \end{enumerate}
  \end{lem}
  
  \begin{proof}
  (a) $\Rightarrow$ (c): Since the kernel of the restriction map
  \[
  H^0(X, L) \to H^0(X, L \otimes \kappa(x)) \cong k
  \]
  is non-zero, we may find a non-zero global section $s \in H^0(X,L)$ which vanishes at $x$.
  Then $x$ is contained in the support of $D : = \Div_L(s)$.
    
  (c) $\Rightarrow$ (b): Obvious.
  
  (b) $\Rightarrow$ (a): This follows from the bijection $|L| \cong (H^0(X,L) \setminus \{0\}) /k^*$ (cf.~\cite[Proposition II. 6.15]{Har}).

  \end{proof}

\begin{lem}\label{intersection with nef big}
Let $X$ be a projective variety over an algebraically closed field $k$, $H$ be a nef and big Cartier divisor on $X$ and $\Gamma$ be a Cartier divisor on $X$ with $h^0(X, \sO_X(\Gamma)) \ge 2$.
Then we have 
\[
  (H^{\dim X -1}\cdot \Gamma) \ge 1.
  \]
\end{lem}

\begin{proof}
By Remark \ref{B+ basic}, we have $\B_+(H) \not=X$.
Take a smooth closed point $x \in X \setminus \B_+(D)$.
It follows from Lemma \ref{movable divisor} that there exists an effective divisor $E \in |\Gamma|$ such that $x \in \Supp(E)$.
Let $Z$ be an irreducible component of $\Supp(E)$ which contains $x$.
Noting that $H|_Z$ is nef and big by Lemma \ref{B+ big}, we have
\[
  (H^{\dim X -1} \cdot \Gamma) \ge (H^{\dim X -1} \cdot Z) =(H|_{Z}^{\dim Z}) =\vol(H|_Z) >0,
\]
as desired.
\end{proof}

\begin{lem}\label{positive Iitaka dim}
Let $X$ be a geometrically normal and geometrically connected projective variety over a field $k$ and $D$ be a Cartier divisor on $X$. 
If we have
\[
  h^0(X,\sO_X(D)) \ge 2,
\]
then one has $\kappa(D) \ge 1$.
\end{lem}
    
  \begin{proof}
  After replacing $X$ by $X \times_k \overline{k}$, we may assume that $k$ is algebraically closed.
  Let $Y$ be the image of $\Phi_{|D|}$ and 
  \[
    \phi : U : = X \setminus \Bs(|D|) \to Y
  \]
be the induced morphism.
Since the image of the homomorphism 
\[
\phi^* : H^0(Y, \sO_Y(1)) \to H^0(U, \sO_U(D))
\]
contains $H^0(X, \sO_X(D))$, we have $h^0(Y, \sO_Y(1)) \ge 2$.
This shows that $\dim(Y) \ge 1$, as desired.
\end{proof}

\subsection{Spreading out}
In this subsection, we recall several facts about spreading out sheaves, schemes and morphisms.

\begin{lem}\label{spreading out sheaves}
Let $A \subseteq C$ be a flat extension between Noetherian rings and $X$ be a projective scheme over $\Spec (A)$.
Suppose that $\mathcal{F}$ is a coherent sheaf on $X_C : = X \times_A \Spec (C)$.
Then there is a sub $A$-algebra $B \subseteq C$ of finite type over $A$ with a coherent sheaf $\mathcal{G}$ on $X_B$ such that 
\[
  \phi^* \mathcal{G} \cong \mathcal{F}, 
  \] where $\phi :X_C \to X_B$ is the natural morphism.
\end{lem}

\begin{proof}
We fix an ample invertible sheaf $\sO_{X}(1)$ on $X$.
Take integers $m,n,a,b \in \Z$ such that there exists an exact sequence
\[
  \sO_{X_C}(a)^{\oplus m} \xrightarrow{f} \sO_{X_C}(b)^{\oplus n} \to \mathcal{F} \to 0.
\]
Then the assertion follows if we take $B$ such that the element
\[
  f \in \Hom(\sO_{X_C}(a)^{\oplus m}, \sO_{X_C}(b)^{\oplus n}) \cong \Hom(\sO_{X}(a)^{\oplus m}, \sO_{X}(b)^{\oplus n}) \otimes_A C
\]
is defined over $B$.
\end{proof}

\begin{lem}\label{spreading out schemes}
Let $A \subseteq C$ be a flat extension between Noetherian rings and $Y$ be a projective scheme over $\Spec (C)$.
Then there is a sub $A$-algebra $B \subseteq C$ of finite type over $A$ with a projective scheme $Z$ over $\Spec (B)$ such that 
\[
  Y \cong Z \times_B \Spec(C).
  \]
\end{lem}
  
\begin{proof}
  We fix a closed immersion $Y \hookrightarrow \PP_C^N$.
  As in the proof of Lemma \ref{spreading out sheaves}, we may find a sub $A$-algebra $B \subseteq C$ of finite type over $A$ with a coherent ideal sheaf $I \subseteq \sO_{\PP^N_B}$ such that the extension $I \cdot \PP^N_{C}$ is the ideal of $Y$.
\end{proof}

\begin{lem}[\textup{\cite[Th\'{e}or\`{e}me (8.8.2)]{EGA}}]\label{spreading out morphisms}
Let $A \subseteq C$ be an extension between Noetherian rings and $X, Y$ be projective schemes over $\Spec (A)$.
\begin{enumerate}[label=\textup{(\arabic*)}]
  \item Suppose that $f: X_C \to Y_C$ is a $C$-morphism.
  Then there is a sub $A$-algebra $B \subseteq C$ of finite type over $A$ with a $B$-morphism $g : X_B \to Y_B$ such that the following diagram commutes:
  \[
    \xymatrix{ X_B \times_B \Spec(C) \ar^-{g \times_B \id_C}[d] \ar^-{\sim}[r] & X_C \ar^-{f}[d] \\
    Y_B \times_B \Spec(C) \ar^-{\sim}[r] & Y_C }
  \]
  \item Suppose that $f,g : X \to Y$ are $A$-morphisms such that $f \times_A \id_C = g \times_A \id_C$.
  Then there is a sub $A$-algebra $B \subseteq C$ of finite type over $A$ such that $f \times_A \id_B = g \times_A \id_B$.
\end{enumerate}
\end{lem}

\subsection{Bounded family}

We recall the definitions and basic properties of boundedness.

\begin{defn}\label{defn bdd}
Let $k$ be an algebraically closed field and $\mathcal{D}$ be a set of pairs $(X, \Delta)$ which consists of a (not necessarily normal) projective variety $X$ over $k$ and an effective $\Q$-Weil divisor $\Delta$ on $X$.
\begin{enumerate}[label=\textup{(\arabic*)}]
\item We say that $\mathcal{D}$ is \emph{bounded} if there exist a projective morphism $\mathcal{Z} \to T$ between schemes of finite type over $k$ and a closed subset $B$ on $\mathcal{Z}$ which satisfies the following property:\\
For every $(X,\Delta) \in \mathcal{D}$, there is a closed point $t \in T$ with an isomorphism 
\[
f: \mathcal{Z}_t \xrightarrow{\ \sim\ } X
\] 
which satisfies $f(B_t)=\Supp(\Delta)$.

\
\item We say that $\mathcal{D}$ is \emph{log birationally bounded} if there exists a bounded family $\mathcal{E}$ which satisfies the following property:
For every $(X,\Delta) \in \mathcal{D}$, there is an element $(Y, \Gamma) \in \mathcal{E}$ with a birational morphism 
\[
f: Y \dashrightarrow X
\]
such that the support $\Supp(\Gamma)$ of $\Gamma$ coincides with the union of the support of the strict transform $f^{-1}\Delta$ and the image of $\nu_*F$ of all $g$-exceptional divisors $F$, where $\nu : Y^n \to Y$ is the normalization and $g : Y^n \dashrightarrow X$ is the induced birational map.

\
\item Let $\mathcal{E} : = \{X_j\}_{j \in J}$ be a set of normal projective varieties over $k$.
We say that $\mathcal{E}$ is \emph{bounded} (resp.~\emph{birationally bounded}) if the set $\{(X_j,0)\}_{j \in J}$ of projective log pairs is bounded (resp.~log birationally bounded).
\end{enumerate}
\end{defn}

The following lemma should be well-known to experts, but we include a proof here for the sake of completeness.

\begin{lem}[\textup{\cite[Lemma 2.4.2]{HMX13}, cf.~\cite[Proposition 5.3]{Tan}}]\label{bb criterion}
Fix a positive integer $n$.
Let $k$ be an algebraically closed field and $\mathcal{D}$ be a set of pairs $(X, \Gamma)$ which consists of an $n$-dimensional projective variety $X$ over $k$ and a reduced Weil divisor $\Gamma$ on $X$.
We further assume that there are constants $V_1$ and $V_2$ such that for every $(X, \Gamma) \in \mathcal{D}$, we may find a very ample invertible sheaf $A$ on $X$ with $(A^n)<V_1$ and $(A^{n-1} \cdot \Gamma)<V_2$.
Then $\mathcal{D}$ is bounded.
\end{lem}

\begin{proof}
We set $N : = V_1+n-1$.
By the proof of \cite[Proposition 5.3]{Tan}, the set of all Hilbert polynomials of an $n$-dimensional closed subvariety of $\PP^N_k$ with degree smaller than $V_1$ is a finite set.
Therefore, by using the Hilbert schemes, we obtain a projective $k$-scheme $H_1$ and a closed subscheme $\mathcal{X}_1 \subseteq H_1 \times_k \PP^N_k$ such that for every $n$-dimensional closed subvariety $X \subseteq \PP^N_k$ with $\deg (X) <V_1$, there is a closed point $t \in H_1$ such that 
\[
  X = (\mathcal{X}_1)_t
\]
as closed subschemes of $\PP^N_k$.
Similarly, we have a projective $k$-scheme $H_2$ and a closed subscheme $\mathcal{X}_2 \subseteq H_2 \times_k \PP^N_k$ which bounds all $n-1$-dimensional closed subvarieties of $\PP^N_k$ with degree smaller than $V_2$.

For every integer $m \ge 0$, we set $T^{(m)}  : = H_1 \times_k H_2^{\times m}$ and let 
\begin{align*}
  p_{1,m+2}: T^{(m)} \times_k \PP^N_k \to H_1 \times_k \PP^N_k & \\
  p_{i, m+2} : T^{(m)} \times_k \PP^N_k \to H_2 \times_k \PP^N_k & \ (i=2, \dots, m+1)
\end{align*}
be the projections. 
We also write
\begin{align*}
  \mathcal{X}^{(m)} & : = p_{1,m+2}^{-1}(\mathcal{X}_1) \subseteq T^{(m)} \times_k \PP^N_k \\
  B^{(m)} & : = \mathcal{X}^{(m)} \cap \bigcup_{i=1}^{m} p_{i+1,m+2}^{-1}(\mathcal{X}_2) \subseteq T^{(m)} \times_k \PP^N_k.
\end{align*}

Take an element $(X, \Gamma) \in \mathcal{D}$.
By repeatedly applying \cite[Excercise I.7.7]{Har}, we can identify $X$ as an $n$-dimensional closed subvariety of $\PP^N_k$ with $\deg(X)<V_1$.
Let $\Gamma = \sum_{i=1}^m a_iE_i$ be the irreducible decomposition.
Then we have $m < V_2$ and $\deg(E_i)<V_2$.
This shows that $\mathcal{D}$ is bounded by
\[
  (\coprod_{m=1}^{V_2} \mathcal{X}^{(m)}, \coprod_{m=1}^{V_2} B^{(m)}) \to \coprod_{m=1}^{V_2} T^{(m)}.
\]
\end{proof}

\begin{lem}\label{bdd of normalization}
  Let $\mathcal{E} : = \{(Y_j, \Gamma_j)\}_{j \in J}$ be a bounded family where $Y_i$ is a projective variety over an algebraically closed field $k$.
  Then the set 
  \[\mathcal{E}^n : = \{(Y^n_j, \nu_j^{-1}(\Gamma_j))\}_{j \in J}\]
  is bounded, where $\nu_j : Y^n_j \to Y_j$ is the normalization.
  \end{lem}
  
  \begin{proof}
  Let $\pi : (\mathcal{Y}, B) \to T$ be a family which bounds $\mathcal{E}$, that is, $\pi : \mathcal{Y} \to T$ is a projective morphism with integral fibers between schemes of finite type over $k$ and $B \subseteq \mathcal{Y}$ is a closed subset such that for every $j \in J$, there is a closed point $t_j \in T$ and an isomorphism
  \[
    f_j : \mathcal{Y}_{t_j} \xrightarrow{\sim} X_j
  \]
  with $f_j(B_{t_j}) = \Supp(\Gamma_j)$.
  If there is a surjective morphism $S \to T$ from a scheme $S$ of finite type over $k$, then we may replace $T$ by $S$ and $\mathcal{Y}$ by $\mathcal{Y} \times_T S$.
  In particular, we may assume that $T$ is reduced.

  For every locally closed subset $U \subseteq T$, we denote by $J_U$ the subset
  \[
    J_U : = \{j \in J \mid t_j \in U\}
  \]
  of $J$ and by $\mathcal{E}_U^n$ the subset
  \[
    \mathcal{E}^n_U : =\{(Y_j^n, \nu_j^{-1}(\Gamma_j))\}_{j \in J_U}
  \]
  of $\mathcal{E}^n$.
  By Noether induction on $T$, we may assume that $\mathcal{E}_C^n$ is bounded for every proper closed subset $C \subsetneq T$.
  Therefore, it suffices to show that there exists a non-empty open subset $U \subseteq T$ such that $\mathcal{E}^n_U$ is bounded.
  In order to show this, we may always shrink $T$.
  In particular, we may assume that $T$ is a variety.
  
  Take the normalization 
  \[
    \nu : (\mathcal{Y}_{\overline{\eta}})^n \to \mathcal{Y}_{\overline{\eta}}
    \]
  of the geometric generic fiber $\mathcal{Y}_{\overline{\eta}}$.
  It then follows from Lemma \ref{spreading out schemes} and Lemma \ref{spreading out morphisms} that after replacing $T$, we may find a morphism 
  \[\mu: \mathcal{Y}' \to \mathcal{Y}\]
  whose geometric generic fiber is isomorphic to $\nu$.

  Since $\nu$ is finite and birational, after replacing $T$ again, we may assume that $\mu$ is finite and birational.
  Then the morphism 
  \[\mu_t : \mathcal{Y}'_t \to \mathcal{Y}_t\]
  induced on a general closed fiber is also finite and birational.
  On the other hand, since the geometric generic fiber of $\mathcal{Y}'$ is normal, so is a general closed fiber $\mathcal{Y}'_t$.
  This shows that $\mu_t$ is the normalization of $\mathcal{Y}_t$ for general $t$.
  After shrinking $T$, $\mathcal{E}^n$ is bounded by the pair $(\mathcal{Y}', \mu^{-1}(B))$ over $T$, as desired.
  \end{proof}

\subsection{Non-vanishing and base point free theorems}

In this subsection, we recall several facts about non-vanishing theorem and base point free theorem in dimension $3$ of characteristic $p>5$.

\begin{thm}[\textup{\cite[Theorem 3]{Jak}}]\label{NV}
  Let $X$ be a $3$-dimensional projective over a field $k$ of characteristic $p>5$ and $\Delta \ge 0$ be an effective $\Q$-Weil divisor on $X$ such that $K_X+\Delta$ is $\Q$-Cartier.
  Suppose that the base change $X_{\overline{k}} : = X \times_k \Spec \overline{k}$ of $X$ to the algebraic closure of $k$ is a normal variety. 
  We further assume that 
  \begin{enumerate}[label=\textup{(\roman*)}]
    \item $K_X+\Delta$ is pseudo-effective, and
    \item $(X_{\overline{k}}, \overline{\Delta})$ is klt, where $\overline{\Delta}$ is the flat pullback of $\Delta$ by the natural morphism $X_{\overline{k}} \to X$.
  \end{enumerate}
  Then $\kappa(X, K_X+\Delta) \ge 0$.
\end{thm}

\begin{proof}
  After replacing $(X, \Delta)$ by $(\overline{X}, \overline{\Delta})$, the assertion follows from \cite[Theorem 3]{Jak}.
\end{proof}

\begin{thm}\label{bpf}
Let $X$ be a $3$-dimensional projective normal variety over a field $k$ of characteristic $p>5$ and $\Delta \ge 0$ be an effective $\R$-Weil divisor on $X$ such that $K_X+\Delta$ is a nef $\R$-Cartier divisor.
Then $K_X+\Delta$ is semiample if one of the following properties holds:
\begin{enumerate}[label=\textup{(\alph*)}]
  \item $(X,\Delta)$ is klt and there is a big $\Q$-Cartier divisor $A$ with $0 \le A \le \Delta$, or
  \item $k$ is algebraically closed, $\Delta$ is a $\Q$-Wei divisor and $\kappa(X, K_X+\Delta) \ge 1$.
\end{enumerate}
\end{thm}

\begin{proof}
We first assume that the first assumption holds.
We write $\Delta = A + B$ and $A=F+H$, where $F$ is effective and $H$ is ample.
Take a real number $\epsilon > 0$ such that $(X, \Delta' : = (1-\epsilon)\Delta + \epsilon (B+F) )$ is klt.
Noting that $(K_X+\Delta) - (K_X+\Delta')$ is ample, the assertion follows from \cite[Theorem 1.1]{Wal}.

In the second case, the assertion follows from \cite[Theorem A]{DWAb}, \cite[Theorem 1.3]{WalAb} and \cite[Theorem 1.4]{BirMMP}.
\end{proof}

\subsection{Models of pairs}

Let $\phi: X \dashrightarrow Y$ be a birational map between normal varieties $X,Y$ over a (not-necessarily algebraically closed) field $k$ and $U_\phi \subseteq X$ be the maximal open subset on which $\phi$ is a morphism.
We say that $\phi$ is \emph{birational contraction} if $\phi : U_\phi \to Y$ is surjective in codimension one.

\begin{defn}
  Let $X,Y$ be normal varieties over a field $k$, $D$ be an $\R$-Cartier divisor on $X$ and $\phi: X \dashrightarrow Y$ be a birational contraction.
  Suppose that $p : W \to X$ and $q: W \to Y$ be proper birational morphisms from a normal variety $W$ with $\phi \circ p = q$.
  \[\xymatrix{ & W \ar_-{p}[ld] \ar^-{q}[rd]& \\ X \ar@{.>}^-{\phi}[rr] && Y}\]
  $\phi$ is said to be \emph{$D$-non-positive} (resp.~\emph{$D$-negative}) if $\phi_*D$ is $\R$-Cartier and $E : =p^*D - q^*(\phi_*D)$ is effective (resp.~$E$ is effective and its support contains the strict transform $p^{-1}_*F$ of every $\phi$-exceptional divisor $F$).
  This definition is independent of the choice of $W$.
\end{defn}

\begin{rem}\label{convexity}
  Let $X,Y$ be normal varieties over a field $k$, $D, D'$ be $\R$-Cartier divisors on $X$ and $\phi: X \dashrightarrow Y$ be a birational contraction.
  If $\phi$ is both $D$-non-positive (resp.~$D$-negative) and $D'$-non-positive, then for any real numbers $a>0$ and $b \ge 0$, $\phi$ is $aD+bD'$-non-positive (resp.~$aD+bD'$-negative).
\end{rem}

\begin{defn}
  Let $\mathcal{X},\mathcal{Y}, T$ be normal varieties over a field $k$ such that $\mathcal{X}$ and $\mathcal{Y}$ are projective over $T$, $\mathcal{D}$ be an $\R$-Cartier divisor on $\mathcal{X}$ and $\Phi: \mathcal{X} \dashrightarrow \mathcal{Y}$ be a birational contraction over $T$.
  $\Phi$ is said to be an \emph{ample model} (resp.~\emph{semiample model}) of $\mathcal{D}$ if 
  \begin{enumerate}
    \item $\Phi$ is $\mathcal{D}$-non-positive and 
    \item $\Phi_* \mathcal{D}$ is ample (resp.~semiample) over $T$.
  \end{enumerate}
\end{defn}

Let $\mathcal{X}, \mathcal{Y}, T$ be normal varieties over a field $k$ such that $\mathcal{X}$ and $\mathcal{Y}$ are projective over $T$, and $\Phi: \mathcal{X} \dashrightarrow \mathcal{Y}$ be a birational map over $T$.
We further assume that $\pi_X: \mathcal{X} \to T$ and $\pi_Y: \mathcal{Y} \to T$ are surjective with geometrically integral fibers.
Then for a general point $t \in T$, $\Phi$ defines a birational map
\[\Phi_t : \mathcal{X}_t \dasharrow \mathcal{Y}_t.\]

Similarly, for an open subset $U \subseteq T$, we write $\mathcal{X}_U : = \pi_X^{-1}(U)$ and $\mathcal{Y}_U : = \pi^{-1}(U)$.
We also denote by 
\[\Phi_U : \mathcal{X}_U \dashrightarrow \mathcal{Y}_U\]
the induced birational map.
  
\begin{lem}\label{model and fiber}
  With the above notation, let $\eta \in T$ be the generic point and $\mathcal{D}$ be an $\R$-Cartier divisor on $\mathcal{X}$.
  Then the following holds:
    \begin{enumerate}[label=\textup{(\arabic*)}]
      \item If $\Phi_{\eta}$ is a birational contraction (resp.~$\mathcal{D}_{\eta}$-non-positive, $\mathcal{D}_{\eta}$-negative, ample model of $\mathcal{D}_{\eta}$, or semiample model of $\mathcal{D}_{\eta}$), then there is an open dense subset $U \subseteq T$ such that the restriction $\Phi_U$ is a birational contraction (resp.~$\mathcal{D}_U$-non-positive, $\mathcal{D}_U$-negative, ample model of $\mathcal{D}_U$, or semiample model of $\mathcal{D}_U$).
      \item We further assume that general fibers $\mathcal{X}_t$ and $\mathcal{Y}_t$ are normal.
      If $\Phi$ is a birational contraction (resp.~$\mathcal{D}$-non-positive, $\mathcal{D}$-negative, ample model of $\mathcal{D}$, or semiample model of $\mathcal{D}$), then for a general point $t \in T$, $\Phi_t$ is a birational contraction (resp.~$\mathcal{D}_t$-non-positive, $\mathcal{D}_t$-negative, ample model of $\mathcal{D}_t$, or semiample model of $\mathcal{D}_t$).
      \item We further assume that general fibers $\mathcal{X}_t$ and $\mathcal{Y}_t$ are normal, $\Phi$ is a birational contraction and $\Phi_*\mathcal{D}$ is $\R$-Cartier.
      Then we have
      \[
        (\Phi_*\mathcal{D})_t =(\Phi_t)_*\mathcal{D}_t
        \]
        for a general point $t \in T$.
    \end{enumerate}
\end{lem}
    
\begin{proof}
Take a proper birational morphism $p : \mathcal{W} \to \mathcal{X}$ and $q: \mathcal{W} \to \mathcal{Y}$ from a normal variety $\mathcal{W}$ such that $q = \Phi \circ p$.
\[
  \xymatrix{ & \mathcal{W} \ar_-{p}[ld] \ar^-{q}[rd] & \\
  \mathcal{X} \ar@{.>}^-{\Phi}[rr] && \mathcal{Y}}
\]
Noting that $\mathcal{X}_{\eta}$ is geometrically integral over $K(T)$, the generic fiber $\mathcal{W}_{\eta}$ is smooth over $K(T)$ in codimension $0$.
This shows that there exists an open dense subset $T^{\circ} \subseteq T$ such that $\mathcal{W}_t$ is integral for every $t \in T^{\circ}$.
Let $\nu_t: \mathcal{W}_t^n \to \mathcal{W}_t$ be the normalization and 
\[
  p_t^n : = p_t \circ \nu_t : \mathcal{W}_t^n \to \mathcal{X}_t
\] be the induced morphism.
Then it is straightforward to show that after shrinking $T^{\circ}$, for every point $t \in T^{\circ}$, we have
\[
\Exc^1(p_{t}^n) = \nu_t^{-1}((\Exc^1(p))_t),
\]
where $\Exc^1(p_t^n)$ (resp.~$\Exc^1(p)$) is the union of all $p_t^n$-exceptional (resp.~$p$-exceptional) divisors.
Similarly, if we set $q_t^n : = q_t \circ \nu_t$, one has
\[
  \Exc^1(q_{t}^n) = \nu_t^{-1}((\Exc^1(q))_t).
\]
Therefore, if $\Phi_{\eta}$ is a birational contraction, then we have
\[
  \Exc^1(p)_{\eta} = \Exc^1(p_{\eta}) \subseteq \Exc^1(q_{\eta}) = \Exc^1(q)_{\eta}.
\]
Then we may find an open dense subset $U \subseteq T$ such that $\Exc^1(p) \cap \mathcal{W}_U \subseteq \Exc^1(q) \cap \mathcal{W}_U$.
This shows the assertion in (1) for the "birational contraction" case.
The proof of the assertion in (2) for the "birational contraction" case is similar.

For the rest of the proof, we assume that $\Phi$, $\Phi_{\eta}$ and $\Phi_t$ are birational contractions.
We write $\mathcal{E} : = \Phi_*\mathcal{D}$.
Take a closed subsets $\mathcal{Z}_1 \subseteq \mathcal{X}$ and $\mathcal{Z}_2 \subseteq \mathcal{Y}$ such that $\cod(\mathcal{Z}_2, \mathcal{Y}) \ge 2$ and that the open subvarieties $\mathcal{X}^{\circ} : = \mathcal{X} \setminus \mathcal{Z}_1$ and $\mathcal{Y}^{\circ} : = \mathcal{Y} \setminus \mathcal{Z}_2$ are isomorphic.
Then the flat pullback $\mathcal{E}_{\eta}$ of $\mathcal{E}$ coincides with the strict transform $(\Phi_{\eta})_*\mathcal{D}_{\eta}$ after restricting to the open subset $\mathcal{Y}_{\eta}^{\circ}$.
Combining this with $\cod((\mathcal{Z}_2)_{\eta}, \mathcal{Y}_{\eta}) \ge 2$, we conclude that 
\begin{align}\label{compatible of push}
  \mathcal{E}_{\eta} = (\Phi_{\eta})_* \mathcal{D}_{\eta}.
\end{align}
In particular, if $(\Phi_{\eta})_* \mathcal{D}_{\eta}$ is $\R$-Cartier, then $\mathcal{E}$ is $\R$-Cartier around $\mathcal{Y}_{\eta}$.
Therefore, for the rest of the proof, we may always assume that $\mathcal{E}$ is $\R$-Cartier.

The assertion in (3) follows from the similar argument as in the proof of the inequality \eqref{compatible of push}.
The assertions in (1) and (2) for the "$\mathcal{D}$-non-positive" case and the "$\mathcal{D}$-negative" case follows from the equality
\[
  (p^*\mathcal{D} - q^*(\mathcal{E}) )_{t} = p_{t}^* \mathcal{D}_t - q_{t}^* ((\Phi_{t})_* \mathcal{D}_{t}).
\]

For the "ample model" case and the "semiample model" case, the assertion in (2) is obvious and the assertion in (1) follows from the following well-known lemma.
\end{proof}

\begin{lem}
Let $ \mathcal{Y} \to T$ be a projective morphism between normal varieties over a field $k$ and $\mathcal{E}$ be an $\R$-Cartier divisor on $\mathcal{Y}$.
If $\mathcal{E}_{\eta}$ is ample (resp.~semiample), then after shrinking $T$, $\mathcal{E}$ is ample (resp.~semiample) over $T$.
\end{lem}

\begin{proof}
We first consider the case where $\mathcal{E}$ is Cartier and $\mathcal{E}_{\eta}$ is very ample.
Take a closed immersion $\iota : \mathcal{Y}_{\eta} \hookrightarrow \PP^N_{K(T)}$ and a hyperplane $H \subseteq \PP^N_{K(T)}$ such that $\mathcal{E}_{\eta}=\iota^*H$.
By Lemma \ref{spreading out morphisms}, after shrinking $T$, there exists a $T$-morphism 
\[
  f: \mathcal{Y} \to \PP^N_T
\]
whose base change to $\Spec(K(T))$ is isomorphic to $\iota$.
It follows from \cite[Proposition 12.93]{GW} that after shrinking $T$ again, we may assume that $f$ is a closed immersion.
After spreading out $H$, we conclude that $\mathcal{E}$ is also very ample.

We next assume that $\mathcal{E}_{\eta}$ is an ample $\R$-Cartier divisor.
Then we may write 
\[
  \mathcal{E}_{\eta} = \sum_{i=1}^r a_i H_i,
\]
where $a_i > 0$ is a real number and $H_i$ is a very ample Cartier divisor on $\mathcal{Y}_{\eta}$.
After shrinking $T$, we may find a Cartier divisor $\mathcal{H}_i$ on $\mathcal{Y}$ such that $\mathcal{H}_{\eta}=H_i$ and 
\[
  \mathcal{E} = \sum_{i=1}^r a_i \mathcal{H}_i.
\]
Then the ampleness of $\mathcal{E}$ follows from that of $\mathcal{H}_i$, which we already proved.

Finally, we assume that $\mathcal{E}_{\eta}$ is semiample.
Take a morphism $f : \mathcal{Y}_{\eta} \to Z$ to a projective $K(T)$-scheme $Z$ and an ample $\R$-Cartier divisor $H$ on $Z$ such that $\mathcal{E}_{\eta} \sim_{\R} f^*H$.
By Lemma \ref{spreading out schemes} and Lemma \ref{spreading out morphisms}, after shrinking $T$, there exists a $T$-morphism 
\[
  g: \mathcal{Y} \to \mathcal{Z}
\]
to a projective $T$-scheme $\mathcal{Z}$ whose generic fiber is isomorphic to $f$.
By spreading out $H$, we may find an $g$-ample $\R$-Cartier divisor $\mathcal{H}$ on $\mathcal{Z}$ such that $\mathcal{E} \sim_{\R} g^* \mathcal{H}$, which completes the proof.
\end{proof}

\begin{defn}
  Let $\mathcal{X},\mathcal{Y}, T$ be normal varieties over a field $k$ such that $\mathcal{X}$ and $\mathcal{Y}$ are projective over $T$, and $\Phi: \mathcal{X} \dashrightarrow \mathcal{Y}$ be a birational contraction over $T$.
  Suppose that $\Delta \ge 0$ is an effective $\R$-Weil divisor on $\mathcal{X}$ such that $K_{\mathcal{X}}+\Delta$ is $\R$-Cartier.
  \begin{enumerate}
  \item $\Phi$ is said to be a \emph{lc model} (resp.~\emph{weak lc model}) of $(\mathcal{X}, \Delta)$ if
  \begin{enumerate}
    \item $(\mathcal{X}, \Delta)$ is lc,
    \item $\Phi$ is $K_{\mathcal{X}}+\Delta$-non-positive, and
    \item $\Phi_*(K_{\mathcal{X}}+\Delta)=K_{\mathcal{Y}}+\Phi_*\Delta$ is an ample (resp.~nef) $\R$-Cartier divisor over $T$.
  \end{enumerate}
  \item $\Phi$ is said to be a \emph{log minimal model} of $(\mathcal{X}, \Delta)$ if
  \begin{enumerate}
    \item $(\mathcal{X}, \Delta)$ is dlt (see \cite[Definition 2.37]{KM} for definition) and $\mathcal{Y}$ is $\Q$-factorial,
    \item $\Phi$ is $K_{\mathcal{X}}+\Delta$-negative, and
    \item $\Phi_*(K_{\mathcal{X}}+\Delta)=K_{\mathcal{Y}}+\Phi_*\Delta$ is nef $\R$-Cartier divisor.
  \end{enumerate}
  \end{enumerate}
\end{defn}

A \emph{small $\Q$-factorization} of a normal variety $X$ is a small proper birational morphism $f : Y \to X$ from a normal $\Q$-factorial variety $Y$.
As is well-known, the existence of log resolutions and log minimal models implies that of small $\Q$-factorization.

\begin{lem}\label{Q-factorization}
Let $X$ be a $3$-dimensional normal quasi-projective variety over an $F$-finite field $k$ of characteristic $p>5$.
Assume that there exists an effective $\R$-Weil divisor $\Delta$ on $X$ such that $(X,\Delta)$ is klt.
Then $X$ admits a small $\Q$-factorization.
\end{lem}

\begin{proof}
Take a log resolution $f : Y \to X$ of $(X,\Delta)$.
By \cite[Theorem 1.6]{DW}, there exists a log minimal model $\phi: Y \dashrightarrow Z$ of $(Y, f^{-1}_*\Delta+ \Exc(f))$ over $X$.
It then follows from the negativity lemma that the induced morphism $Z \to X$ is small.
\end{proof}

\begin{lem}\label{rmk on wlc}
  Let $(X, \Delta)$ be a $3$-dimensional projective klt pair over an $F$-finite field $k$ of characteristic $p>5$ and $\phi : X \dashrightarrow Y$ be a birational contraction to a normal projective variety $Y$.
  We further assume that one of the following conditions holds.
  \begin{enumerate}[label=\textup{(\alph*)}]
    \item There is a big $\Q$-Cartier divisor $A$ such that $0 \le A \le \Delta$, or
    \item $\Delta$ is a $\Q$-Weil divisor, $\kappa(X,K_X+\Delta) \ge 1$ and $(X_{\overline{k}}, \Delta_{\overline{k}})$ is klt, where $\Delta_{\overline{k}}$ is the flat pullback of $\Delta$ to $X_{\overline{k}} : =X \times_k \Spec(\overline{k})$.
  \end{enumerate}
  Then $\phi$ is a weak lc model of $(X, \Delta)$ if and only if it is a semiample model of $K_X+\Delta$.
\end{lem}
  
\begin{proof}
The if part is obvious.
For the other direction, we assume that $\phi$ is a weak lc model.
We first assume (a).
Since $(Y, \phi_* \Delta)$ is klt, we can take a small $\Q$-factorization $\pi: \tilde{Y} \to Y$ by Lemma \ref{Q-factorization}.
After replacing $Y$ by $\tilde{Y}$, we may assume that $Y$ is $\Q$-factorial, and in particular, $\phi_*A$ is $\Q$-Cartier.
Therefore, the assertion follows from the base point free theorem (Theorem \ref{bpf} (a)).

We next assume (b).
Since $\phi$ is a birational contraction, $Y$ is geometrically regular over $k$ in codimension one, and in particular, the base change $Y_{\overline{k}}$ is a normal variety.
Let 
\[
  \phi_{\overline{k}} : X_{\overline{k}} \dashrightarrow Y_{\overline{k}}
\]
be the birational map induced by $\phi$.
By \cite[Proposition II.8.10]{Har}, the flat pullback $p^*K_Y$ of a canonical divisor $K_Y$ to $Y_{\overline{k}}$ is also a canonical divisor, where $p : Y_{\overline{k}} \to Y$ is the natural morphism.
Combining this with the equation 
\[
(\phi_{\overline{k}})_* \Delta_{\overline{k}} = p^* (\phi_*\Delta),
\]
the birational contraction $\phi_{\overline{k}}$ is a weak lc model of $(X_{\overline{k}}, \Delta_{\overline{k}})$.
Therefore, $(Y_{\overline{k}}, (\phi_{\overline{k}})_*\Delta_{\overline{k}})$ is klt and 
\[
  \kappa(Y_{\overline{k}}, K_{Y_{\overline{k}}} + (\phi_{\overline{k}})_*\Delta_{\overline{k}}) = \kappa(X_{\overline{k}}, K_{X_{\overline{k}}} + \Delta_{\overline{k}}) \ge 1.
\]
It then follows from Theorem \ref{bpf} (b) that 
\[
  K_{Y_{\overline{k}}} + (\phi_{\overline{k}})_*\Delta_{\overline{k}} =p^*(K_Y+\phi_*\Delta)
  \]
  is semiample.
  This shows that $K_Y+\phi_* \Delta$ is semiample.
\end{proof}

\section{Lower bounds for Seshadri constants}\label{sec3}

In this section, we give a lower bound for the Seshadri constant of a nef and big invertible sheaf at a very general closed point of a projective variety $X$ with $\dim X \le 3$.
The strategy is similar to that of \cite{EKL}.
Instead of \cite[Proposition 2.3]{EKL}, we use Proposition \ref{key lemma for lower bound dim3} below.
We first consider the case where $X$ is one-dimensional.

\begin{eg}\label{ses dim1}
Let $X$ be a projective variety of dimension one over an algebraically closed field $k$ and $L$ be an ample invertible sheaf.
Then for a regular closed point $x \in X$, we have
\[
\ses{L} =\frac{\deg_X(L)}{\mult_x(X)} = \deg_X(L) \ge 1.
\]
\end{eg}

\subsection{Two dimensional case}
We next consider the case where $X$ is two-dimensional.

\begin{defn}
Let $\Sigma$ be a set of pointed curves in a variety $X$ and $S \subseteq X$ be a closed subset.
We define the subset $\Sigma \cdot S$ of $X(k)$ by
\[
\Sigma \cdot S : = \bigcup_{(C,x) \in \Sigma, \ x \in S} C(k).
\]
\end{defn}

\begin{lem}\label{tower}
Let $X$ be a projective variety over an algebraically closed field $k$, $\Sigma$ be a set of pointed curves in $X$ and $L$ be an invertible sheaf such that $h^0(X,L) \ge 2$.
Assume that there exists a closed point $y \in X$ and a sequence of closed subsets of $X$
\[
\{y\} = Z_0 \subseteq Z_1 \subseteq \cdots \subseteq Z_n=X
\]
such that $Z_{i+1}$ is contained in the closure $\overline{\Sigma \cdot Z_i}$ of $\Sigma \cdot Z_i$ for every $i=0, \dots, n-1$.
Then there exists a pointed curve $(C, x) \in \Sigma$  such that
\[
\frac{(L \cdot C)}{\mult_{x}C} \ge 1.
\]
\end{lem}

\begin{proof}
By Lemma \ref{movable divisor}, there exists an element $E \in |L|$ such that $y \in \Supp(E)$.
Take the maximal number $i$ such that $Z_i \subseteq \Supp(E)$.
Since $\Sigma \cdot Z_{i} \not\subseteq \Supp(E)$, there exists a pointed curve $(C,x) \in \Sigma$ such that $x \in Z_i$ and $C \not\subseteq \Supp(E)$. 

It follows from the definition of the intersection number (\cite[Definition 2.4.2]{Ful}) that
\[
(L \cdot C) = \sum_{y \in D \cap C} \dim_k (\sO_{E \cap C, y}) \ge \dim_k(\sO_{E \cap C,x}) \ge \mult_x(C),
\]
where the last inequality follows from \cite[Theorem 14.9]{Mat}.
\end{proof}  

\begin{prop}\label{ses dim2}
Let $X$ be a $2$-dimensional projective variety of over an uncountable algebraically closed field $k$ and $L$ be a nef and big invertible sheaf such that $h^0(X,L) \ge 2$.
Then for a very general closed point $x \in X$, we have
\[
\ses{L} \ge 1.
\]
\end{prop}

\begin{proof}
Let $\Sigma$ be the set of all pointed curves $(C, x)$ in $X$ such that
\[
\frac{(L \cdot C)}{\mult_{x}C} < 1.
\]
If the assertion of the proposition fails, then it follows from \cite[Paragraph (3.3)]{EKL} that the locus $\{y \in X(k) \mid \exists (C, y) \in \Sigma\} \subseteq X$ contains an open dense subset $X^{\circ}$ of $X(k)$.
After shrinking $X^{\circ}$, we may assume that $X^{\circ} \cap \B_+(L) = \emptyset$.

Take a closed point $y \in X^{\circ}$ and a curve $C \subseteq X$ such that $(C,y) \in \Sigma$.
Let $z \in C$ be a general closed point.
We may assume that $z \in X^{\circ}$ and therefore, there exists a curve $\Gamma$ such that $(\Gamma, z) \in \Sigma$.
On the other hand it follows from Lemma \ref{B+ big} and Example \ref{ses dim1} that $\Gamma \neq C$.
Since $z$ is general, we conclude that 
\[
\overline{\Sigma \cdot C}=X.
\]
It then follows from Lemma \ref{tower} that there exists a pointed curve $(C, x) \in \Sigma$  such that
\[
\frac{(L \cdot C)}{\mult_{x}C} \ge 1,
\]
which is a contradiction to the choice of $\Sigma$.
\end{proof}

\subsection{Three dimensional case}
We consider three-dimensional case and prove Theorem \ref{C}.

\begin{defn}
Let $X$ and $T$ be varieties over an algebraically closed field $k$ and $\Sigma$ be a set of pointed curves in $X$.
Then for any subset $\mathcal{Z} \subseteq X \times_k T$, we define the subset $\Sigma \cdot \mathcal{Z} \subseteq X(k) \times T(k)$ by
\[
\Sigma \cdot \mathcal{Z} : = \{ (x,t) \in X(k) \times T(k) \mid x \in \Sigma \cdot \mathcal{Z}_t \}.
\] 
\end{defn}

\begin{lem}[\textup{cf.~\cite[Lemma 3.5.1]{EKL}}]\label{constructible}
Let $X, T$ be varieties over an algebraically closed field $k$ and $\mathcal{Z} \subseteq X \times_k T$ be a closed subset.
Suppose that $\Sigma$ is a set of pointed curves in $X$ parametrized by a variety $B$, that is, there is a closed subset $\mathcal{C} \subseteq X \times_k B$ and a $k$-morphism $g: B \to X$ such that
\[
  \Sigma = \{(\mathcal{C}_b, g(b))\}_{b \in B(k)}.
\] 
Then $\Sigma \cdot \mathcal{Z}$ is a constructible subset of $(X \times_k T)(k)$.
\end{lem}

\begin{proof}
We define the map $G : \mathcal{C} \times_k T \to X \times_k T$ by the composite 
\[
  \mathcal{C} \times_k T \hookrightarrow (X \times_k B) \times_k T \xrightarrow{p_{2,3}} B \times_k T \xrightarrow{g \times \id} X \times_k T.
\] 
Then $\Sigma \cdot \mathcal{Z}$ is the set of closed points of the image of $G^{-1}(\mathcal{Z})$ by the morphism 
\[
  F: \mathcal{C} \times_k T \hookrightarrow (X \times_k B) \times_k T \xrightarrow{p_{1,3}} X \times_k T.
\]
Therefore, the assertion follows from Chevalley's theorem.
\end{proof}

\begin{lem}[\textup{cf.~\cite[Lemma 3.6.1]{EKL}}]\label{swipe out}
Let $X$ be a variety over an algebraically closed field $k$ and $\Sigma$ be a set of pointed curves in $X$ parametrized by a variety.
We further assume that the subset $\{x \in X(k) \mid \exists (C,x) \in \Sigma \} \subseteq X(k)$ contains an open dense subset $X^{\circ}$ of $X(k)$.
Then there exist a dominant $k$-morphism $f: T \to X$ from a variety $T$ and a sequence
\[
  \mathcal{Z}_0 \subseteq \mathcal{Z}_1 \subseteq \dots \subseteq \mathcal{Z}_m
\]
of irreducible closed subsets of $X \times_k T$ which satisfies the following properties:
\begin{enumerate}[label=\textup{(\roman*)}]
\item $\mathcal{Z}_0$ is the graph of $f$.
\item Every closed fiber $(\mathcal{Z}_i)_t$ is irreducible for every $i=0,1, \dots, m$.
\item $(\mathcal{Z}_{i+1})_t$ is contained in the closure $\overline{\Sigma \cdot (\mathcal{Z}_i)_t}$ of $\Sigma \cdot (\mathcal{Z}_i)_t$ for every $t \in T(k)$ and every $i=0,1, \dots, m-1$.
\item The subset $Q(t) : = \{x \in (\mathcal{Z}_m)_t(k) \mid \Sigma \cdot \{x\} \subseteq (\mathcal{Z}_m)_t\} $ contains an open dense subset of $(\mathcal{Z}_m)_t(k)$ for every $t \in T(k)$.
\end{enumerate}
\end{lem}

\begin{proof}
By induction on $n \ge 0$, we will construct, until the algorithm terminates, a dominant morphism $f_n : T_n \to X$ from a variety $T_n$ and a sequence
\[
  \mathcal{Z}_0^{(n)} \subseteq \mathcal{Z}_1^{(n)} \subseteq \dots \subseteq \mathcal{Z}_n^{(n)}
\]
of closed subsets of $X \times_k T_n$ which satisfies the following properties:
\begin{enumerate}[label=\textup{(\roman*)$_n$}]
\item $\mathcal{Z}_0^{(n)}$ is the graph of $f_n$.
\item A general closed fiber $(\mathcal{Z}_i^{(n)})_t$ is irreducible for every $i=0,1, \dots, n$.
\item $(\mathcal{Z}_{i+1}^{(n)})_t$ is contained in $\overline{\Sigma \cdot (\mathcal{Z}_i^{(n)})_t}$ for general $t \in T(k)$ and every $i=0,1, \dots, n-1$.
\end{enumerate}

\textbf{Step.0}:
When $n=0$, then we set $T_0 : = X$ and $f_0 : = \id_X$.
Then the diagonal $\mathcal{Z}_0^{(0)}$ satisfies the assertion.

\textbf{Step.1}:
We assume that we already constructed $f_n : T_n \to X$ and $\mathcal{Z}_i^{(n)}$ as above.
We first show that after replacing $T_n$, we may also assume that a general fiber of every irreducible component of 
\[
  \mathcal{W} : = \overline{\Sigma \cdot \mathcal{Z}_n^{(n)}}.
\]
over $T_n$ is irreducible.
In order to do this, take a flat dominant morphism $g: T_{n}' \to T_n$ from a variety $T_{n}'$ such that the generic fiber of every irreducible component of $\mathcal{W}' : = \mathcal{W} \times_{T_n} T_{n}'$ over $T_{n}'$ is geometrically irreducible.
We note that the dominant morphism $f_{n}' :  = f_n \circ g : T_{n}' \to X$ and the sequence 
\[
  \{ (\mathcal{Z}_i^{(n)})' : = \mathcal{Z}_i^{(n)} \times_{T_n} T_n'\}_{i=0}^n
\] 
also satisfy the properties (i)$_n$, (ii)$_n$ and (iii)$_n$.
Moreover, we have 
\[
  \Sigma \cdot (\mathcal{Z}_n^{(n)})' =((\id \times g)(k))^{-1} ( \Sigma \cdot \mathcal{Z}_n^{(n)}).\]
Since $(\id \times g) (k)$ is an open map, the closure of $\Sigma \cdot (\mathcal{Z}_n^{(n)})'$ in $(X \times T_n')(k)$ is equal to the pullback of $\mathcal{W}(k) = \overline{\Sigma \cdot \mathcal{Z}_n^{(n)}}$ by $(\id \times g)(k)$.
After taking closure of them in $X \times_k T_n'$, we have
  \[
    \overline{\Sigma \cdot (\mathcal{Z}_n^{(n)})'} =\mathcal{W}' .
  \]
  
Therefore, after replacing $T_n$ by $T_n'$, $\mathcal{Z}_i^{(n)}$ by $(\mathcal{Z}_i^{(n)})'$ and $\mathcal{W}$ by $\mathcal{W}'$, we may assume that a general fiber of every irreducible component of $\mathcal{W}$ is irreducible, as desired.

\textbf{Step.2}:
After shrinking $T_n$, we also assume that $\mathcal{Z}_i^{(n)}$ is flat over $T_n$ and that every closed fiber of $\mathcal{Z}_i^{(n)}$ over $T_n$ is irreducible for every $i$.
In particular, $\mathcal{Z}_i^{(n)}$ is irreducible.
We set $\mathcal{Z} : = \mathcal{Z}_n^{(n)}$.
Noting that the open dense subset $(X^{\circ} \times T_n(k)) \cap \mathcal{Z}$ of $\mathcal{Z}(k)$ is contained in $\Sigma \cdot \mathcal{Z}$, we have the inclusion
\[
  \mathcal{Z} \subseteq \mathcal{W}  = \overline{\Sigma \cdot \mathcal{Z}}.
\]

\textbf{Case.1}:
We first consider the case where $\mathcal{Z}$ is an irreducible component of $\mathcal{W}$.
In this case, we can take an open subset $\mathcal{W}^{\circ}$ of $\mathcal{W}$ such that 
\[
  \mathcal{W}^{\circ} \subseteq \mathcal{Z}.
\]
Take a general closed point $t \in T_n(k)$ so that $\mathcal{W}^{\circ}_t$ is non-empty.
We also take any element $(C,x) \in \Sigma$ with $x \in \mathcal{W}^{\circ}_t$.
Since we have 
\[
  C \subseteq \Sigma \cdot \mathcal{Z}_t \subseteq \mathcal{W}_t,
\]
the subset $C \cap \mathcal{W}_t^{\circ} \subseteq C$ is open, which shows that
\[
  C =\overline{C \cap \mathcal{W}_t^{\circ} } \subseteq \overline{\mathcal{W}_t^{\circ}} \subseteq \mathcal{Z}_t.
\]
Therefore, the set $Q(t)=\{x \in \mathcal{Z}_t(k) \mid \Sigma \cdot \{x\} \subseteq \mathcal{Z}_t\}$ contains the open dense subset $\mathcal{W}^{\circ}_t(k)$.
We finish the algorithm and the assertion of this lemma follows by setting $m : =n$, $T: = T_n$, $f : = f_n$ and $\mathcal{Z}_i : = \mathcal{Z}_i^{(n)}$.

\textbf{Case.2}:
We assume that $\mathcal{Z}$ is not an irreducible component of $\mathcal{W}$.
Since $\Sigma \cdot \mathcal{Z}$ is a constructible subset by Lemma \ref{constructible}, we have
\[
  \mathcal{W}_t= \overline{\Sigma \cdot (\mathcal{Z}_t)}
\]
for general $t \in T(k)$.
In this case, we set $T_{n+1} : = T_n$, $f_{n+1} : = f_n$ and $\mathcal{Z}_i^{(n+1)} : = \mathcal{Z}_i^{(n)}$ for $i=0,1,\dots, n$.
We also take an irreducible component $\mathcal{Z}_{n+1}^{(n+1)}$ of $\mathcal{W}$ which contains $\mathcal{Z}$.
Then the sequence
\[
\mathcal{Z}_0^{(n+1)} \subseteq \mathcal{Z}_1^{(n+1)} \subseteq \dots \subseteq \mathcal{Z}_{n+1}^{(n+1)}
\]
satisfies the properties (i)$_{n+1}$, (ii)$_{n+1}$ and (iii)$_{n+1}$.
Then we go back to Step.1 and start anew.
\end{proof}

The following result is a key ingredient in the proof of Theorem \ref{C}.

\begin{prop}\label{key lemma for lower bound dim3}
  Let $X$ be a smooth projective variety over an algebraically closed field $k$, $\Gamma \subseteq X$ be a prime divisor and $N$ be an invertible sheaf on $X$.
  We assume that the following conditions are satisfied:
  \begin{enumerate}[label=\textup{(\roman*)}]
  \item $h^0(X, \sO_X(\Gamma)) \ge 2$.
  \item There exists a nef and big invertible sheaf $H$ on $X$ with
  \[
  h^0(X, N) > (H^{\dim X-1} \cdot N) +1.
  \]
    \end{enumerate}
  Then we have
  \[
    h^0(\Gamma, N|_{\Gamma}) \ge 2.
    \]
  \end{prop}
  
  \begin{proof}
  Assume to the contrary that we have $h^0(\Gamma, N|_{\Gamma}) \le 1$.
  We write 
  \[\PP : = \PP(H^0(X, \sO_X(\Gamma))^*)
  \] 
  and let $\mathcal{Z} : = \mathcal{Z}_{\Gamma} \subseteq X \times_k \PP$ be the effective divisor as in Lemma \ref{universal family}.

  \textbf{Step.1}
  Since the second projection $p_2: \mathcal{Z} \to \PP$ is flat and proper, and some closed fiber is isomorphic to $\Gamma$, it follows from the semicontinuity \cite[Theorem III.12.8]{Har} that there is an open dense subset $\PP^{\circ} \subseteq \PP$ such that we have
  \[
  h^0(\mathcal{Z}_t , N|_{\mathcal{Z}_t}) \le 1
  \]
  for every closed point $t \in \PP^{\circ}$.
  Similarly, it follows from \cite[(12.2.1)]{EGA} that after shrinking $\PP^{\circ}$, the fiber of $\mathcal{Z}_t$ is an integral scheme for every $t \in \PP^{\circ}$.
  Combining this with the flatness of $p_2$, we see that $\mathcal{Z}$ is irreducible.
  
  Since $p_2^{-1}(\PP^\circ)$ dominates $X$ by Lemma \ref{movable divisor}, the subset
  \[
  \mathcal{Z}^{\circ} : = p_2^{-1}(\PP^{\circ}) \cap p_1^{-1}(X \setminus \Bs(|N|))
  \]
  is non-empty.
  We will prove the following claim:
  \begin{claim}
  For any closed point $(x, t) \in \mathcal{Z}^{\circ}$, the prime divisor $E: = \mathcal{Z}_t \in |\Gamma|$ corresponding to $t$ satisfies the following property:
  \begin{equation}\label{key trick}
    \textup{For any divisor $D \in |N|$, if $x \in \Supp(D)$, then $E \le D$}.
  \end{equation}
  \end{claim}
  
  \begin{proof}
  Assume to the contrary that there exist a closed point $(x, t) \in \mathcal{Z}^{\circ}$ and $D \in |N|$ such that $x \in \Supp(D)$ and $E : = \mathcal{Z}_t \not\subseteq \Supp(D)$.
  Then we obtain an effective divisor $D|_{E} \in |N|_{E}|$ whose support contains $x$.
  
  On the other hand, since $x \not\in \Bs(|N|)$, there exists an effective divisor $D' \in |N|$ such that $x \not\in \Supp(D')$.
  This defines an effective divisor $D'|_{E} \in |N|_{E}|$ with $x \not\in \Supp(D'|_E)$.
  Therefore, we have $\#|N|_{E}| \ge 2$, which is a contradiction to Lemma \ref{movable divisor}.
  \end{proof}
  
  \textbf{Step.2}
  In this step, we show that $\dim(\PP)=1$.
  Assume to the contrary that $\dim(\PP) \ge 2$.
  Take a general closed point $x \in X$ and a divisor $D \in |N|$ whose support contains $x$.
  Since $\mathcal{Z}^{\circ}$ dominates $X$, if $x$ is sufficiently general, then the dimension of the fiber $\mathcal{Z}^{\circ}_x$ at $x \in X$ is equal to 
  \[
  \dim(\mathcal{Z}^{\circ}) - \dim (X) =\dim(\PP)-1 \ge 1.
  \]
  Combining this with the claim in Step.1, there exist infinitely many prime divisors $E \in |\Gamma|$ such that $E$ is contained in $\Supp(D)$, which is a contradiction.
  
  \textbf{Step.3}
  Let $m \ge 0$ be the maximal integer such that the addition map
  \[
  |N(-m\Gamma)| \times |\Gamma|^m \to |N| \ ; \ (F, \Gamma_1, \dots, \Gamma_m) \mapsto F+\Gamma_1+ \cdots + \Gamma_m
  \]
  is surjective.
  In this step, we will show that $\#|N(-m\Gamma)|=1$.
  Assume to the contrary, the effective divisor $\mathcal{Z}_{N(-m\Gamma)} \subseteq X \times_k T$ dominates $X$ by Lemma \ref{movable divisor}, where we write $T : = \PP(H^0(X, \sO_X(N(-m\Gamma)))^*)$.
    
  By the maximality of $m$, the addition map
  \[
  |N(-(m+1)\Gamma)| \times |\Gamma| \to |N(-m\Gamma)|
  \]
  is not surjective.
  Combining this with Lemma \ref{universal family} (2), there exists an open dense subset $T^{\circ} \subseteq T$ such that for every closed point $t \in T^{\circ}$, the corresponding divisor $F = (\mathcal{Z}_{N(-m\Gamma)})_t$ satisfies the following condition:
  \begin{equation}\label{maximality of m}
  F - E \not\ge 0 \ \textup{for any divisor $E \in |\Gamma|$}
  \end{equation}
  
  Since $\mathcal{Z}_{N(-m\Gamma)}$ is flat over $T$, the pullback $p_2^{-1}(T^{\circ})$ is open dense in $\mathcal{Z}_{N(-m\Gamma)}$.
  Therefore, $p_1(p_2^{-1}(T^{\circ}))$ is dense in $X$, that is, for a general closed point $x \in X$, there exists a divisor $F \in |N(-m \Gamma)|$ with $x \in \Supp(F)$ and satisfying the condition \eqref{maximality of m}.
  
  Since $x$ is general, we may take a prime divisor $E \in |\Gamma|$ with $E \neq \Gamma$ which satisfies the condition \eqref{key trick} in Step.1.
  Noting that $F+m \Gamma \in |N|$, it follows from \eqref{key trick} that we have
  \[
  E \le F+m\Gamma,
  \]
  which is a contradiction to the condition \eqref{maximality of m}.
  
  \textbf{Step.4}
  In this step, we will prove the assertion in the proposition.
  We first observe that one has a surjective $k$-morphism
  \[
  (\PP^1)^m \onto \PP(H^0(X,N)^*)
  \]
  by Step.2, Step.3 and Lemma \ref{universal family} (2).
  Therefore, we have $\dim(\PP(H^0(X,N)^*)) \le m$.
  
  Noting that the linear system $|N(-m\Gamma)|$ is non-empty, there is an effective divisor $F \in |N(-m\Gamma)|$.
  Since $N \sim F+m\Gamma$, we have the equation
  \[
  (H^{\dim X-1} \cdot N) =(H^{\dim X-1} \cdot F) + m (H^{\dim X-1} \cdot \Gamma).
  \]
  It follows from Kleiman's theorem (\cite[Theorem 1.4.9]{Laz}) that $(H^{\dim X-1} \cdot F)$ is non-negative.
  On the other hand, by Lemma \ref{intersection with nef big}, we have $(H^{\dim X-1} \cdot \Gamma) \ge 1$.
  Therefore, we conclude that 
  \[
  (H^{\dim X-1} \cdot N) \ge m \ge \dim(\PP(H^0(X,N)^*))=h^0(X,N) -1,
  \] which is a contradiction.
  \end{proof}

\begin{thm}[Theorem \ref{C}]\label{Lower bound for seshadri}
  Let $X$ be a $3$-dimensional normal projective variety  with rational singularities over an uncountable algebraically closed field, $\ell \ge 1$ be an integer and $L$ be a nef and big invertible sheaf on $X$.
  We further assume that the following conditions are satisfied:
  \begin{enumerate}[label=\textup{(\roman*)}]
  \item $H^1(X, \sO_X)=0$.
  \item $h^0(X,L^{\ell}) > \ell \vol(L)+1$.
  \end{enumerate}
  Then we have
  \[
    \ses{L} \ge \frac{1}{\ell}
    \]
  for a very general closed point $x \in X$.
\end{thm}

\begin{proof}
After replacing $X$ by a resolution of singularities, we may assume that $X$ is smooth.
Assume to the contrary, it follows from \cite[Paragraph (3,3)]{EKL} that there exists a set $\Sigma$ of pointed curves in $X$ which is parametrized by a variety such that $\{x \in X \mid \exists (C,x) \in \Sigma\}$ contains an open dense subset $X^{\circ}$ of $X(k)$ and that for every $(C,x) \in \Sigma$, we have
  \begin{align}\label{choice of Sigma}
    \frac{(L \cdot C)}{\mult_x(C)} < \frac{1}{\ell}.
  \end{align}
  Take a dominant $k$-morphism $f: T \to X$ from a variety $T$ and a sequence
  \[
    \mathcal{Z}_0 \subseteq \mathcal{Z}_1 \subseteq \dots \subseteq \mathcal{Z}_m \subseteq X \times_k T
  \]
  as in Lemma \ref{swipe out}.
  
  \textbf{Case.1}:
  We first consider the case where $\mathcal{Z}_m=X \times_k T$.
  In this case, after restricting to a general closed fiber, we obtain a sequence of closed subsets of $X$ as in Lemma \ref{tower}.
  Noting that we have $2 \le h^0(X, L^{\ell})$, it follows from the lemma that we may find a pointed curve $(C, x) \in \Sigma$ with 
  \[
    \frac{(L^{\ell} \cdot C)}{\mult_x(C)} \ge 1,
  \]
  which is a contradiction to the inequality \eqref{choice of Sigma}.
  
  \textbf{Case.2}:
  We next consider the case where a general closed fiber $S : = (\mathcal{Z}_m)_t$ of $\mathcal{Z}_m$ is one-dimensional.
  By the property (iv) of Lemma \ref{swipe out}, for a general closed point $x \in S \cap X^{\circ}$, we can take an element $(C,x) \in \Sigma$ such that $C \subseteq S$.
  Combining this with the inequality \eqref{choice of Sigma}, we have 
  \[
    \ses{L|_S} \le \frac{(L|_S \cdot C)}{\mult_x(C)} <\frac{1}{\ell} \le 1.
  \]
  Since $t$ is general, we may also assume that one has $S \not\subseteq \B_+(L)$ and in particular, $L|_S$ is big by Lemma \ref{B+ big}.
  This is a contradiction to Example \ref{ses dim1}.
  
  \textbf{Case.3}:
  We finally consider the case where a general closed fiber $S : = (\mathcal{Z}_m)_t$ of $\mathcal{Z}_m$ is two-dimensional.
  As in Case.2, we may assume that $L|_S$ is big and one has
  \[
  \ses{L^{\ell}|_S} \le \frac{(L^{\ell}|_S \cdot C)}{\mult_x(C)} <1
  \]
  for general $x \in S$.
  It then follows from Proposition \ref{ses dim2} that 
  \[
    h^0(S, L^{\ell}|_S) \le 1.
  \]
  
  On the other hand, since $\bigcup_{s \in T(k)} (\mathcal{Z}_m)_s (k)$ is dense in $X(k)$, if $t' \in T(k)$ is another general point, then $S' : = (\mathcal{Z}_m)_{t'}$ is a prime divisor of $X$ with $S' \neq S$.
  It follows from the vanishing $H^1(X, \sO_X)=0$ that the algebraically equivalent divisors $S$ and $S'$ are linearly equivalent (cf.~\cite[Excercise III.12.6]{Har}).
  In particular, we have 
  \[
    h^0(X, \sO_X(S)) \ge 2.
  \]
  This is a contradiction by Proposition \ref{key lemma for lower bound dim3} with setting $N := L^{\ell}$ and $H : =L$.
  \end{proof}

We next give sufficient conditions for the assumption (ii) in Theorem \ref{Lower bound for seshadri}.
A \emph{numerical polynomial} is a polynomial $P(z) \in \Q[z]$ such that $P(n) \in \Z$ for all $n \in \Z$.

\begin{lem}\label{sansuu}
Let $d>0$ be an integer, $c \in \R$ be a real number and $g(z) \in \Q[z]$ be a polynomial with degree smaller than $d$.
Then there exists an integer $M>0$ such that there is no numerical polynomial $P(z) \in \Q[z]$ which satisfies the following properties:
\begin{enumerate}[label=\textup{(\roman*)}]
\item The degree of $P(z)$ is $d$.
\item The leading coefficient of $P(z)$ is positive.
\item $c< P(n)$ for every $n=0, 1, \dots, d$.
\item $P(n) \le g(n)$ for all integers $0 \le n \le M$.
\end{enumerate}
\end{lem}

\begin{proof}
For an integer $M>0$, we denote by $\Theta_M$ the set of all numerical polynomials $P(z)$ which satisfy the conditions (i), (ii), (iii) and (iv).
Since a numerical polynomial $P(z) \in \Q[z]$ with degree $d$ is uniquely determined by the tuple $(P(0),P(1), \dots, P(d)) \in \Z^{d+1}$, it follows from the conditions (iii) and (iv) that the set $\Theta_d$ is a finite set.
On the other hand, noting that we have $\deg(g)<d$, for any polynomial $P(z)$ of degree $d$ with positive leading coefficient, there exists an integer $m_P >0$ such that $P(m_P)>g(m_P)$.
We put 
\[
M : = \max \{d\} \cup \{m_P \mid P \in \Theta_d\}.
\]
Then we can see $\Theta_M=\emptyset$, as desired.
\end{proof}

\begin{cor}\label{lower bound for seshadri vanishing}
  Let $V,A>0$ be integers.
  Then there exists a real number $\delta =\delta(V, A)>0$ with the following property:
  For any $3$-dimensional normal projective variety $X$ with rational singularities over an uncountable algebraically closed field $k$ and any nef and big invertible sheaf $L$ on $X$, we have 
  \[
    \ses{L} \ge \delta
  \] for a very general closed point $x \in X$ if the following conditions are satisfied:
  \begin{enumerate}[label=\textup{(\roman*)}]
  \item $\vol(L) < V$.
  \item $H^1(X, \sO_X)=0$.
  \item $H^2(X, L^i)=0$ for every integer $i \ge 0$.
  \item $h^1(X,L^i) +h^3(X, L^i) <A$ for $i=0, 1,2,3$.
  \end{enumerate}
  \end{cor}

    \begin{proof}
    By Lemma \ref{sansuu}, there exists an integer $M=M(V,A)$ such that for any numerical polynomial $P(z)$ of degree $3$ with positive leading coefficient, if we have $-A<P(n)$ for $n=0,1,2,3$, then one has $V\ell+1 <P(\ell)$ for some $0 \le \ell \le M$.
    
    Fix $X$ and $L$ as in the statement of the corollary.
    Let $P(n) = \chi(L^n)$ be the Hilbert polynomial.
    It follows from the assumption (iv) that we have
    \[
      P(n) \ge -h^1(X, L^n) - h^3(X, L^n) > -A
    \]
    for $n=0,1,2,3$.
    By the definition of $M$, we may find an integer $0 \le \ell \le M$ such that one has 
    \begin{align}\label{dim3 key ineq}
      \ell \vol(L) +1 < V\ell +1 < P(\ell) \le h^0(X, L^{\ell}) + h^2(X, L^{\ell}) = h^0(X, L^{\ell}).
    \end{align}
    We note that $\ell \neq 0$ by this inequality.
    If we set $\delta : = 1/M$, then it follows from Theorem \ref{Lower bound for seshadri} that we have
    \[
      \delta \le \frac{1}{\ell} \le \ses{L}
    \]
    for a very general closed point $x \in X$.
    \end{proof}
    
    \begin{cor}\label{lower bound for seshadri weak Fano}
      Let $r, V,A>0$ be integers.
      Then there exists a real number $\delta =\delta(r, V, A)>0$ with the following property:
      For any $3$-dimensional normal projective variety $X$ with rational singularities over an uncountable algebraically closed field $k$, we have 
      \[
        \ses{-K_X} \ge \delta
      \] for a very general closed point $x \in X$ if the following conditions are satisfied:
      \begin{enumerate}[label=\textup{(\roman*)}]
      \item $-rK_X$ is a nef and big Cartier divisor.
      \item $\vol(-K_X) < V$.
      \item $H^1(X, \sO_X(irK_X))=0$ for every integer $i \ge 0$.
      \item $h^2(X,\sO_X(irK_X)) <A$ for $i=0, 1,2,3$.
    \end{enumerate}
    \end{cor}
    
\begin{proof}
For a variety $X$ as in the statement, we set $L_X : = \sO_X(-rK_X)$ and consider the numerical polynomial 
\[
  P_X (n) : = \chi(\omega_X \otimes L_X^{n})
\]
of degree $3$ with a positive leading coefficient.
By the Serre duality, we have
\[
  P_X(i) \ge -h^0(X, L_X^{-i}) - h^2(X, L_X^{-i}) \ge -1 - h^2(X, L_X^{-i})> -1 -A
\]
for $i=0,1,2,3$.
Similarly, by the assumption (iv), we have
\[
  P_X(i) \le h^0(X, \omega_X \otimes L_X^i) + h^2(X, \omega_X \otimes L_X^i) = h^0(X, \omega_X \otimes L_X^i)
\]
for every $i \ge 0$.

It then follows from Lemma \ref{sansuu} that there exists an integer $M= M(r,V,A)>0$ such that for any $X$ as in the statement, there exists an integer $0 \le m \le M$ with the following property:
If we set $\ell : = rm -1$, then one has
\[
  h^0(X, \sO_X(-\ell K_X)) = h^0(X, \omega_X \otimes L_X^m) > \ell r^3V +1.
\]
Therefore, we have
\begin{align*}
  h^0(X, L_X^{\ell}) =h^0(X, \sO_X( r (-\ell K_X))) &\ge h^0(X, \sO_X(-\ell K_X)) \\
  & > \ell r^3V +1 \\
  & > \ell \vol(L_X) +1.
\end{align*}
It then follows from Theorem \ref{Lower bound for seshadri} that
\[
  \ses{L_X} \ge \frac{1}{\ell} \ge \frac{1}{rM-1}
\]
for a very general point $x \in X$, as desired.
\end{proof}

\section{Birationally boundedness}\label{sec4}

In this section, we will prove Theorem \ref{B}.

\subsection{Birationality of the adjoint linear system}

In this subsection, we give a criterion in terms of Seshadri constants for the adjoint divisor $K_X+D$ to define a birational map.
When $D$ is ample, this is proved in \cite[Theorem 3.1]{MS}.
We extend this to the case where $D$ is nef and big.

\begin{lem}[\textup{cf. \cite[Remark 2.10]{MS}}]\label{MS 2.10}
Let $D$ be a nef and big Cartier divisor on a projective variety $X$ over an algebraically closed field $k$, $x \in X \setminus \B_+(D)$ be a closed point and $\m_x \subseteq \sO_X$ be the maximal ideal at $x$.
Then there exists an integer $m>0$ such that for any integers $n, \ell>0$, if $nD$ separates $\ell$-jets at $x$, then 
\[
\m_x^{\ell} \otimes \sO_X( (m+n)D)
\]
is globally generated around $X \setminus \B_+(D)$.
\end{lem}

\begin{proof}
Take an ample $\Q$-Cartier divisor $A$.
By \cite[Proposition 1.5]{ELMNP},  after replacing $A$, we have
\[
\B_+(D) = \B(D-A).
\]
Take an integer $N>0$ such that $NA$ is very ample and $\B(D-A)=\Bs(|nN(D-A)|)$ for every integer $n>0$.
By Fujita's vanishing theorem (cf. ~\cite[Theorem 1.4.35]{Laz}),  after replacing $N$ by its multiple,  we may assume that for any integers $i, u,v>0$, one has
\[
H^i(X, \sO_X(uD+vNA))=0.
\]

We set $m : =(\dim X+ 1)N$.
Take integers $n, \ell>0$ such that $nD$ separates $\ell$-jets at $x$ and we will show that $\m_x^{\ell} \otimes \sO_X((m+n)D)$ is globally generated around $X \setminus \B_+(D)$.
Since we have
\[
\m_x^{\ell} \otimes \sO_X((m+n)D) \cong (\m_x^{\ell} \otimes \sO_X(nD+mA)) \otimes \sO_X(m(D -A)) 
\]
and $\B_+(D)=\Bs(|m(D-A)|)$, it suffices to show that $\m_x^{\ell} \otimes \sO_X(nD+mA)$ is globally generated.

Let $v>0$ be a positive integer and consider the exact sequence
\[
0 \to \m_x^{\ell} \otimes \sO_X(nD + vNA) \to \sO_X(nD +vNA) \xrightarrow{\phi} \sO_X/\m_x^{\ell} \to 0.
\]
Since $nD$ separates $\ell$-jets at $x$ and $A$ is globally generated at $x$,  the morphism 
\[
H^0(X, \phi) : H^0(X, \sO_X(nD+vNA)) \to \sO_X/\m_x^{\ell}
\] 
is surjective.
Combining this with the Fujita's vanishing,  we conclude that 
\[
H^i(X, \m_x^{\ell} \otimes \sO_X(nD+vNA))=0
\]
for every $v, i>0$.
Therefore,  $\m_x^{\ell} \otimes \sO_X(nD+mA)$ is $0$-regular with respect to $NA$ (cf.~\cite[Section 1.8]{Laz}) and in particular, it is globally generated, as desired.
\end{proof}

\begin{prop}[\textup{\cite[Proposition 2.16]{MS}}]\label{MS 2.16}
Let $D_1, \dots, D_r$ be nef and big $\Q$-Cartier divisors on a projective variety over an algebraically closed field $k$ and $Z =\{x_1, \dots, x_r \} \subseteq X$ be a set of closed points.
Assume that $\B_+(D_i) \cap Z = \emptyset$ for every $i$.
Then we have
\[
\ejet[Z]{D_1+\cdots+D_r} \ge \min_{i} \ejet[x_i]{D_i}.
\]
\end{prop}

\begin{proof}
The proof is similar to that of \cite[Proposition 2.16]{MS}.

\textbf{Step.1}
Noting that $\B(D_i) \subseteq \B_+(D_i)$, by Lemma \ref{ejet basic} (3), we may assume that $D_i$ is an effective Cartier divisor with $\Supp(D_i) \cap Z=\emptyset$ for every $i$.
Fix real numbers $0<\epsilon$ and $0 < \alpha < \min_i \ejet[x_i]{D_i}$.
For every $i$, there exist integers $n_i, \ell_i>0$ such that $\alpha<\frac{\ell_i}{n_i}$ and that $n_iD_i$ separates $\ell_i$-jets at $x_i$.
By Lemma \ref{ejet basic},  after replacing $n_i$ and $\ell_i$ by its multiple,  we may assume that $\ell_i=\ell$ for all $i$ and $1/\ell < \epsilon$.

On the other hand,  it follows from Lemma \ref{MS 2.10} that for every $i$, there exists an integer $m_i>0$ such that the sheaf $\m_{x_i}^{\ell} \otimes \sO_X( (m_i+n_i)D_i)$ is globally generated around $Z$.
Moreover, since $\sO_X(D_i)$ is globally generated around $Z$, for any integer $M \ge m_i+n_i$,  the sheaf $\m_{x_i}^{\ell} \otimes \sO_X( MD_i)$ is globally generated around $Z$.
We set $n : = \max_{i} n_i$, $m : = \max_{i} m_i$.
Then the sheaf
\[
\m_{x_i}^{\ell} \otimes \sO_X( (m+n)D_i)
\]
is globally generated around $Z$.

\textbf{Step.2}
We set $D : = D_1 + \cdots + D_r$ and we will show that the restriction morphism
\[
\Phi: H^0(X, \sO_X( (m+n)D)) \to \sO_X/I_Z^{\ell}.
\]
is surjective.
It suffices to show that for every $i$ and for every element $a \in \sO_X/\m_{x_i}^{\ell}$, the element
\[
(0, \dots , 0, a, 0 ,  \dots, 0 ) \in \bigoplus_{i=1}^r \sO_X/\m_{x_i}^{\ell} \cong \sO_X/I_{Z}^{\ell}
\]
is contained in $\Im(\Phi)$.
For simplicity, we only consider the case of $i=1$.
By the Step. 1,  for every $i \neq 1$,  there exists a global section 
\[
  s_i \in H^0(X, \m_{x_i}^{\ell} \otimes \sO_X((m+n)D_i)) \subseteq H^0(X, \sO_X((m+n)D_i))
  \] whose stalk $(s_i)_{x_1}$ at $x_1$ is not contained in $\m_{x_1} \sO_X((m+n)D_i)_{x_1}$.
On the other hand, since $(m+n)D_1$ separates $\ell$-jets at $x_1$, there exists an element $t \in H^0(X, \sO_X( (m+n)D_1))$ whose image in $\sO_X/\m_{x_1}^{\ell}$ is 
\[
(s_2^{-1})_{x_1} (s_{3})^{-1}_{x_1} \cdots (s_{r})^{-1}_{x_1} a.
\]
Then we have
\[
\Phi(t s_2 s_3 \cdots s_r) = (a,0,0, \dots, 0),
\]
as desired.
Since $\Phi$ is surjective, we have
\[
\ejet[Z]{D} \ge \frac{\ell}{m+n} = \frac{1}{(m/\ell) +(n/\ell)}>\frac{1}{\epsilon m+(1/\alpha)} = \frac{1}{\epsilon m \alpha +1} \alpha.
\]
By taking limits as $\epsilon \to 0$ and $\alpha \to \min_i \ejet[x_i]{D_i}$, we complete the proof.
\end{proof}

\begin{prop}[\textup{\cite[Theorem 7.3.2]{Mur} cf.~\cite[Theorem 3.1]{MS}}]\label{Mur 7.3.1}
Let $X$ be a normal projective variety of dimension $n$ over an algebraically closed field $k$ of positive characteristic, $D$ be a nef Cartier divisor on $X$ and $Z \subseteq X$ be a finite set of closed points such that $X$ is $F$-injective around $Z$.
If we have $\ejet[Z]{D}>n+\ell$ for an integer $\ell \ge 0$,  then the sheaf $\omega_X \otimes \sO_X(D)$ separates $\ell$-jets at $Z$.
\end{prop}

\begin{proof}
We first note that for any point $x \in Z$,  we have 
\[
\ejet{D} \ge \ejet[Z]{D}>0.
\]
Therefore, it follows from \cite[Lemma 7.2.6]{Mur} that $\B_+(D) \cap Z=\emptyset$.
Then the proof is similar to that of \cite[Theorem 7.3.1]{Mur}.
\end{proof}

\begin{thm}[\textup{cf.~\cite[Corollary 3.2]{MS}}]\label{MS 3.2}
Let $X$ be a normal projective variety of dimension $n$ over an algebraically closed field of positive characteristic and $D$ be a nef and big Cartier divisor on $X$.
If $\ses{D} >2n$ at a smooth closed point $x \in X$,  then the rational map $\Phi_{|K_X+D|}$ induces a birational map onto the image.
\end{thm}

\begin{proof}
By Lemma \ref{ses vs ejet} and Proposition \ref{ejet semiconti}, there is an open dense subset $V \subseteq X_{\mathrm{reg}}$ of $X$ for every closed point $y \in V$,  we have $\ejet[y]{D} >2n$.
We may also assume that $V \cap \B_+(D)=\emptyset$.
It then follows from Proposition \ref{Mur 7.3.1} that $\sO_{X}(K_X+D)$ separates $1$-jet at every closed point $y$.
In particular, we have $H^0(X, \sO_X(K_X+D)) \neq 0$.

Let $Y$ be the image of 
\[
\Phi_{|K_X+D|} : X \dashrightarrow \PP(H^0(X, \sO_X(K_X+D)).
\]
and $\phi: X \dashrightarrow Y$ be the induced rational map.
Then $\phi$ is defined around $V$ and for every $y \in V(k)$, the induced map
\[
\sO_{Y, \phi(y)} \to \sO_{X,y} \onto \sO_{X,y}/\m_y^2
\]
is surjective.

Take a resolution $f : Z \to X$ of the indeterminacy of $\phi$ such that $f$ is isomorphic around $V$ and we denote by $g : Z \to Y$ the induced morphism.
\[
\xymatrix{
 & Z \ar_-{f}[ld] \ar^-{g}[rd]& \\
X \ar@{.>}^-{\phi}[rr]&& Y
}\]
Then for every closed point $z \in f^{-1}(V)$, the ring homomorphism
\[
\sO_{Y, g(z)} \xrightarrow{g^{\#}} \sO_{Z,z} \onto \sO_{Z,z}/\m_z^2
\]
is surjective.

On the other hand, by Proposition \ref{MS 2.16}, for every two closed points $y \neq z \in V$, we have $\ejet[\{y,z\}]{2D}>2n$.
It again follows from Proposition \ref{Mur 7.3.1} that $\sO_{X}(K_X+D)$ separates $0$-jet at $\{y,z\}$, which implies that $\phi|_{V(k)}: V(k) \to Y$ is injective.
Then the restriction $g|_{f^{-1}(V(k))} : f^{-1}(V(k)) \to Y$ is also injective. 
By Lemma \ref{birational criterion} below,  $g$ is birational, and therefore, $\phi$ is birational as desired.
\end{proof}

The following lemma should be well-known to experts,  but we include a proof here for the sake of completeness.

\begin{lem}\label{birational criterion}
Let $g: Z \to Y$ be a proper surjective morphism between varieties over an algebraically closed field $k$.
Assume that there exists an open dense subset $U$ of $Z$ such that $g|_{U(k)}$ is injective and that $\sO_{Y, g(z)} \xrightarrow{g^{\#}} \sO_{Z,z} \onto \sO_{Z,z}/\m_z^2$ is surjective for every closed point $z \in U$.
Then $g$ is birational.
\end{lem}

\begin{proof}
After shrinking $Y$, we may assume that $g$ is flat.
Since $Z$ and $Y$ are irreducible,  every fiber of $g$ is purely $r$-dimensional,  where we set $r : =\dim Z - \dim Y$.
Take a closed point $x \in U$.
If $r$ is larger than $0$,  then there are infinitely many closed points in $U \cap g^{-1}(g(x))$, which is a contradiction to the injectivity of $g|_{U(k)}$.
Therefore, we have $r=0$.
Combining this with \cite[Lemma 02LS]{Sta}, we conclude that $g$ is finite.
Since $g(Z \setminus U)$ is a proper closed subset of $Y$,  after shrinking $Y$, we may assume that $Z=U$.

It follows from \cite[Lemma II. 7.4]{Har} that the morphism $g^{\#}: \sO_{Y} \to g_*\sO_Z$ is surjective, and in particular, $g$ is a closed immersion.
Since $g$ is surjective,  it is isomorphic, as desired.
\end{proof}

\subsection{Birationally boundedness}

\begin{cor}\label{bb dim2}
Let $X$ be a normal projective variety over an uncountable algebraically closed field of positive characteristic with $\dim(X)=2$ and $D$ be a nef and big Cartier divisor on $X$.
If we have $h^0(X, \sO_X(D)) \ge 2$,  then $\Phi_{|K_X+5D|}$ is birational.
\end{cor}

\begin{proof}
This follows from Proposition \ref{ses dim2} and Theorem \ref{MS 3.2}.
\end{proof}

\begin{cor}[Theorem \ref{B}]\label{bb dim3}
  Let $X$ be a $3$-dimensional normal projective variety with rational singularities over an uncountable algebraically closed field of positive characteristic, $\ell \ge 1$ be an integer and $D$ be a nef and big Cartier divisor on $X$.
  We further assume that the following conditions are satisfied:
  \begin{enumerate}[label=\textup{(\roman*)}]
  \item $H^1(X, \sO_X)=0$.
  \item $h^0(X,\sO_X({\ell}D)) > \ell \vol(D)+1$.
  \end{enumerate}
Then $\Phi_{|K_X+(6\ell+1)D|}$ is birational.
\end{cor}

\begin{proof}
This follows from Theorem \ref{Lower bound for seshadri} and Theorem \ref{MS 3.2}.
\end{proof}

We also give variants in terms of vanishing of cohomology.

\begin{cor}\label{bb vanishing}
  Let $V,A>0$ be integers.
  Then there exists an integer $N=N(V, A)>0$ with the following property:
  For any $3$-dimensional normal projective variety $X$ with rational singularities over an uncountable algebraically closed field $k$ of positive characteristic and any nef and big Cartier divisor $D$ on $X$, $\Phi_{|K_X+ND|}$ is birational if the following conditions are satisfied:
  \begin{enumerate}[label=\textup{(\roman*)}]
  \item $\vol(D) < V$.
  \item $H^1(X, \sO_X)=0$.
  \item $H^2(X, \sO_X(iD))=0$ for every integer $i \ge 0$.
  \item $h^1(X,\sO_X(iD)) +h^3(X, \sO_X(iD)) <A$ for $i=0, 1,2,3$.
  \end{enumerate}
  \end{cor}

  \begin{proof}
    This follows from Corollary \ref{lower bound for seshadri vanishing} and Theorem \ref{MS 3.2}.
    \end{proof}
    
    \begin{cor}\label{bb weak Fano}
      Let $r, V,A>0$ be integers.
      Then there exists a positive integer $N=N(r, V, A)>0$ with the following property:
      For any $3$-dimensional normal projective variety $X$ with rational singularities over an uncountable algebraically closed field $k$ of positive characteristic, $\Phi_{|-NK_X|}$ is birational if the following conditions are satisfied:
      \begin{enumerate}[label=\textup{(\roman*)}]
      \item $-rK_X$ is a nef and big Cartier divisor.
      \item $\vol(-K_X) < V$.
      \item $H^1(X, \sO_X(irK_X))=0$ for every integer $i \ge 0$.
      \item $h^2(X,\sO_X(irK_X)) <A$ for $i=0,1,2,3$.
    \end{enumerate}
    \end{cor}

    \begin{proof}
      This follows from Corollary \ref{lower bound for seshadri weak Fano} and Theorem \ref{MS 3.2}.
      \end{proof}
  
\section{Log birationally boundedness}\label{sec5}

In this section, we provide an upper bound for intersection
numbers in terms of volumes (Proposition \ref{HMX lemma}).
As a corollary, we give a sufficient condition for a set of projective log pairs to be log birationally bounded (Theorem \ref{D}).

\begin{prop}[\textup{\cite[Theorem 8.1.1]{Mur}}]\label{Mur 8.1.1}
Let $X$ be a normal projective variety over an algebraically closed field $k$ of positive characteristic, $\Delta$ be an effective $\Q$-Weil divisor on $X$ such that $K_X+\Delta$ is $\Q$-Cartier and $D$ be a Cartier divisor on $X$.
Assume that $(X, \Delta)$ is $F$-pure at some closed point $x$ and 
\[
\ejet{D-(K_X+\Delta)}>\dim(X).
\]
Then $\sO_X(D)$ is globally generated at $x$.
\end{prop}

\begin{proof}
Let $\m_x$ be the maximal ideal at $x$.
It follows from \cite[Proposition 2.6]{TW} that we have
\[
\fpt_x(X,\Delta; \m_x) \le \fpt_x(X ; \m_x) \le \dim(X).
\]
Therefore, the assertion follows from \cite[Theorem 8.1.1]{Mur}.
\end{proof}

\begin{prop}[\textup{cf.~\cite[Lemma 3.2]{HMX13}}]\label{HMX lemma}
Let $X$ be a normal projective variety of dimension $n$ over an algebraically closed field $k$ of positive characteristic and $M$ be a base point free Cartier divisor such that $\Phi_{|M|}$ is birational onto its image.
We set $H : = 2(2n+1)M$.

Let $D$ be a reduced Weil divisor on $X$ such that a log resolution of $(X,D)$ exists.
Then we have
\[
(H^{n-1} \cdot D) \le 2^n \vol(K_X+D+H).
\]
\end{prop}

\begin{proof}
The main idea of the proof is similar to that of \cite[Lemma 3.2]{HMX13}, but we will use Proposition \ref{Mur 8.1.1} instead of Kodaira type vanishing.
As in the proof of \cite[Lemma 3.2]{HMX13}, we may assume that both $X$ and $D$ are smooth, and that every irreducible component of $D$ is not contracted by $\Phi_{|M|}$.
Let $Y$ be the normalization of $\Im(\Phi_{|M|})$ and $\phi :X \to Y$ be the birational morphism induced by $\Phi_{|M|}$.
For every integer $m>0$, we set 
\[
A_m : = K_X+D+mH.
\]

\textbf{Step.1}
Since there exists an ample and free Cartier divisor $A$ on $Y$ with $M = \phi^*A$, it follows from Lemma \ref{ses vs ejet} and \cite[Example 5.1.18]{Laz} that we have 
\[
\ejet{M} =\ejet[\phi(x)]{A} =\ses[\phi(x)]{A} \ge 1,
\]
for every closed point $x \in X \setminus \Exc(\phi)$.
Combining this with the fact that $X$ is $F$-pure (Remark \ref{hierarchy} (i)), it the follows from Proposition \ref{Mur 8.1.1} that $\sO_X(K_X+(2n+1)M)$ is globally generated outside $\Exc(\phi)$ and in particular, $H^0(X, \sO_X(K_X+(2n+1)M)) \neq 0$.
Therefore, we have 
\begin{equation}\label{inequality HMX}
H^0(X, \sO_X(2A_1-D))= H^0(X, \sO_X(2K_X+D+2H) ) \neq 0.
\end{equation}

Since $(X,D)$ is $F$-pure by Remark \ref{hierarchy} (ii), it follows from the same reason as in the previous paragraph that the sheaf $\sO_X(K_X+D+(2n+1)M)$ is globally generated outside $\Exc(\phi)$.
In particular,  it is globally generated around the generic point of every component of $D$.
Take global sections $s \in H^0(X, \sO_X(K_X+D+(2n+1)M))$ and $l \in H^0(X, \sO_X((2n+1)M))$ whose restrictions to each component of $D$ are non-zero.
Then the restriction of 
\[
t: = s^{\otimes 2m-1} \otimes l \in H^0(X, \sO_X(2mA_1 - A_m)),
\]
to every component of $D$ is non-zero.

\textbf{Step.2}
Consider the following commutative diagram
\[
\xymatrix{
0 \ar[r] & \sO_X(A_m-D) \ar[r] \ar@{^{(}->}[d] & \sO_X(A_m) \ar@{^{(}->}[d] \ar^-{\alpha}[r] & \sO_D(A_m) \ar[r] \ar@{^{(}->}^-{v}[d] & 0 \\
0 \ar[r] & \sO_X(2mA_1-D) \ar[r] & \sO_X(2mA_1) \ar^-{\beta}[r] & \sO_D(2mA_1) \ar[r] & 0 ,
}\]
where each line is exact and the vertical morphisms are injections induced by multiplying by $t$.
By taking global sections,  we obtain the following commutative diagram
\[
\xymatrix{
 H^0(X, \sO_X(A_m)) \ar@{^{(}->}[d] \ar^-{H^0(\alpha)}[r] & H^0(D, \sO_D(A_m))  \ar@{^{(}->}^-{H^0(v)}[d] \\
 H^0(X, \sO_X(2mA_1)) \ar^-{H^0(\beta)}[r] & H^0(D, \sO_D(2mA_1))    ,
}\]
with $\Ker(H^0(\beta))$ isomorphic to $H^0(X, \sO_X(2mA_1-D))$ and $\Cok(H^0(\alpha))$ contained in $H^1(X, \sO_X(A_m -D))$.
Now we set
\[
P(m) : = h^0(D, \sO_D(A_m)), \ Q(m):=h^0(X, \sO_X(2m A_1)), \ R(m) :=h^1(X, \sO_X(A_m-D)).
\]
Then we have
\begin{align*}
P(m) - R(m) &\le P(m) - \dim(\Cok(H^0(\alpha)))\\
&=\dim(\Im(H^0(\alpha))) \\
& \le \dim(\Im(H^0(\beta)))\\
&=Q(m) - h^0(X, 2mA_1-D)
\end{align*}
Since $h^0(2mA_1-D) \ge Q(m-1)$ by \eqref{inequality HMX},  we obtain the inequality
\begin{equation}\label{P,Q,R}
P(m) - R(m) \le Q(m)-Q(m-1).
\end{equation}

\textbf{Step.3}
Take an integer $m>0$ and consider the Leray Spectral sequence
\[
E_2^{p,q} = H^p(Y, R^q\phi_*\sO_X(K_X+mH)) \Rightarrow H^{p+q}(X, \sO_X(K_X+mH)).
\]
Noting that $E_2^{p,q} =H^p(Y, (R^q\phi_*\sO_{X}(K_X))(2m(2n+1)A))$ is zero if $p>0$ and $m$ is sufficiently large,  we have $R(m)=h^0(Y, (R^1\phi_*\sO_X(K_X))(2m(2n+1)A))$ for sufficiently large $m$.
Since the support of the sheaf $R^1\phi_*\sO_X(K_X)$ is contained in the non-isomorphic locus of $\phi$ on $Y$, we have $\dim(\Supp(R^1\phi_*\sO_X(K_X))) \le n-2$.
Therefore, it follows from the Asymptotic Riemann-Roch (cf.~\cite[Example 1.2.19]{Laz}) that we have 
\[
\lim_{\ell \to \infty} \frac{R(\ell)}{\ell^{n-1}} =0.
\]

Similarly, since $H|_D$ is nef, it follows from the Asymptotic Riemann-Roch (cf.~\cite[Corollary 1.4.41]{Laz}) that we have
\[
\lim_{\ell \to \infty} \frac{P(\ell)}{\ell^{n-1}} = \frac{(H^{n-1} \cdot D)}{(n-1)!}.
\]

\textbf{Step.4}
We fix a real number $\epsilon>0$.
Then it follows from Step.3 that there exists an integer $m_1$ such that for any integer $m \ge m_1$, we have
\[
\frac{P(m)}{m^{n-1}} \ge \frac{(H^{n-1} \cdot D)}{(n-1)!}-\epsilon , \textup{ and } \frac{R(m)}{m^{n-1}} \le \epsilon.
\]
On the other hand, it follows from the definition of the volume, after replacing $m_1$ by a larger number, we may assume that 
\[
\sup_{m \ge m_1} \frac{Q(m)}{(2m)^n} \le \frac{\vol(A_1)}{n!}+\epsilon.
\]
Take any integer $L \ge m_1$.
By the inequality \eqref{P,Q,R}, we have
\begin{align*}
\sum_{m=m_1+1}^L m^{n-1} \left( \frac{(H^{n-1}\cdot D)}{(n-1)!} - 2\epsilon \right) & \le \sum_{m=m_1+1}^L P(m)-R(m) \\
& \le Q(L) - Q(m_1) \\
& \le Q(L) \\
& \le (2L)^{n} \left(\frac{\vol(A_1)}{n!} + \epsilon \right).
\end{align*}
Combining this with the inequality $(L^n-m_1^n)/n = \int_{m_1}^L x^{n-1} dx \le \sum_{m=m_1+1}^L m^{n-1}$, one has
\[
\left(1-\left(\frac{m_1}{L}\right)^n\right) \left(\frac{(H^{n-1}\cdot D)}{n!} - \frac{2\epsilon}{n}\right) \le 2^{n} \left(\frac{\vol(A_1)}{n!} + \epsilon\right).
\]
By taking the limit as $L \to \infty$ and $\epsilon \to 0$, we have
\[
\frac{(H^{n-1}\cdot D)}{n!} \le 2^{n} \frac{\vol(A_1)}{n!} ,
\]
which completes the proof.
\end{proof}

\begin{thm}[Theorem \ref{D}]\label{criterion of lbb}
Let $V>0$ be an integer and $k$ be an algebraically closed field of positive characteristic. 
Suppose that $\mathcal{E}$ is a set of $n$-dimensional projective log pairs over $k$ such that for every $(X,\Gamma) \in \mathcal{E}$, the following conditions are satisfied:
\begin{enumerate}[label=\textup{(\roman*)}]
\item $\Gamma$ is reduced, 
\item there is a log resolution of $(X, \Gamma)$, and 
\item there exists a Weil divisor $E$ on $X$ such that $\Phi_{|E|}$ is birational onto its image, $\vol(E)<V$ and $\vol(K_X+\Gamma+2(2n+1) E)<V$.
\end{enumerate}
Then $\mathcal{E}$ is log birationally bounded.
\end{thm}

\begin{proof}
Take an element $(X, \Gamma) \in \mathcal{E}$ and a Weil divisor $E$ on $X$ satisfying the condition (iii).
After removing the fixed divisor of $|E|$, we may assume that 
\[
  \mathrm{cod}(\Bs(|E|)) \ge 2.
\]
Let $Y$ be the image of $\Phi_{|E|}$, $\nu: Y^n \to Y$ be the normalization and $A$ be the very ample Cartier divisor on $Y$ associated to $\Phi_{|E|}$.
We denote by $\Gamma_{Y^n}$ the sum of the strict transform $\phi_*\Gamma$ of $\Gamma$ and all $\phi^{-1}$-exceptional divisors, where $\phi : X \dashrightarrow Y^n$ is the induced birational map. 
By Lemma \ref{bb criterion}, it suffices to show that $(A^n)$ and $(A^{n-1} \cdot \nu_*\Gamma_{Y^n})$ are bounded above by a constant.

Take a resolution $f : Z \to X$ of the indeterminacy of $\phi$ and let $g : Z \to Y$ be the induced morphism.
\[
\xymatrix{
 & Z \ar_-{f}[ld] \ar^-{g}[rd]&& \\
X \ar@{.>}^-{\phi}[rr]&& Y^n \ar^-{\nu}[r] & Y
}\]
We set $M : = (\nu \circ g)^*A$ and let $\Gamma_Z$ be the sum of the strict transform $f^{-1}_*\Gamma$ of $\Gamma$ and all $f$-exceptional divisors.
Since we have $g_* \Gamma_Z=\Gamma_{Y^n}$ and $f_*M=E$, one has
\begin{align*}
  (A^n) = (M^n) =\vol(M) \le \vol(E)<V
\end{align*}
and
\begin{align*}
(A^{n-1} \cdot \nu_* \Gamma_{Y^n}) & = (M^{n-1} \cdot \Gamma_Z) \\
& \le 2^n \vol(K_Z+\Gamma_Z+2(2n+1)M) \\
& \le 2^n \vol(f_*(K_Z+\Gamma_Z+2(2n+1)M)) \\
&= 2^n \vol(K_X+\Gamma + 2(2n+1)E) \le 2^nV,
\end{align*}
where the inequality in the second line follows from Proposition \ref{HMX lemma}.
\end{proof}

\section{Finiteness of log minimal models and lc models}\label{sec6}

In this section, we consider a projective morphism $\pi : \mathcal{X} \to T$ between normal varieties with relative dimension $3$ (cf.~Setting \ref{relative setting}) and we prove the finiteness of log minimal models (Theorem \ref{finiteness of minimal model relative}) and lc models (Theorem \ref{finiteness of lc model relative}).
We also prove that after shrinking $T$, we may check the relative bigness of adjoint divisors after restricting to any closed fiber (Theorem \ref{bigness constancy}).

\subsection{Finiteness of models I --absolute case--}

In this subsection, We will work with the following setup.

\begin{setting}\label{absolute setting}
  Let $X$ be a $3$-dimensional projective normal variety over an $F$-finite field $k$ of characteristic $p>5$, $A \ge 0$ be a big effective $\Q$-Cartier divisor and $\mathcal{C} \subseteq \WDiv_{\R}(X)$ be a rational polytope (that is, the convex hull of a finite subset of $\WDiv_{\Q}(X)$). 
  We assume that $(X, \Delta)$ is klt and $A \le \Delta$ for every $\Delta \in \mathcal{C}$.
\end{setting}

\begin{prop}[\textup{cf.~\cite[Theorem 3.11.1]{BCHM}}]\label{Cone theorem}
  With notation as in Setting \ref{absolute setting}, let $\mathcal{N}(\mathcal{C})$ be the subset
  \[\mathcal{N}(\mathcal{C}) : = \{\Delta \in \mathcal{C} \mid K_X+\Delta \textup{ is nef}\}\] 
  of $\mathcal{C}$.
  Then $\mathcal{N}(\mathcal{C})$ is a rational polytope.
  \end{prop}
  
  \begin{proof}
  By replacing $X$ by a small $\Q$-factorization, we may assume that $X$ is $\Q$-factorial.
  As in \cite[Remark 4.1]{BW}, we may also assume that $A$ is ample.
  Take any element $\Delta \in \mathcal{C}$.
  It then follows from the cone theorem (\cite[Theorem 1.1(2)]{DW}) that the number of $K_X+\Delta-A/2$-negative extremal rays is finite.
  Therefore, the proof is similar to that of \cite[Theorem 3.11.1]{BCHM}.
  \end{proof}

\begin{cor}[\textup{cf.~\cite[Corollary 3.11.2]{BCHM}}]\label{rational polytope}
  With notation as in Setting \ref{absolute setting}, let $\phi : X \dashrightarrow Y$ be a birational contraction.
  Then 
  \[\mathcal{W}_{\phi}(\mathcal{C}) : = \{ \Delta \in \mathcal{C} \mid \phi \textup{ is a weak lc model for $(X, \Delta)$}\}
    \] is a rational polytope.
  \end{cor}
  
  \begin{proof}
  The proof is similar to that of \cite[Corollary 3.11.2]{BCHM}, but we include it here for convenience.
  We may assume that $\mathcal{W}_{\phi}(\mathcal{C})$ is non-empty and therefore, $Y$ admits an effective $\R$-Weil divisor $\Gamma$ with $(Y,\Gamma)$ klt.
  After replacing $Y$ by a small $\Q$-factorization, we may assume that $\phi_*A$ is $\Q$-Cartier.

  After replacing $\mathcal{C}$ by the sub polytope
  \[\mathcal{C}' : = \{\Delta \in \mathcal{C} \mid \phi \textup{ is $(K_X+\Delta)$-non-positive}\},\]
  we may assume that $\mathcal{W}_{\phi}(\mathcal{C})$ is the pullback of $\mathcal{N}(\mathcal{D})$ by the $\R$-linear map
  \[
    \phi_* : \mathcal{C} \onto \mathcal{D} : = \phi_*(\mathcal{C}) \ ; \ \Delta \mapsto \phi_*\Delta .\]
    Therefore, the assertion follows from Proposition \ref{Cone theorem}.
  \end{proof}

\begin{prop}[\textup{cf.~\cite[Lemma 7.1]{BCHM}}]\label{finiteness of minimal model}
With notation as in Setting \ref{absolute setting}, there are finitely many birational contractions $\psi_i : X \dashrightarrow Z_i$ such that for each $\Delta \in \mathcal{C}$ with $K_X+\Delta$ pseudo-effective, there is an index $i$ such that $\psi_i$ is a log minimal model of $(X,\Delta)$.
\end{prop}

\begin{proof}
After replacing $X$ by its small $\Q$-factorization, the assertion follows from \cite[Theorem 7.3]{Wal}.
\end{proof}

\begin{cor}[\textup{cf.~\cite[Corollary 1.1.5]{BCHM}}]\label{rational polytope2}
With notation as in Setting \ref{absolute setting}, let $\mathcal{E}(\mathcal{C})$ be the subset
\[\mathcal{E}(\mathcal{C}) : = \{\Delta \in \mathcal{C} \mid K_X+\Delta \textup{ is pseudo-effective}\}\] 
of $\mathcal{C}$.
Then $\mathcal{E}(\mathcal{C})$ is a rational polytope.
\end{cor}

\begin{proof}
Take finitely many birational contractions $\psi_1, \psi_2, \dots, \psi_n$ as in Proposition \ref{finiteness of minimal model}.
Since we have 
\[\mathcal{E}(\mathcal{C}) = \bigcup_{i=1}^n \mathcal{W}_{\psi_i}(\mathcal{C}),\]
the assertion follows from Corollary \ref{rational polytope}.
\end{proof}

\begin{prop}[\textup{cf.~\cite[Section 7]{BCHM}}]\label{finiteness of lc model}
  With notation as in Setting \ref{absolute setting}, there are finitely many birational contractions $\psi_i : X \dashrightarrow Z_i$ such that for each $\Delta \in \mathcal{C}$ with $K_X+ \Delta$ big, there is an index $i$ such that $\psi_i$ is the lc model of $(X, \Delta)$.
\end{prop}

\begin{proof}
By Proposition \ref{finiteness of minimal model} and Corollary \ref{rational polytope}, after shrinking $\mathcal{C}$, we may assume that there is a birational contraction $\phi : X \dashrightarrow Y$ such that $\phi$ is a weak lc model of $K_X+\Delta$ for every $\Delta \in \mathcal{C}$.
We note that $K_Y+\phi_*\Delta$ is semiample by Lemma \ref{rmk on wlc}.
Take a birational morphism 
\[ g_{\Delta} : Y \to Z_{\Delta} \]
to a normal projective variety $Z_{\Delta}$ such that $K_Y+\phi_*\Delta$ is the pullback of an ample $\R$-Cartier divisor on $Z_{\Delta}$.
Then the composite birational map $g_{\delta} \circ \phi$ is the lc model of $(X, \Delta)$.
We denote by $G_{\Delta}$ the smallest face of the rational polytope $\mathcal{N}(\phi_* \mathcal{C})$ containing $\phi_*\Delta$, here $\phi_*\mathcal{C}$ is the image of $\mathcal{C}$ via the homomorphism 
\[\phi_* : \WDiv_{\R}(X) \to \WDiv_\R(Y).\]
It follows from the rigidity lemma (cf.~\cite[Lemma 1.15]{Deb}) that $g_{\Delta}=g_{\Delta'}$ if $G_{\Delta}=G_{\Delta'}$.
Since $\mathcal{N}(\phi_* \mathcal{C})$ has finitely many faces (Proposition \ref{Cone theorem}), this completes the proof of the proposition. 
\end{proof}

\subsection{Invariance of Iitaka dimensions}

\begin{prop}\label{Iitaka dimension of fiber}
Let $\pi : \mathcal{X} \to T$ be a projective morphism with connected and normal fibers between normal varieties over an algebraically closed field $k$ of characteristic $p>5$ such that the dimension of the generic fiber $\mathcal{X}_\eta$ is $3$.
Suppose that $\Delta \ge 0$ is a $\Q$-Weil divisor on $\mathcal{X}$ such that $K_{\mathcal{X}}+\Delta$ is $\Q$-Cartier.
We denote by $\Delta_{\eta}$ and $\Delta_{\overline{\eta}}$ the flat pullback of $\Delta$ to the generic fiber $\mathcal{X}_{\eta}$ and the geometric generic fiber $\mathcal{X}_{\overline{\eta}}$, respectively. 
We further assume that $(\mathcal{X}_{\eta}, \Delta_{\eta})$ is klt.
Then the following assertions hold.
  \begin{enumerate}[label=\textup{(\arabic*)}]
    \item If $K_{\mathcal{X}} + \Delta$ is pseudo-effective over $T$ and there is a $\Q$-Cartier divisor $A$ on $\mathcal{X}$ which is big over $T$ and $0 \le A \le \Delta$, then we have 
    \[\kappa(\mathcal{X}/T, K_{\mathcal{X}}+\Delta)=\kappa(\mathcal{X}_t, (K_{\mathcal{X}}+\Delta)_t) \ge 0\] 
    for general $t \in T$.
    \item If $\kappa(\mathcal{X}/T, K_{\mathcal{X}}+\Delta) \ge 1$ and the pair $(\mathcal{X}_{\overline{\eta}}, \Delta_{\overline{\eta}})$ on the geometric generic fiber $X_{\overline{\eta}}$ is klt, then we have 
    \[\kappa(\mathcal{X}/T, K_{\mathcal{X}}+\Delta)=\kappa(\mathcal{X}_t, (K_{\mathcal{X}}+\Delta)_t)\] 
    for general $t \in T$.
    \item If $K_{\mathcal{X}}+\Delta$ is not pseudo-effective over $T$, then the restriction $(K_{\mathcal{X}}+\Delta)_t$ to a general fiber $\mathcal{X}_t$ is not pseudo-effective.

  \end{enumerate}
\end{prop}

\begin{proof}
We first consider the case where $\mathcal{D} : = K_{\mathcal{X}}+\Delta$ is pseudo-effective over $T$.
By \cite[Theorem 1.6]{DW} and Lemma \ref{rmk on wlc}, there exists a semiample model
\[ \phi : \mathcal{X}_{\eta} \dashrightarrow Y \]
of $\mathcal{D}_{\eta}$ over $\Spec(K(T))$.
By Lemma \ref{spreading out schemes}, after shrinking $T$, there exists a normal variety $\mathcal{Y}$ which is projective over $T$ with $\mathcal{Y}_{\eta} \cong Y$.
Then $\phi$ defines a birational map 
\[\Phi: \mathcal{X}  \dashrightarrow \mathcal{Y} \]
such that $\phi=\Phi_{\eta}$. 
Then it follows from Lemma \ref{model and fiber} that the restriction
\[\Phi_t : \mathcal{X}_t \dashrightarrow \mathcal{Y}_t \]
is a semiample model of $\mathcal{D}_t$ for general $t \in T$.
Noting that 
\[
  H^0(\mathcal{X}_t, \sO_{\mathcal{X}_t} (m \mathcal{D}_t)) = H^0(\mathcal{Y}_t, \sO_{\mathcal{Y}_t} (m (\Phi_t)_*\mathcal{D}_t))
\]
for every integer $m>0$, we have
\[\kappa(\mathcal{X}_t, \mathcal{D}_t) = \kappa(\mathcal{Y}_t, (\Phi_t)_*\mathcal{D}_t) = \kappa(\mathcal{Y}_t, (\Phi_*\mathcal{D})_t).\] 
After replacing $\mathcal{X}$ by $\mathcal{Y}$ and $\mathcal{D}$ by $\Phi_* \mathcal{D}$, we may assume that $\mathcal{D}$ is semiample over $T$.
Take a semiample fibration $f : \mathcal{X} \to \mathcal{Z}$ of $\mathcal{D}$ over $T$.
Then we have 
\[
  \kappa(\mathcal{X}_t, \mathcal{D}_t)=\dim(\mathcal{Z}_t)\] for general $t \in T$, which proves the assertions in (1) and (2).

Finally, we consider the case where $\mathcal{D}=K_{\mathcal{X}}+\Delta$ is not pseudo-effective over $T$.
Take a $\pi$-ample $\Q$-Cartier divisor $H$ on $\mathcal{X}$ such that $\mathcal{D}+H$ is $\pi$-big.
By Lemma \ref{Bertini}, after replacing $H$ by its $\Q$-linear equivalent $\Q$-divisor and shrinking $T$, we may assume that $H$ is effective and that $(\mathcal{X}_\eta, \Delta_{\eta}+H_{\eta})$ is klt.
Let $\lambda>0$ be the smallest real number such that $\mathcal{D}+ \lambda H$ is pseudo-effective over $T$.
We note that $\lambda$ is a rational number by Corollary \ref{rational polytope2} and that $\mathcal{D}+\lambda H$ is not big over $T$.
Applying (1) to the pair $(\mathcal{X}, \Delta+ \lambda H)$, the $\Q$-divisor $\mathcal{D}_t+\lambda H_t$ is not big for general $t \in T$.
Therefore, 
\[
  \mathcal{D}_t=(\mathcal{D}_t+\lambda H_t) - \lambda H_t
  \]
  is not pseudo-effective, as desired.
\end{proof}

From now on, we will work with the following setup.

\begin{setting}\label{relative setting}
  Let $\pi : \mathcal{X} \to T$ be a projective surjective morphism between normal quasi-projective varieties over an algebraically closed field $k$ of characteristic $p>5$ such that a general fiber is $3$-dimensional normal variety.
  Suppose that $A \ge 0$ is a $\pi$-big effective $\Q$-Cartier divisor on $X$ and $\mathcal{C} \subseteq \WDiv_{\R}(\mathcal{X})$ is a rational polytope such that every element $\Delta \in \mathcal{C}$ satisfies the following properties:
  \begin{itemize}
    \item $A \le \Delta$,
    \item $K_{\mathcal{X}} + \Delta$ is $\R$-Cartier, and 
    \item $(\mathcal{X}_{\eta}, \Delta_{\eta})$ is klt, where $\eta \in T$ is the generic point.
  \end{itemize}
\end{setting}

\begin{thm}\label{bigness constancy}
  With notation as in Setting \ref{relative setting}, we further assume that there exists an element $\Phi \in \mathcal{C}$ such that $K_{\mathcal{X}}+\Phi$ is $\pi$-big.
  Then there exists an open dense subset $U$ of $T$ such that for any $\Delta \in \mathcal{C}$, if $(K_{\mathcal{X}}+\Delta)_t$ is big (resp.~pseudo-effective) at some closed point $t \in U$, then $K_{\mathcal{X}}+\Delta$ is big over $T$ (resp.~pseudo-effective over $T$).
\end{thm}

\begin{proof}
  By Proposition \ref{rational polytope2}, 
  \begin{align*}
    \mathcal{E}:=\{\Delta \in \mathcal{C} \mid K_{\mathcal{X}}+ \Delta \textup{ is pseudo-effective over $T$}\}
  \end{align*} 
  is a rational polytope.
  Take a generator $\Theta_1, \dots , \Theta_m \in \mathcal{E}$.
  By Lemma \ref{Iitaka dimension of fiber} (1), we may find an open dense subset $U \subseteq T$ such that $(K_{\mathcal{X}}+\Theta_{i})_t$ is pseudo-effective for every $i$ and $t \in U$.
  In particular, for every $t \in U$, the convex set
  \[\mathcal{E}_t : = \{\Delta \in \mathcal{C} \mid (K_{\mathcal{X}}+\Delta)_t \textup{ is pseudo-effective}\}\]
  contains $\mathcal{E}$.
  Let $F_0 = \mathcal{E}, F_1, \dots, F_n \subseteq \mathcal{E}$ be the all faces of $\mathcal{E}$.
  We take a $\Q$-divisor $\Delta_i \in \mathcal{C}$ which is contained in the relative interior $\RInt(F_i)$ of $F_i$.
  After shrinking $U$, we may assume that for every $i$, the Iitaka dimension $\kappa(\mathcal{X}_t, (K_{\mathcal{X}}+\Delta_i)_t)$ is independent of $t \in U$.
  
By the assumption, $K_{\mathcal{X}}+\Theta$ is $\pi$-big for every $\Theta \in \RInt(\mathcal{E})$. 
Therefore, it suffices to show that if $(K_{\mathcal{X}}+\Delta)_t$ is big (resp.~pseudo-effective) for some $\Delta \in \mathcal{C}$ and some $t \in U$, then one has $\Delta \in \RInt(\mathcal{E})$ (resp.~$\Delta \in \mathcal{E}$).

Assume to the contrary, we may find an element $\Delta \in \mathcal{C}$ such that $(K_{\mathcal{X}}+\Delta)_t$ is big (resp.~pseudo-effective) at some $t \in U$, but $\Delta \not\in \RInt(\mathcal{E})$ (resp.~$\Delta \not\in \mathcal{E}$).
Let $a \in [0,1]$ be the minimal number such that $\Gamma : = (1-a) \Delta + a \Delta_0 \in \mathcal{E}$.
  Since bigness is an open condition, we have $a<1$ and therefore, $\Gamma$ is contained in the boundary $\partial \mathcal{E}$ of $\mathcal{E}$.
Since we have $0 \le a < 1 $ (resp.~$0<a<1$), the $\R$-Cartier divisor 
\[
  (K_\mathcal{X}+\Gamma)_t = (1-a)(K_{\mathcal{X}}+\Delta)_t + a (K_{\mathcal{X}} + \Delta_0)_t
\]
is big.
Take an index $0<i$ such that $\Gamma$ is contained in the face $F_i$.
Then we may find a real number $\epsilon > 0$ such that 
\[\Theta : = \Gamma + (1+\epsilon)(\Delta_i - \Gamma)\]
is contained in $F_i$.
We note that $(K_{\mathcal{X}}+\Theta)_t$ is pseudo-effective by the inclusion $F_i \subseteq \mathcal{E} \subseteq \mathcal{E}_t$.
Combining this with the bigness of $(K_{\mathcal{X}}+\Gamma)_t$, we conclude that 
  \[
    (K_{\mathcal{X}}+\Delta_i)_t = \frac{(K_{\mathcal{X}}+\Theta)_t + \epsilon(K_{\mathcal{X}}+ \Gamma)_t}{1+\epsilon}\]
    is also big, which is a contradiction.
  \end{proof}

\subsection{Finiteness of models II --relative case--}

\begin{thm}\label{finiteness of minimal model relative}
  With notation as in Setting \ref{relative setting}, there exist an open dense subset $U \subseteq T$ and finitely many birational contractions 
  \[\Phi_i : \mathcal{X}_U : =\pi^{-1}(U) \dashrightarrow \mathcal{Z}_i\] 
  over $U$ with the following property:
  For every $\Delta \in \mathcal{C}$ with $K_{\mathcal{X}}+\Delta$ pseudo-effective over $T$, there is an index $i$ such that for every $t \in U$, the restriction
    \[(\Phi_i)_t : \mathcal{X}_t \dashrightarrow (\mathcal{Z}_i)_t\] 
    is $(K_{\mathcal{X}}+\Delta)_t$-negative semiample model of $(K_{\mathcal{X}}+\Delta)_t$.
  \end{thm}

\begin{proof}
By Proposition \ref{finiteness of minimal model}, there are finitely many birational contractions 
\[\phi_i : \mathcal{X}_{\eta} \dashrightarrow Z_i\] such that for any divisor $\Delta \in \mathcal{C}$ with $K_{\mathcal{X}}+\Delta$ pseudo-effective over $T$, there exists an index $i$ such that $\phi_i$ is a log minimal model of $(\mathcal{X}_{\eta}, \Delta_{\eta})$.
Noting that $\phi_i$ is surjective in codimension one and $\mathcal{X}_{\eta}$ is smooth over $K(T)$ in codimension one, $Z_i$ is also geometrically normal over $K(T)$.

By Lemma \ref{spreading out schemes}, we may take an open dense subset $U \subseteq T$ and a normal variety $\mathcal{Z}_i$ projective over $T$ such that $(\mathcal{Z}_i)_{\eta} = Z_i$.
Moreover, the birational map $\phi_i$ defines a birational map $\Phi_i : \mathcal{X}_U \dashrightarrow \mathcal{Z}_i$ over $U$ such that $\Phi_\eta =\phi$.

For every $i$, we set 
\begin{align*}
\mathcal{W}_i &: = \{\Delta \in \mathcal{C} \mid \phi_i \textup{ is a weak lc model of $(\mathcal{X}_{\eta}, \Delta_{\eta})$}\},\textup{and} \\
\mathcal{M}_i & : =\{ \Delta \in \mathcal{C} \mid \phi_i \textup{ is a $(K_{\mathcal{X}}+\Delta)_{\eta}$-negative semiample model of $(K_{\mathcal{X}}+\Delta)_{\eta}$}\}.
\end{align*}
By Lemma \ref{rmk on wlc}, Corollary \ref{rational polytope} and Remark \ref{convexity}, there exist faces $F_{i,1} , \dots, F_{i,\ell_i}$ of $\mathcal{W}_i$ such that 
\[
  \mathcal{M}_i = \coprod_{j=1}^{\ell_i} \RInt(F_{i,j}).
\]
For each $i,j$, take a finite subset $\{\Delta_{i,j}^{(1)}, \dots, \Delta_{i,j}^{(a_{i,j})}\} \subseteq F_{i,j}$ whose convex hull is $F_{i,j}$ and take an element $\Delta_{i,j}^{(0)} \in \RInt(F_{i,j})$.

We fix $i,j \ge 1$.
By Lemma \ref{model and fiber}, after shrinking $U$, we may assume that for every $t \in U$ and every $\ell \ge 0$, the restriction $(\Phi_i)_t$ is a semiample model of $(K_{\mathcal{X}}+ (\Delta_{i,j}^{(\ell)}))_t$.
Similarly, it again follows from Lemma \ref{model and fiber} that $(\Phi_i)_t$ is $(K_{\mathcal{X}}+ (\Delta_{i,j}^{(0)}))_t$-negative for every $t \in U$.
Therefore, $(\Phi_i)_t$ is $(K_{\mathcal{X}}+\Delta)_t$-negative semiample model of $(K_{\mathcal{X}}+\Delta)_t$ for every $\Delta \in \RInt(F_{i,j})$, as desired. 
\end{proof}
  
\begin{thm}[cf. \textup{\cite[Lemma 9.1]{HMX2}}]\label{finiteness of lc model relative}
With notation as in Setting \ref{relative setting}, there exist an open dense subset $U \subseteq T$ and finitely many birational contractions 
\[\Phi_i : \mathcal{X}_U : =\pi^{-1}(U) \dashrightarrow \mathcal{Z}_i\] 
over $U$ with the following property:
For every $\Delta \in \mathcal{C}$ with $K_{\mathcal{X}}+\Delta$ big over $T$, there is an index $i$ such that for every $t \in U$, the restriction 
\[(\Phi_i)_t : \mathcal{X}_t \dashrightarrow (\mathcal{Z}_i)_t
\]
is the ample model of $(K_{\mathcal{X}}+\Delta)_t$.
\end{thm}

\begin{proof}
By Proposition \ref{finiteness of lc model}, there are finitely many birational contractions 
\[\phi_i : \mathcal{X}_{\eta} \dashrightarrow Z_i\] such that for any divisor $\Delta \in \mathcal{C}$ with $K_{\mathcal{X}}+\Delta$ big over $T$, there exists an index $i$ such that $\phi_i$ is the lc model for $(\mathcal{X}_{\eta}, \Delta_{\eta})$.
As in the proof of Theorem \ref{finiteness of minimal model relative}, we may take an open dense subset $U \subseteq T$ and birational contractions $\Phi_i : \mathcal{X}_U \dashrightarrow \mathcal{Z}_i$ over $U$ such that $\Phi_\eta =\phi$.
Then the rest of the proof is similar to that of Theorem \ref{finiteness of minimal model relative} by considering
\begin{align*}
\mathcal{A}_i & : =\{ \Delta \in \mathcal{C} \mid \phi_i \textup{ is the ample model of $(K_{\mathcal{X}}+\Delta)_{\eta}$} \}
\end{align*}
instead of $\mathcal{M}_i$.
\end{proof}

\section{Boundedness of weak Fano varieties}\label{sec7}

In this section, we first consider the behavior of N\'{e}ron-Severi groups of fibers and prove a result about constancy of them (Proposition \ref{constancy of NS}).
We next give a sufficient condition for general closed fibers to be $\Q$-factorial (Proposition \ref{Q-factoriality on fibers}).
As a consequence, we prove Theorem \ref{E} (Theorem \ref{Bir bdd to bdd}) and Theorem \ref{A} (Theorem \ref{bdd of weak Fano}).
\subsection{N\'{e}ron-Severi groups of fibers}

Let $T$ be a Noetherian scheme and $\mathcal{X}$ be a proper $T$-scheme.
Then the \emph{N\'{e}ron-Severi group}
\[
  N^1(\mathcal{X}/T) : = (\Pic(\mathcal{X})/\equiv_T) \otimes_{\Z} \R
\]
is a finitely dimensional $\R$-vector space (cf.~\cite[IV, \S 4]{Kle}).
We denote by $\rho(\mathcal{X}/T) : = \dim_{\R}(N^1(\mathcal{X}/T))$ the \emph{Picard number} of $\mathcal{X}/T$.
When $T$ is the spectrum of a field, we simply write $N^1(\mathcal{X})$ and $\rho(\mathcal{X})$.

\begin{lem}\label{injection on NS}
Let $f : \mathcal{X} \to T$ be a flat projective morphism with normal and connected fibers between normal varieties over an algebraically closed field $k$, $t \in T$ be a point and $\ell$ be an algebraic extension of the residue field $\kappa(t)$.
We further assume that $T$ is smooth and $H^1(\mathcal{X}_t, \sO_{\mathcal{X}_t})=0$.
Then the restriction map
  \[
  \mathrm{res}_{\ell}:  N^1 (\mathcal{X}/T) \to N^1(\mathcal{X}_{\ell})\]
  is injective.
\end{lem}

\begin{proof}
Take a Cartier divisor $D$ on $\mathcal{X}$ such that the restriction $D_{\ell}$ is numerically trivial.
It follows from \cite[Corollary 1.4.38]{Laz} and \cite[Excercise III. 12.6]{Har} that there exists an integer $m>0$ such that $mD_{\ell}$ is principal.
Take a subfield $\ell' \subseteq \ell$ which is finite over $\kappa(t)$ such that the defining equation 
\[
  r \in K(\mathcal{X}_{\ell}) = K(\mathcal{X}_t) \otimes_{\kappa(t)} \ell
\]
of $mD_{\ell}$ lies in $K(\mathcal{X}_{\ell'})$,
It then follows from \cite[Proposition 1.4]{Ful} that $dmD_t$ is principal, where $d$ is the extension degree $[\ell':\kappa(t)]$.
Therefore, $dmD$ is linearly equivalent to a Cartier divisor $E$ such that $E_t=0$.
In particular, we have $f(\Supp(E)) \neq T$.

Let $E=\sum_{i} a_i E_i$ be the irreducible decomposition.
It then follows from assumptions of $f$ that $F_i: = f(E_i)$ is a prime divisor on $T$ and $E_i = f^* F_i$.
Therefore, we conclude that $D$ is $\Q$-linearly equivalent to $f^*(\sum_{i=1}^n \frac{a_i}{dm}F_i )$, as desired.
\end{proof}

\begin{prop}\label{constancy of NS}
Let $f : \mathcal{X} \to T$ be a surjective projective morphism with normal and connected fibers between normal varieties over an algebraically closed field $k$.
We further assume that $H^1(\mathcal{X}_t, \sO_{\mathcal{X}_t})=H^2(\mathcal{X}_t, \sO_{\mathcal{X}_t})=0$ for general $t \in T$.
Then there exists a dominant morphism $\mu: S \to T$ from a smooth variety $S$ such that the restriction map
\[
  \mathrm{res}_s: N^1 (\mathcal{X}_S/S) \to N^1((\mathcal{X}_S)_s)\]
  is isomorphic for every closed point $s \in S$, where $\mathcal{X}_S $ is the base change $\mathcal{X} \times_T S$ of $\mathcal{X}$.
\end{prop}

\begin{proof}
After shrinking $T$, we may assume that $T=\Spec (A)$ is smooth, $f$ is flat and we have $H^1(\mathcal{X}_t, \sO_{\mathcal{X}_t})=  H^2(\mathcal{X}_t, \sO_{\mathcal{X}_t})=0$ for every $t \in T$.
We first note that the set of Picard ranks $\{\rho(\mathcal{X}_t)\}_{t \in T}$ is bounded by \cite[Theorem 7.7]{CJLO}.
We will prove the assertion by induction on 
\[
  r(\mathcal{X}/T) : = \max_{t \in T(k)} \rho(\mathcal{X}_t) - \rho(\mathcal{X}/T).\]
It follows from Lemma \ref{injection on NS} that we have $r(\mathcal{X}/T) \ge 0$ and that if $r(\mathcal{X}/T)=0$, then the assertion holds true by putting $S=T$.

We assume $r(\mathcal{X}/T) >0$.
Take a closed point $t \in T$ such that the injective morphism
\[
  \mathrm{res}_t : N^1(\mathcal{X}/T) \to N^1(\mathcal{X}_t)
  \]
is not surjective.
Then we may find an invertible sheaf $L \in \Pic(\mathcal{X}_t)$ whose class $[L] \in N^1(\mathcal{X}_t)$ is not contained in the image of $\mathrm{res}_t$.
It follows from the proof of \cite[Proposition 6.3]{Vis} that there is an invertible sheaf $\mathcal{L}$ on $\mathcal{X} \times_T \Spec(\widehat{\sO_{T,t}})$ such that the restriction $\mathcal{L}_t$ of $\mathcal{L}$ to $\mathcal{X}_t$ is isomorphic to $L$.
By Lemma \ref{spreading out sheaves}, there exists a sub $A$-algebra $B$ in $\widehat{\sO_{T,t}}$ of finite type over $A$ with a coherent sheaf $\mathcal{F}$ on $\mathcal{X} \times_T \Spec(B)$ such that the pullback of $\mathcal{F}$ to $\mathcal{X} \times_T \Spec(\widehat{\sO_{T,t}})$ is isomorphic to $\mathcal{L}$.

After shrinking $S : = \Spec(B)$, we may assume that $S$ is a smooth variety and $\mathcal{F}$ is invertible.
Let $s \in S$ be the image of the unique maximal point of $\Spec(\widehat{\sO_{T,t}})$.
Then the image of
\[\mathrm{res}_s : N^1(\mathcal{X}_S/S) \to N^1((\mathcal{X}_S)_s) \cong N^1(\mathcal{X}_t)
\]
contains both $[L] \in N^1(\mathcal{X}_t)$ and $\Im(\mathrm{res}_t : N^1(\mathcal{X}/T) \to N^1(\mathcal{X}_t))) \subseteq N^1(\mathcal{X}_t)$.
Therefore, we have
\[\rho(\mathcal{X}_S/S) \ge \dim(\Im(\mathrm{res}_s))> \dim (\Im(\mathrm{res}_t)) =\rho(\mathcal{X}/T),\]
where the last equation follows from Lemma \ref{injection on NS}.
Therefore, we have $r(\mathcal{X}_S/S)<r(\mathcal{X}/T)$. 
The assertion now follows from the induction hypothesis.
\end{proof}

\subsection{$\Q$-factoriality of fibers}

Let $\pi:X \to T$ be a projective morphism between normal varieties over an algebraically closed field $k$.
We define 
\[
  N_1(X/T) : = \bigoplus_{C \subseteq X} \R [C]/\equiv_T,
\] 
where $C$ runs through all curves in $X$ over $T$, that is, $C$ is a curve with $\dim (\pi(C))=0$.
We also denote by $\NE(X/T)$ the closed convex cone generated by 
\[
  \{[C] \in N_1(X/T) \mid C \subseteq X \textup{ is a curve over $T$} \}.
\]
We note that $N_1(X/T) \cong N^1(X/T)^*$ and $\NE(X/T)$ is the dual cone of the nef cone $\mathrm{Nef(X/T)}$ (cf.~\cite[Proposition 1.4.28]{Laz}).
Let $Y$ be a normal variety over $k$ which is projective over $T$.
A morphism $f :X \to Y$ over $T$ is an \emph{extremal ray contraction} if $f_*\sO_X \cong \sO_Y$ and the closed convex cone $F$ generated by
\[ 
  \{[C] \in N_1(X/T) \mid C \subseteq X \textup{ is a curve over $Y$} \}
\]
is an extremal ray (cf.~\cite[Definition 1.15]{KM}) of $\NE(X/T)$.

\begin{lem}\label{criterion for extremal ray contraction}
Let $f : X \to Y$ be a projective morphism with $f_* \sO_X = \sO_Y$ between normal projective varieties over an algebraically closed field $k$.
We further assume that $\dim(X)=3$, $\mathrm{ch}(k)>5$ and that there exists an effective $\Q$-Weil divisor $\Delta$ on $X$ such that $(X, \Delta)$ is klt and $-(K_X+\Delta)$ is $f$-ample.
Then the following conditions are equivalent to each other.
\begin{enumerate}[label=\textup{(\alph*)}]
  \item $f$ is an extremal ray contraction.
  \item $\rho(X/T)-\rho(Y/T)=1$.
  \item $\rho(X/Y)=1$.
\end{enumerate}
\end{lem}

\begin{proof}
We first show that the sequence
\[
  0 \to N^1(Y/T) \xrightarrow{f^*} N^1(X/T) \xrightarrow{\alpha} N^1(X/Y) \to 0.
\]
is exact, where $\alpha$ is the morphism induced by the identity.
It suffices to show that $\Ker(\alpha)$ is contained in $\Im(f^*)$.
Let $L \in \Pic(X)$ be an element whose class $[L] \in N^1(X/T)$ is contained in $\Ker(\alpha)$.
By the base point free theorem (\cite[Theorem 1.2]{BW}), $L$ is semiample over $Y$.
It then follows from the rigidity lemma (cf.~\cite[Lemma 1.15]{Deb}) that the semiample fibration of $L$ over $Y$ is isomorphic to $f: X \to Y$.
Therefore, we have $[L] \in \Im(f^*)$, as desired.
The equivalence of (b) and (c) follows from this exact sequence.

Let $F \subseteq N_1(X/T)$ be the closed convex cone generated by all curves over $Y$.
The implication (a) $\Rightarrow$ (c) follows from the isomorphism $\NE(X/Y) \cong F$ induced by the injection $f_* : N_1(X/Y) \hookrightarrow N_1(X/T)$.

Since the dual sequence
\[
  0 \to N_1(X/Y) \xrightarrow{\beta} N_1(X/T) \xrightarrow{f_*} N_1(Y/T) \to 0
\]
is exact, $N_1(X/Y) \cap \NE(X/T)$ is an extremal subcone of $\NE(X/T)$.
Combining this with the inclusion $F \subseteq N_1(X/Y) \cap \NE(X/T)$, if we assume (c), then $F$ is an extremal ray.
This proves the implication (c) $\Rightarrow$ (a).
\end{proof}

\begin{lem}\label{rationality of image}
  Let $f: X \to Y$ be a birational projective morphism between normal varieties over an algebraically closed field $k$ of characteristic $p>5$ with $\dim X = 3$.
  Assume that there exists an effective $\Q$-Weil divisor $\Delta$ on $X$ such that $(X, \Delta)$ is klt and $-(K_X+\Delta)$ is $f$-ample.
  Then $Y$ has rational singularities.
\end{lem}
  
\begin{proof}
  We may assume that $Y$ is the spectrum of a local ring.
  By \cite[Lemma 2.34 and Proposition 2.14]{BMPS+}, we may find an effective $\Q$-Weil divisor $0 \le A \sim_{\Q} -(K_X+\Delta)$ such that $(X, \Delta+A)$ is klt.
  We set $\Delta_Y:=g_*(\Delta+A)$.
  Then it follows from the equation 
  \[
    f^*(K_Y+\Delta_Y) = K_X+\Delta+A
    \]
  that $(Y, \Delta_Y)$ is also klt.
  Therefore, the assertion follows from Theorem \ref{ABL}.
\end{proof}

\begin{prop}\label{fiber of ray cont}
Let $\mathcal{X},\mathcal{Y}, T$ be normal varieties over algebraically closed field $k$ of characteristic $p>5$ such that $\mathcal{X}$ and $\mathcal{Y}$ are projective over $T$ and $\Phi : \mathcal{X} \to \mathcal{Y}$ be a birational morphism over $T$.
We further assume that the following properties hold:
\begin{enumerate}[label=\textup{(\roman*)}]
  \item The geometric generic fibers $X : = \mathcal{X}_{\overline{\eta}}$ and $Y : = \mathcal{Y}_{\overline{\eta}}$ are $3$-dimensional normal varieties.
  \item There exists an effective $\Q$-Weil divisor $\Delta$ on $X$ such that $(X, \Delta)$ is klt and $-(K_X+\Delta)$ is ample over $Y$.
  \item We have $H^1(X, \sO_X)=H^2(X, \sO_X)=0$.
\end{enumerate}
If the birational morphism $\Phi_{\overline{\eta}} : X \to Y$ induced by $\Phi$ is an extremal ray contraction, then $\rho(\mathcal{X}_t) -\rho(\mathcal{Y}_t)=1$ for a general closed point $t \in T$.
\end{prop}

\begin{proof}
Fix invertible sheaves $L_1, \dots, L_n$ on $X$ whose classes generate $N^1(X)$.
It from Lemma \ref{spreading out sheaves} that after replacing $T$, we may assume that every $L_i$ lifts to an invertible sheaf on $\mathcal{X}$.
Combining this with Lemma \ref{injection on NS}, the restriction morphism
\[
  \mathrm{res}_{\overline{\eta}}^{\mathcal{X}} : N^1(\mathcal{X}/T) \to N^1(X)
\]
is isomorphic.
It then follows from Proposition \ref{constancy of NS} that we have 
\[
  \rho(\mathcal{X}_t) = \rho(X)
\]
for general closed point $t$.

On the other hand, it follows from Theorem \ref{ABL} and Lemma \ref{rationality of image} that $X$ and $Y$ have rational singularities.
Therefore, by the Leray spectral sequence, we have 
\[
  H^1(Y, \sO_Y)=H^2(Y, \sO_Y)=0.
\]
It then follows from the similar argument in the first paragraph that after replacing $T$, we have
\[
  \rho(\mathcal{Y}_t) = \rho(Y)
\]
for general closed point $t$.
The assertion now follows from Lemma \ref{criterion for extremal ray contraction}.
\end{proof}

\begin{prop}\label{Q-factoriality on fibers}
Let $\pi : \mathcal{X} \to T$ be a projective morphism between normal varieties over an algebraically closed field $k$ of characteristic $p>5$.
We further assume that the following properties hold:
\begin{enumerate}[label=\textup{(\roman*)}]
  \item The geometric generic fiber $X : = \mathcal{X}_{\overline{\eta}}$ is a $\Q$-factorial klt normal variety of dimension $3$.
  \item We have $H^1(X, \sO_X)=H^2(X, \sO_X)=0$.
\end{enumerate}
Then a general closed fiber $\mathcal{X}_t$ is also $\Q$-factorial.
\end{prop}

\begin{proof}
\textbf{Step.1}:
Let $\mu: Z \to X$ be a log resolution which is a projective morphism.
Take an exceptional effective $\Q$-divisor $D \ge 0$ on $Z$ such that $(Z, D)$ is klt and $D \ge -K_{Z/X}+\epsilon \Exc(\mu)$ for some $\epsilon>0$.
We fix an (absolute) ample $\Q$-divisor $A$ on $Z$ such that $K_Z+D+A$ is ample.
By Lemma \ref{Bertini}, we may assume that $(Z, D+A)$ is also klt.
Then it follows from \cite[Theorem 1.5]{BW} that we can run $(K_Z+D)$-MMP over $X$ which terminates, in other words, there exists a sequence of normal varieties
\begin{align}\label{MMP on geometric generic}
  \xymatrix{
    Z=Z_0 \ar_-{f_0}[rdd] && Z_1 \ar^-{g_0}[ldd] \ar_-{f_1}[rdd] &&  \ar^-{g_1}[ldd] && \ar_-{f_n}[rdd] && Z_{n+1} \ar^-{g_n}[ldd] \\
  & && && \cdots && & \\
  & W_0 && W_1 &&&&  W_n &
  }
\end{align}
over $X$ such that if we denote $D_i$ the strict transform of $D$ to $Z_i$, then the following properties hold:
\begin{itemize}
  \item $K_{Z_{n+1}}+D_{n+1}$ is nef over $X$.
  \item For every $i$, $f_i$ is an extremal ray contraction such that $-(K_{Z_i}+D_i)$ is $f_i$-ample.
  \item $g_i$ is an isomorphism (resp.~a flip of $f_i$) if $f_i$ is divisorial (resp.~small).
\end{itemize}

Since log minimal models are unique up to isomorphic in codimension one (\cite[Theorem 3.52 (2)]{KM}), the birational morphism $Z_{n+1} \to X$ is small.
Since $X$ is $\Q$-factorial, any divisor on $\mathcal{Z}_{n+1}$ is the pullback of some $\Q$-divisor on $X$, which implies that the morphism $Z_{n+1} \to X$ is isomorphic.

We also note that it follows from Lemma \ref{rationality of image} and \cite[Corollary 3.42]{KM} that $Z_i$ has rational singularities for all $i$.
By the Leray spectral sequence, we have
  \[
    H^1(Z_i, \sO_{Z_i})=H^2(Z_i, \sO_{Z_i})=0
  \]
for all $i$.
Moreover, it follows from \cite[Proposition 3.36, Proposition 3.37]{KM} that $Z_i$ is $\Q$-factorial for every $i$.

\textbf{Step.2}:
By Lemma \ref{spreading out schemes} and Lemma \ref{spreading out morphisms}, after replacing $T$, we can take a sequence
\begin{align}\label{MMP on total space}
  \xymatrix{
    \mathcal{Z} = \mathcal{Z}_0 \ar_-{F_0}[rdd] && \mathcal{Z}_1 \ar^-{G_0}[ldd] \ar_-{F_1}[rdd] &&  \ar^-{G_1}[ldd] && \ar_-{F_n}[rdd] && \mathcal{Z}_{n+1} \ar^-{G_n}[ldd] \\
  & && && \cdots && & \\
  & \mathcal{W}_0 && \mathcal{W}_1 &&&&  \mathcal{W}_n &
  }
\end{align}
of projective $\mathcal{X}$ schemes whose base change to the geometric generic point of $T$ is isomorphic to the sequence \eqref{MMP on geometric generic}.
After replacing $T$ again, we may also assume that the following properties hold:
\begin{itemize}
  \item A general closed fiber $(\mathcal{Z}_i)_t$ and $(\mathcal{W}_i)_t$ is a normal variety.
  \item For every $i$ with $f_i$ small, both $F_i$ and $G_i$ are small birational morphisms.
  \item For every $i$ with $f_i$ divisorial, $F_i$ is a non-small birational morphism and $G_i$ is an isomorphism. 
  \item $\mathcal{Z}_{n+1}$ is isomorphic to $\mathcal{X}$.
\end{itemize}

After replacing $T$ again, we may also assume that the pair $(\mathcal{Z}, \Exc(\mathcal{Z} \to \mathcal{X}))$ is relatively SNC over $T$ and that every irreducible component of $\Exc(\mathcal{Z} \to \mathcal{X})$ has integral fibers over $T$.
Let $\mathcal{D}$ be the effective $\Q$-divisor on $\mathcal{Z}$ whose support is exceptional over $\mathcal{X}$ and whose restriction to $Z=\mathcal{Z}_{\overline{\eta}}$ is $D$.
We denote by $\mathcal{D}_i$ the strict transform of $\mathcal{D}$ to $\mathcal{Z}_i$.
Then the flat pullback $(\mathcal{D}_i)_{\overline{\eta}}$ of $\mathcal{D}_i$ is $D_i$.
Similarly, we have $(K_{\mathcal{Z}_i})_{\overline{\eta}} = K_{Z_i}$.
It follows from the $\Q$-factoriality of $Z_i$ that after replacing $T$ again, we may assume that for every $i$, both $\mathcal{D}_i$ and $K_{\mathcal{Z}_i}$ are $\Q$-Cartier.
In particular, after shrinking $T$, we may assume that $K_{\mathcal{Z}_i} + \mathcal{D}_i$ is $F_i$-anti-ample and $G_i$-ample for every $i$.

\textbf{Step.3}:
After shrinking $T$, we may assume that $T$ is smooth and $\omega_T \cong \sO_T$.
It then follows from \cite[Lemma 0B6U, 0ATX and 0EA0]{Sta} that $\omega_{\mathcal{Z}} \cong \Omega_{\mathcal{Z}/T}^{\wedge 3}$.
Therefore, the restriction $(K_{\mathcal{Z}})_t$ is a canonical divisor of $\mathcal{Z}_t$.
Combining this with Lemma \ref{model and fiber} (3), the restriction $(K_{\mathcal{Z}_i})_t$ is a canonical divisor of $(\mathcal{Z}_i)_t$ for every $i$ and general $t \in T$.
Moreover, by the same lemma, the restriction $(\mathcal{D}_{i+1})_t$ is the strict transform of the restriction $(\mathcal{D}_i)_t$ for every $i$ and general $t$.
It then follows from Lemma \ref{criterion for extremal ray contraction} and Proposition \ref{fiber of ray cont} that the base change of the sequence \eqref{MMP on total space} to a general closed point $t \in T$ is a $(K_{\mathcal{Z}_t} + \mathcal{D}_t)$-minimal model program.
It then follows from \cite[Proposition 3.36, Proposition 3.37]{KM} that $\mathcal{X}_t \cong (\mathcal{Z}_{n+1})_t$ is $\Q$-factorial, as desired.
\end{proof}

\subsection{On boundedness of birationally bounded set}

In this subsection, we prove Theorem \ref{E}.
The strategy is similar to that of \cite{HX}.
However, we need some modifications because of the lack of "invariance of plurigenera".

\begin{defn}[\textup{\cite[Definition 2.6]{HX}}]
A normal projective variety $X$ over an algebraically closed field $k$ is said to be \emph{of Fano type} if $-K_X$ is big Weil divisor and there exists an effective $\Q$-Weil divisor $0 \le \Delta \sim_{\Q} -K_{X}$ on $X$ such that $(X,\Delta)$ is klt.
\end{defn}

The following lemma should be well-known to experts,  but we include a proof here for the sake of completeness.

\begin{lem}\label{of Fano basic}
Let $X$ be a normal projective variety of dimension $3$ over an algebraically closed field $k$.
Then $X$ is of Fano type if and only if there exists an effective $\Q$-Weil divisor $\Delta'$ on $X$ such that $(X,\Delta')$ is klt and $-(K_X+\Delta')$ is ample.
\end{lem}

\begin{proof}
Since there exists a log resolution of $(X, \Delta')$, the "if" part follows from Lemma \ref{Bertini}.
For the "only if" part, we take a small $\Q$-factorization $f: Y \to X$.
Let $A$ be an ample Cartier divisor on $X$.
Since $f^{-1}_*\Delta$ is big, it follows from Kodaira's lemma that there exists an integer $n>0$ and an effective Weil divisor $F \ge 0$ on $Y$ such that 
\[
  F \sim nf^{-1}_*\Delta - f^*A.
\]
By taking $f_*$, we conclude that $f_*F \sim n\Delta - A$.
We then take a rational number $\epsilon>0$ such that $(X, (1-\epsilon)\Delta + (\epsilon f_*F)/n)$ is klt, which completes the proof.
\end{proof}

\begin{thm}[\textup{cf.~\cite[Proposition 2.9]{HX}}]\label{Sbb to Bdd}
Let $\pi : \mathcal{Y} \to T$ be a projective morphism between normal varieties over an algebraically closed field $k$ of characteristic $p>5$ and $B \subseteq \mathcal{Y}$ be a closed subset purely codimension one.
We further assume that the geometric generic fiber $ Y: = \mathcal{Y}_{\overline{\eta}}$ satisfies the following conditions:
\begin{enumerate}[label=\textup{(\roman*)}]
  \item $Y$ is $3$-dimensional and of Fano type.
  \item $H^1(Y, \sO_Y) =H^2(Y, \sO_Y) =0$.
\end{enumerate}

Then there exists an open dense subset $U \subseteq T$ such that any set $\mathcal{D}$ of a projective log pairs is bounded if for every $(X,\Delta) \in \mathcal{D}$, there exists a small birational map 
\[
  f : X \dashrightarrow \mathcal{Y}_t
\]
to the fiber $\mathcal{Y}_t$ at some point $t \in U(k)$ with $\Supp(f_*\Delta) = B_t$.
\end{thm}

\begin{proof}

\textbf{Step.1}:
Let $\Delta_0$ be a big $\Q$-Weil divisor on $Y$ such that $(Y, \Delta_0)$ is klt and $K_Y+\Delta_0 \sim_{\Q} 0$.
Take a small $\Q$-factorization $\mu : \widetilde{Y} \to Y$.
Noting that $(\widetilde{Y}, \mu^{-1}_* \Delta_0)$ is klt and $\mu^{-1}_*\Delta_0 \sim_{\Q} -K_{\widetilde{Y}}$ is big, $\widetilde{Y}$ is also of Fano type.
Since both $Y$ and $\widetilde{Y}$ are rational singularities, we also have
\[
  H^1(\widetilde{Y}, \sO_{\widetilde{Y}})=H^2(\widetilde{Y}, \sO_{\widetilde{Y}})=0.
\]
By Lemma \ref{spreading out schemes} and Lemma \ref{spreading out morphisms}, after replacing $T$, we obtain a small birational morphism $\nu : \widetilde{\mathcal{Y}} \to \mathcal{Y}$ from a normal variety $\widetilde{\mathcal{Y}}$ whose geometric generic fiber is isomorphic to $\mu$.
By replacing $\mathcal{Y}$ by $\widetilde{\mathcal{Y}}$, we may assume that $Y$ is $\Q$-factorial.

By Proposition \ref{Q-factoriality on fibers}, after shrinking $T$, we may assume that every fiber is $\Q$-factorial.
It follows from Proposition \ref{constancy of NS} that we may also assume that the restriction 
\[
  \mathrm{res}_t : N^1(\mathcal{Y}/T) \to N^1(\mathcal{Y}_t)
\]
is isomorphic for every $t \in T(k)$.
By the semicontinuity, after shrinking $T$ again, we have $H^1(\mathcal{Y}_t ,\sO_{\mathcal{Y}_t})=0$ for every $t \in T$.

\textbf{Step.2}:
Take an effective $\Q$-Weil divisor $\overline{\Gamma}$ on $Y$ such that $(Y, \overline{\Gamma})$ is klt and $-(K_{Y} + \overline{\Gamma})$ is ample.
After replacing $T$ again, there is an effective $\Q$-Weil divisor $\Gamma$ on $\mathcal{Y}$ whose flat pullback to the geometric generic fiber is $\overline{\Gamma}$.
Since $(Y, \overline{\Gamma})$ is klt, it follows from \cite[Lemma 6.17]{DW} that $(\mathcal{Y}_{\eta}, \Gamma_{\eta})$ is also klt.
We may also assume that $\mathcal{H} : = -(K_{\mathcal{Y}}+\Gamma)$ is $\Q$-Cartier and $\pi$-ample.

By Lemma \ref{Bertini}, after shrinking $T$, we can take an effective $\Q$-Cartier divisor $\mathcal{A}$ on $\mathcal{Y}$ which is $\Q$-linearly equivalent to $\mathcal{H}/2$ so that $(\mathcal{Y}_{\eta}, \Gamma_{\eta} + \mathcal{A}_{\eta})$ is klt.
It follows from the openness of the ample cone that we may find $\pi$-ample $\Q$-Cartier divisors $\mathcal{H}_1, \dots, \mathcal{H}_m$ such that the convex full of them in $N^1(\mathcal{Y}/T)$ contains an open neighborhood of $\mathcal{H}/2$.
By applying Lemma \ref{Bertini} again, we may assume that $(\mathcal{Y}_{\eta}, \Gamma_{\eta}+ \mathcal{A}_{\eta} + (\mathcal{H}_i)_{\eta})$ is klt for every $i$.

Let $\mathcal{C} \subseteq \WDiv_{\R}(\mathcal{Y})$ be the convex hull of $\{\Gamma+\mathcal{A} + \mathcal{H}_i \mid i =1, \dots, m\}$.
Then the numerical class
\[
  [-K_{\mathcal{Y}}] = [\Gamma + \mathcal{A}] + [\mathcal{H}/2] \in N^1(\mathcal{Y}/T)
\]
is contained in the interior of the image of $\mathcal{C}$ in $N^1(\mathcal{Y}/T)$.
Therefore, for all $\Q$-Cartier divisor $\mathcal{D}$ on $\mathcal{Y}$, there exists a rational number $c>0$ and an element $\Theta \in \mathcal{C}$ such that
\begin{align}\label{K is in interior}
  K_{\mathcal{Y}} + \Theta \equiv_T c \mathcal{D}.
\end{align}

By Theorem \ref{bigness constancy} and Theorem \ref{finiteness of lc model relative}, there is an open dense subset $U \subseteq T$ and finitely many birational contractions 
\[
  \Phi_i: \mathcal{Y}_U \dashrightarrow \mathcal{Z}_i
  \]
($i=1, \dots, \ell < \infty$) such that for every $\Theta \in \mathcal{C}$ and every $t \in U(k)$, if $(K_{\mathcal{Y}}+\Theta)_t$ is big, then $(\Phi_i)_t$ is the ample model of $(K_{\mathcal{Y}}+\Theta)_t$ for some $i$.
After shrinking $U$, we may assume that the support of $((\Phi_i)_*B)_t$ coincides with that of $((\Phi_i)_t)_* B$ for every $i$ and every $t \in U(k)$.

\textbf{Step.3}:
We will show that the set $\mathcal{D}$ as in the statement is bounded by 
\[
  (\coprod_{i} \mathcal{Z}_i, \coprod_i (\Phi_i)_*B )
\] over $\coprod_{i} T$.
Take an element $(X, \Delta) \in \mathcal{D}$.
Then there exists a closed point $t \in U(k)$ and a small birational map $f: X \dashrightarrow \mathcal{Y}_t$ such that $\Supp(f_*\Delta) = B_t$.
It suffices to show that there exists an isomorphism $g : X \xrightarrow{\sim} (\mathcal{Z}_i)_t$ such that $f = (\Phi_i)_t^{-1} \circ g $.

Take an ample divisor $D$ on $X$.
Since $\mathcal{Y}_t$ is $\Q$-factorial, the strict transform $f_*D$ is a big $\Q$-Cartier divisor.
By the surjectivity of 
\[
  \mathrm{res}_t : N^1(\mathcal{Y}/T) \to N^1(\mathcal{Y}_t),
\]
there exists a $\Q$-Cartier divisor $\mathcal{D}$ on $\mathcal{Y}$ with $\mathcal{D}_t \equiv f_*D$.
Combining this with the equation \eqref{K is in interior}, we may take a rational number $c>0$ and a rational element $\Theta \in \mathcal{C}$ such that 
\[
  (K_{\mathcal{Y}}+\Theta)_t \equiv cf_*D.
\]
Noting that we have $H^1(\mathcal{Y}_t, \sO_{\mathcal{Y}_t})=0$, this is equivalent to saying that 
\[
  (K_{\mathcal{Y}}+\Theta)_t \sim_{\Q} c f_*D.
\]
Since $D$ is ample, the small birational map $f^{-1}: \mathcal{Y}_t \dashrightarrow X$ is the ample model of $(K_{\mathcal{Y}}+\Theta)_t$.
The assertion now follows from the uniqueness of ample models (cf.~\cite[Theorem 3.52 (1)]{KM}).
\end{proof}

\begin{thm}[Theorem \ref{E}]\label{Bir bdd to bdd}
Let $0<\delta, \epsilon$ be rational numbers, $\pi : \mathcal{Z} \to T$ be a smooth projective morphism between smooth varieties over an algebraically closed field $k$ of characteristic $p>5$ and $B$ be a reduced divisor on $\mathcal{Z}$ such that $(\mathcal{Z}, B)$ is relatively SNC over $T$.

For an open dense subset $U \subseteq T$, we denote by $\mathcal{D}_U$ the set of all $3$-dimensional projective log pairs $(X,\Delta)$ with the following properties:
\begin{enumerate}[label=\textup{(\roman*)}]
  \item There exist a closed point $t \in U(k)$ and a birational contractions 
  \[
    f: \mathcal{Z}_{t} \dashrightarrow X
  \] 
  such that $B_t$ coincides with the union of $\Supp(f^{-1}_*\Delta)$ and all $f$-exceptional divisors.
  \item $-K_{X}$ is nef and big.
  \item $K_{X}+\Delta \sim_{\Q} 0$, $\epsilon < \mld(X,\Delta)$ and every coefficient of $\Delta$ is larger than $\delta$
  \item $H^1(X, \sO_{X})=H^2(X, \sO_{X})=0$.
\end{enumerate}
We further assume that we have 
\[
  \kappa(\mathcal{Z}/T, K_{\mathcal{Z}}+(1-\epsilon)B) \ge 1.
\]
Then there exists an open dense subset $U \subseteq T$ such that $\mathcal{D}_U$ is bounded.
\end{thm}
  
\begin{proof}
After replacing $T$, we may assume that every component of $B$ has integral fibers over $T$.

\textbf{Step.1}:
In this step, we will show that $K_{\mathcal{Z}}+(1-\epsilon)B$ is big over $T$.
By Proposition \ref{Iitaka dimension of fiber} (ii), after shrinking $T$, we may assume that 
\begin{align}\label{equation of iitaka}
  \kappa(\mathcal{Z}/T, K_{\mathcal{Z}} + (1-\epsilon)B) = \kappa(\mathcal{Z}_t, K_{\mathcal{Z}_t} + (1-\epsilon)B_t)
\end{align}
for every closed point $t \in T$.

We may assume that there exist a projective log pair $(X,\Delta)$, closed point $t \in T(k)$ and a birational contraction $f : \mathcal{Z}_t \dashrightarrow X$ satisfying conditions (i), (ii) and (iii).
(Otherwise, there is nothing to prove.)
Since we have $\epsilon< \mld(X,\Delta)$, there exists a rational number $a>0$ with $\epsilon < \mld(X, (1+a)\Delta)$.
Let $p: W \to \mathcal{Z}_t$ be a resolution of indeterminacy of $f$ and $q : W \to X$ be the induced morphism.
\[\xymatrix{ & W \ar_-{p}[ld] \ar^-{q}[rd]& \\ \mathcal{Z}_t \ar@{.>}^-{f}[rr] && X}\]
It then follows from the negativity lemma that we have
\[
  p^*(K_{\mathcal{Z}_t} + (1+a)f^{-1}_* \Delta + (1-\epsilon)(\Exc^1(f))_t) \ge q^*(K_{X} + (1+a)\Delta),
\]
where $\Exc^1(f)$ denotes the union of all $f$-exceptional divisors.
Therefore, we have
\begin{align*}
\vol(K_{\mathcal{Z}_t}+(1-\epsilon)B_t) & \ge \vol(K_{\mathcal{Z}_t}+(1+a)f^{-1}_*\Delta + (1-\epsilon)(\Exc^1(f))_t) \\
& \ge \vol(K_X+(1+a)\Delta)\\
& >0.
\end{align*}
Combining this with the equality \eqref{equation of iitaka}, we conclude that $K_{\mathcal{Z}}+(1-\epsilon)B$ is big over $T$.

\textbf{Step.2}:
Take an effective $\Q$-divisor
\[
0 \le D \sim_{\Q, T} K_{\mathcal{Z}}+(1-\epsilon) B.
\]
Let $0<c \le \delta/(1-\epsilon)$ be a rational number such that the pair $(\mathcal{Z} , (1-\epsilon)B+A)$ is klt around the generic fiber $\mathcal{Z}_{\eta}$, where we set
\[
A = \frac{c}{1-c} D.
\]
By applying Theorem \ref{bigness constancy} and Theorem \ref{finiteness of minimal model relative}, there exist an open dense subset $U \subseteq T$ and finitely many birational contractions 
\[
  \Phi_i: \mathcal{Z}_U \dashrightarrow \mathcal{W}_i
\]
($i=1, \dots, \ell<\infty $) over $U$ such that for every $\R$-Weil divisor $\Theta'$ with $0 \le \Theta' \le (1-\epsilon) B$ and every $t \in U(k)$ , if $(K_{\mathcal{Z}}+(\Theta'+A))_t$ is pseudo-effective, then $(\Phi_i)_t$ is a $(K_{\mathcal{Z}} + (\Theta'+A))_t$-negative semiample model of $(K_{\mathcal{Z}} + (\Theta'+A))_t$.

We write $\mathcal{D}_U=\{(X_j, \Delta_j)\}_{j \in J}$.
For every $j \in J$, take a closed point $t_j \in U(k)$ and a birational contraction 
\[
  f_j: Z_j : =\mathcal{Z}_{t_j} \dashrightarrow X_j
\]
as in (i).
We will show that the composite birational map
\[
  f_j \circ (\Phi_i)_{t_j}^{-1}: (\mathcal{W}_i)_{t_j} \dashrightarrow \mathcal{Z}_{t_j}  \dashrightarrow X_j
\]
is small for some $i$.
We may replace $X_j$ by its small $\Q$-factorization.

Take a rational number $0<a_j$ such that $\epsilon< \mld(X_j, (1+a_j)\Delta_j)$.
By the similar argument as in Step.1, we can show that $f_j$ is a log minimal model of $(Z_j, (1+a_j)(f_j^{-1})_* \Delta_j + (1-\epsilon)\Exc^1(f_j))$.
Let $\delta B \le \Theta_j \le (1-\epsilon)B$ be the $\Q$-divisor on $\mathcal{Z}$ such that 
\[
  (\Theta_j)_{t_j} =(1+a_j)(f_j^{-1})_* \Delta_j + (1-\epsilon)\Exc^1(f_j).
\]
We also take the $\Q$-divisor $0 \le \Theta'_j \le (1-\epsilon)B$ such that
\begin{align*} 
  \Theta_j = c (1-\epsilon)B + (1-c)\Theta_j'.
\end{align*} 
Noting that we have
\begin{align*} 
K_{\mathcal{Z}}+\Theta_j \sim_{\Q, T} (1-c) (K_{\mathcal{Z}}+A+\Theta_j'),
\end{align*}
the birational contraction $f_j$ is also a log minimal model of $(Z_j, (\Theta_j'+A)_{t_j})$.

On the other hand, since $(K_{\mathcal{Z}} + \Theta'_j +A)_{t_j} = K_{Z_j} + (\Theta'_j +A)_{t_j}$ is pseudo-effective, there is an index $i$ such that $(\Phi_i)_t$ is $(K_{Z_j}+(\Theta_j'+A)_{t_j})$-negative semiample model of $(K_{\mathcal{Z}} + \Theta' +A)_{t_j}$.
It then follows from the proof of \cite[Theorem 3.52 (2)]{KM} that the composite birational map
\[
  f_j \circ (\Phi_i)_{t_j}^{-1}: (\mathcal{W}_i)_{t_j} \dashrightarrow \mathcal{Z}_{t_j}  \dashrightarrow X_j
\]
is small.

\textbf{Step.3}:
We fix an index $i$ and a reduced divisor $0 \le B' \le B$ and we define
\[
  J(i, B') : = \{j \in J \mid f_j \circ (\Phi_i)_{t_j}^{-1} \textup{ is small and } \Exc^1(f_j)=B'_{t_j}\}.
\]
It suffices to show that after shrinking $U$, the subset 
\[
  \mathcal{D}_U(i, B') : =\{(X_j,\Delta_j)\}_{j \in J(i,B')}
\] 
of $\mathcal{D}_U$ is bounded.
We write $\mathcal{W} : =\mathcal{W}_i$ and $\Phi:=\Phi_i$.
We may assume that the subset $\{t_j \in T \mid j \in J(i,B')\}$ is dense in $T$.
It suffices to show that the assumptions in Theorem \ref{Sbb to Bdd} are satisfied for the geometric generic fiber $W : = \mathcal{W}_{\overline{\eta}}$.

Take $j \in J(i,B')$.
Since $g_j : = f_j \circ \Phi_{t_j}^{-1} : W_j : = \mathcal{W}_{t_j} \dashrightarrow X_j$ is small, the pair $(W_j, (g_j)^{-1}_* \Delta_j)$ is also klt.
It then follows from Theorem \ref{ABL} that both $X_j$ and $W_j$ are rational singularities.
This implies that $H^{\ell}(W_j, \sO_{W_j}) = H^{\ell}(X_j, \sO_{X_j}) = 0$ for $\ell=1,2$.
By the semicontinuity, we have 
\[
  H^1(W,\sO_W)=H^2(W,\sO_W)=0.
\]

\textbf{Step.4}:
In this step, we will show that $W$ is of Fano type.

We fix $j \in J(i, B')$.
Let $\Delta_j'$ be the $\Q$-divisor on $Z_j : = \mathcal{Z}_{t_j}$ such that $(Z_j, \Delta_j')$ is crepant to $(X_j, \Delta_j)$.
In other words, if we write $m( K_{X_j} + \Delta_j)=\Div_{X_j}(r)$, then $\Delta_j'$ is the $\Q$-divisor such that $m(K_{Z_j} + \Delta_j')=\Div_{Z_j}(r)$.
Let $P$ be the lift of $\Delta_j'$ to $\mathcal{Z}$, that is, $P$ is the $\Q$-divisor on $\mathcal{Z}$ with $\Supp(P) \subseteq B$ such that $P_{t_j}=\Delta_j'$.
Since we have $m(K_{\mathcal{Z}} + P)_{t_j} \sim 0$, one has 
\[
  m(K_{\mathcal{Z}}+P) \sim_{T} 0,
\] and in particular, we have $K_{\mathcal{Z}_{\overline{\eta}}} + P_{\overline{\eta}} \sim_{\Q} 0$.

Since we have $\Phi_*(B')=0$ and the support of the negative part $P^{-}$ of $P$ is contained in $B'$, the pushforward $\Phi_*P$ is effective.
Therefore, $\Gamma : = (\Phi_{\overline{\eta}})_*P_{\overline{\eta}}  = (\Phi_* P)_{\overline{\eta}}$ is also effective.
Moreover, since we have
\begin{align*}
  0 \sim_{\Q} K_{\mathcal{Z}_{\overline{\eta}}} + P_{\overline{\eta}} = (\Phi_{\overline{\eta}})^* (K_{W} + \Gamma),
\end{align*}
the pair $(W, \Gamma)$ is klt with $K_{W}+\Gamma \sim_{\Q} 0$.

Finally, we will show that $\Gamma$ is big.
Since the support of positive part $P^+$ of $P$ contains $B-B'$ and the support of the negative part $P^{-}$ is contained in $B'$, for sufficiently large integers $L_1, L_2>0$, we have 
\[
  (1-\epsilon)B \le L_1P + L_2 B'.
  \]
Since $K_{\mathcal{Z}}+(1-\epsilon)B$ is big over $T$ by Step.1, so is $K_{\mathcal{Z}} + L_1P + L_2 B' \sim_{\Q, T} (L_1-1)P + L_2 B'.$
Therefore, the restriction 
\[
  (L_1-1)P_{\overline{\eta}} + L_2 B'_{\overline{\eta}}
\]
is big.
By taking $(\Phi_{\overline{\eta}})_*$, we conclude that $(L_1-1)\Gamma$ is big, as desired.
The assertion of this theorem now follows from Theorem \ref{Sbb to Bdd}.
\end{proof}

\subsection{Boundedness}

In this subsection, we prove the main theorem of this paper (Theorem \ref{A}).

\begin{prop}\label{bdd of weak Fano with fixed complement index}
  Fix integers $N, M,V>0$ and an algebraically closed field $k$ of characteristic larger than $5$.
  Suppose that $\mathcal{D} =\{(X_j, \Delta_j)\}_{j \in J}$ is a set of $3$-dimensional projective log pairs over $k$ such that for every $j \in J$, the following conditions are satisfied:
  \begin{enumerate}[label=\textup{(\roman*)}]
    \item $N(K_{X_j}+\Delta_j) \sim 0$ and $(X_j, \Delta_j)$ is klt.
    \item $-K_{X_j}$ is nef and big with $\vol(-K_{X_j})<V$.
    \item $\Phi_{|-MK_{X_j}|}$ is a birational map onto a image.
    \item $H^1(X_j, \sO_{X_j}) = H^2(X_j, \sO_{X_j})=0$.
  \end{enumerate}
  Then $\mathcal{D}$ is bounded.
\end{prop}
  
\begin{proof}
  
\textbf{Step.1}:
After replacing $N$ by $NM$, we may assume that $h^0(X_j, \sO_{X_j}(N\Delta_j)) \ge 2$ for every $j \in J$.
We denote by $Y_j$ the image of $\Phi_{|-MK_{X_j}|}$, by $\phi_j : X_j \dashrightarrow Y_j$ the induced birational map and by $\nu_j:Y_j^n \to Y_j$ the normalization.
We fix an element 
\[
  \Theta_j \in \frac{|N\Delta_j|}{N} : = \{\Theta \ge 0 \mid N \Theta \sim N \Delta_j\}
\]
such that $\Theta_j \neq \Delta_j$.
Since the reduced divisor $\Gamma_j : = (\Delta_j + \Theta_j)_{\mathrm{red}}$ satisfies
\[
  \Gamma_j \le N\Delta_j + N \Theta_j \sim -2NK_{X_j},
\]
we have 
\[
  \vol(K_{X_j}+\Gamma_j-14MK_{X_j}) \le (-1+2N+14M)^3 V.
\]
By Lemma \ref{bdd of normalization} and Theorem \ref{criterion of lbb}, the set $\{(Y_j^n, (\psi_j^{-1})_*\Gamma_j + \Exc^1(\psi_j))\}$ is bounded, where 
\[
  \psi_j : Y_j^n \dashrightarrow X_j
  \]
is the composite of $\phi_j^{-1}$ and $\nu_j$, and $\Exc^1(\psi_j)$ is the union of all $\psi_j$-exceptional divisors.

Let $\pi : \mathcal{Y} \to T$ be a flat projective morphism between algebraic schemes over $k$, $C \subseteq \mathcal{Y}$ be a closed subset and $t_j \in T$ be a closed point such that 
\[
  (\mathcal{Y}_{t_j}, C_{t_j}) \cong (Y_j^n, (\psi_j^{-1})_* \Gamma_j + \Exc^1(\psi_j))
\]
for every $j \in J$.
By Noether induction, we may always shrink $T$ and therefore we may assume that $T$ is irreducible.
By spreading out (Lemma \ref{spreading out schemes} and Lemma \ref{spreading out morphisms}) a log resolution of geometric generic fiber $(\mathcal{Y}_{\overline{\eta}}, C_{\overline{\eta}})$, we may find a log resolution 
\[ \mu: \mathcal{Z} \to \mathcal{Y}\]
of the pair $(\mathcal{Y}, C)$ such that $(\mathcal{Z}, C' : = \mu^{-1}_* C + \Exc(\mu))$ is relatively SNC over $T$.
Moreover, as in the proof of \cite[Theorem 1.6]{HMX2}, after replacing $\mathcal{Z}$ by its blowup, we may assume that 
\[
\psi'_j : = \psi_j \circ \mu_{t_j} : Z_j : =\mathcal{Z}_{t_j}  \dashrightarrow X_j
\]
is a birational contraction for every $j \in J$.
Finally, after replacing $T$ again, we may assume that every irreducible components of $C'$ has geometrically integral fibers over $T$.

We fix a reduced divisor $0 \le B \le C'$ and we set
\[
  J(B) : = \{j \in J \mid \Supp((\psi'_j)^{-1}_* \Delta_j) \cup \Exc^1(\psi_j') = B_{t_j}\}.
\]
It suffices to show that the subset
\[
  \mathcal{D}(B) : = \{(X_j, \Delta_j)\}_{j \in J(B)}
\]
of $\mathcal{D}$ is bounded.
We may assume that $\{t_j \mid j \in J(B)\}$ is dense in $T$.

\textbf{Step.2}:
Noting that we have $\mld(X_j, \Delta_j) \ge 1/N > 1/2N$ for every $j$, by Theorem \ref{Bir bdd to bdd}, it suffices to show the inequality
\[
    \kappa(\mathcal{Z}/T, K_{\mathcal{Z}} + (1-\frac{1}{2N})B) \ge 1.
\]

Fix an index $j \in J(B)$. 
Let $\Delta_j'$ (resp.~$\Theta_j'$) be the $\Q$-divisor on $Z_j$ such that $(Z_j, \Delta_j')$ (resp.~$(Z_j, \Theta_j')$) is crepant to $(X_j, \Delta_j)$ (resp.~$(X_j,\Theta_j)$) and $P$ (resp.~$Q$) be the lift of $\Delta_j'$ (resp.~$\Theta_j'$) to $\mathcal{Z}$ (cf.~Step.4 of the proof of Theorem \ref{Bir bdd to bdd}).
We note that the support of $P$ is contained in $B$.
Since the restriction of the negative part $Q^{-}$ of the $\Q$-divisor $Q$ to $Z_j$ is the negative part of $\Theta_j'$, which is exceptional over $X_j$, we also have 
\[
  \Supp(Q^{-}) \subseteq B.
\]
Therefore, there exists an integer $\ell>0$ such that 
\[
  \ell B - P +Q  \ge \ell B - P-Q^{-} \ge 0.
\]
Combining this with the equivalence $\ell B \sim_{T, \Q} \ell B - P +Q$, it follows from Lemma \ref{positive Iitaka dim} that we have 
\begin{align}\label{inequality kappa 4}
  \kappa(\mathcal{Z}/T, B) \ge 1.
\end{align}

On the other hand, for every $q \in J(B')$, if we define $\Delta_{q}'$ similarly, then we have 
\[
  0 \sim_{\Q} K_{Z_{q}}+ \Delta_{q}' \le K_{Z_{q}} + (1-1/N) B_{t_{q}} = (K_{\mathcal{Z}} + (1-1/N)B)_{t_{q}}.
\]
Since $\{t_q \mid q \in J(B')\}$ is dense in $T$, it follows from Proposition \ref{Iitaka dimension of fiber} (3) that $K_{\mathcal{Z}} + (1-1/N)B$ is pseudo-effective over $T$.
By the non-vanishing theorem (Theorem \ref{NV}), we have $\kappa(\mathcal{Z}/T, K_{\mathcal{Z}} + (1-\frac{1}{N})B) \ge 0.$
Combining this with the inequality \eqref{inequality kappa 4}, we obtain the inequality
\begin{align*}
\kappa(\mathcal{Z}/T, K_{\mathcal{Z}} + (1-\frac{1}{2N})B) & = \kappa(\mathcal{Z}/T, K_{\mathcal{Z}} + (1 - \frac{1}{N})B + \frac{1}{2N}B) \\
& \ge \kappa(\mathcal{Z}/T, \frac{1}{2N} B) \ge 1,
\end{align*}
as desired.
Then the assertion of this proposition follows from Theorem \ref{Bir bdd to bdd}.
\end{proof}

\begin{thm}[Theorem \ref{A}]\label{bdd of weak Fano}
  Fix a DCC subset $I \subseteq [0,1] \cap \Q$, a rational number $\epsilon>0$, integers $r,A>0$ and an uncountable algebraically closed field $k$ of characteristic larger than $5$.
  Suppose that $\mathcal{D} =\{(X_j, \Delta_j)\}_{j \in J}$ is a set of $3$-dimensional projective log pairs over $k$ such that for every $j \in J$, the following conditions are satisfied:
  \begin{enumerate}[label=\textup{(\roman*)}]
    \item $-rK_{X_j}$ is a nef and big Cartier divisor.
    \item $K_{X_j}+\Delta_j \sim_{\Q} 0$, $\epsilon< \mld(X_j,\Delta_j)$ and every coefficient of $\Delta_j$ is contained in $I$.
    \item $H^1(X_j, \sO_{X_j}(irK_{X_j}))=0$ for every integer $i \ge 0$.
    \item $H^2(X_j, \sO_{X_j})=0$ and $h^2(X_j, \sO_{X_j}(irK_{X_j})) < A $ for $i=1,2,3$.
  \end{enumerate}
  Then $\mathcal{D}$ is bounded.
\end{thm}

\begin{proof}
It follows from \cite{Das} that there exists an integer $V>0$ such that $\vol(-K_{X_j})<V$ for all $j \in J$.
By Corollary \ref{bb weak Fano}, there exists an integer $M>0$ such that $\Phi_{|-MK_{X_j}|}$ is birational onto the image $Y_{j}$.
We denote by $\phi_j : X_j \dashrightarrow Y_j$ the induced birational map and by $\nu_j: Y_j^n \to Y_j$ the normalization.
Let $\delta>0$ be the minimum number of non-zero elements of the DCC set $I$.
Since the reduced divisor $(\Delta_j)_{\mathrm{red}}$ satisfies 
\[
  (\Delta_j)_{\mathrm{red}} \le \frac{1}{\delta} \Delta_j,
  \] 
we have
\[
  \vol(K_{X_j}+(\Delta_j)_{\mathrm{red}} -14MK_{X_j}) \le (-1+\frac{1}{\delta}+14M)^3 V.
\]
By the similar argument as in Step.1 of Proposition \ref{bdd of weak Fano with fixed complement index}, we reduce to the case where there exists a log pair $(\mathcal{Z},B)$ which is projective and SNC over a smooth variety $T$ such that for every $j \in J$, we have a closed point $t_j \in T$ and a birational contraction 
\[
  f_j : Z_j : = \mathcal{Z}_{t_j} \dashrightarrow X_j
\]
with $B_{t_j}=\Supp((f_j^{-1})_*\Delta_j) \cup \Exc^1(f_j)$.
Moreover, as in Step.2 of the above proposition, it suffices to show the inequalities
\begin{align}\label{last}
  \kappa(\mathcal{Z}/T, B)  \ge 0 \textup{ and} \\
  \kappa(\mathcal{Z}/T, K_{\mathcal{Z}}+(1-\epsilon) B)  \ge 1.
\end{align}
The latter inequality follows from the similar argument as in the proof of Proposition \ref{bdd of weak Fano with fixed complement index}.
We will prove the former inequality.

For every $j \in J$, we denote by $n_j$ the minimal positive integer such that $n_j(K_{X_j}+\Delta_j)$ is a principal divisor and by $P_j$ the lift of the crepant $\Q$-divisor $\Delta_j'$ on $Z_j$ of $\Delta_j$ (cf.~Step.4 of Theorem \ref{Bir bdd to bdd}).
If $\{n_j \mid j \in J \}$ is bounded, then $\mathcal{D}$ satisfies the assumptions in Proposition \ref{bdd of weak Fano with fixed complement index}, and in particular, $\mathcal{D}$ is bounded.

We consider the case where $\{n_j \mid j \in J\}$ is unbounded.
We fix $j_1 \in J$.
Since we have $n_{j_1}(K_{Z_{j_1}} + \Delta_{j_1}') \sim 0$, after shrinking $T$, we may assume that 
\[
  n_{j_1}(K_{\mathcal{Z}}+P_{j_1}) \sim 0.
\]
Take another index $j_2 \in J(C')$ such that $n_{j_2} > n_{j_1}$.
Since we also have 
\[
  n_{j_2}(K_{\mathcal{Z}}+P_{j_2}) \sim_{T} 0,
\]
one has $P_{j_1} \sim_{\Q, T} P_{j_2}$.
On the other hand, noting that $n_1(K_{Z_{j_2}}+ P_{j_1}|_{t_{j_2}}) \sim 0$, it follows from the minimality of $n_2$ that we have $P_{j_1} \neq P_{j_2}$.
Take an integer $\ell>0$ such that $\ell B + P_{j_1} - P_{j_2}$ is effective.
Then the inequality \eqref{last} follows Lemma \ref{positive Iitaka dim}.
\end{proof}



\end{document}